\documentclass[11pt, leqno]{amsart} 
\usepackage{amssymb,amscd,amsfonts,amsbsy}
\usepackage{latexsym}
\usepackage{exscale}
\usepackage{amsmath,amsthm,amsfonts}
\usepackage{mathrsfs}
\usepackage{xcolor} 
\usepackage[colorlinks=true,linkcolor=blue,citecolor=red,urlcolor=red, backref=page]{hyperref} 
\usepackage{esint} 
\usepackage{stmaryrd}
\usepackage{pifont}

\usepackage[utf8]{inputenc}

\calclayout

\allowdisplaybreaks

\newtheorem{theorem}{Theorem}[section]

\newtheorem{lemma}[theorem]{Lemma}
\newtheorem{corollary}[theorem]{Corollary}

\theoremstyle{definition}
\newtheorem{definition}[theorem]{Definition}
\theoremstyle{remark}
\newtheorem{remark}[theorem]{Remark}

\numberwithin{equation}{section} 

\def\B{\mathcal{B}}
\def\C{\mathbb{C}}
\def\N{\mathbb{N}}
\def\D{\mathcal{D}}
\def\G{\mathcal{G}}
\def\F{\mathcal{F}}
\def\K{\mathbf{K}}
\def\M{\mathcal{M}}
\def\S{\mathcal{S}}
\def\I{\mathcal{I}}
\def\Z{\mathbb{Z}}
\def\R{\mathbb{R}}
\def\Rn{\mathbb{R}^n}
\def\Sn{\mathbb{S}^{n-1}}
\def\b{\mathbf{b}}

\def\p{\mathfrak{p}} 
\def\m{\mathfrak{m}} 

\def\BMO{\operatorname{BMO}}
\def\CMO{\operatorname{CMO}}
\def\loc{\operatorname{loc}}
\def\Lip{\operatorname{Lip}}
\def\supp{\operatorname{supp}}
\def\Re{\operatorname{Re}}
\def\Im{\operatorname{Im}}

\DeclareMathOperator*{\esssup}{ess\,sup}
\DeclareMathOperator*{\essinf}{ess\,inf}

\begin{document}

\title[Extrapolation for multilinear compact operators]
{Extrapolation for multilinear compact operators and applications}

\author{Mingming Cao}
\address{Mingming Cao\\
Instituto de Ciencias Matem\'aticas CSIC-UAM-UC3M-UCM\\
Con\-se\-jo Superior de Investigaciones Cient{\'\i}ficas\\
C/ Nicol\'as Cabrera, 13-15\\
E-28049 Ma\-drid, Spain} \email{mingming.cao@icmat.es}

\author{Andrea Olivo}
\address{Andrea Olivo\\
Departamento de Matem\'atica\\ 
Facultad de Ciencias Exactas y Naturales\\ 
Universidad de  Buenos Aires and IMAS-CONICET\\
Pabell\'on I (C1428EGA), Ciudad de Buenos Aires, Argentina} \email{aolivo@dm.uba.ar}

\author{K\^{o}z\^{o} Yabuta}
\address{K\^{o}z\^{o} Yabuta\\ 
Research Center for Mathematics and Data Science\\
Kwansei Gakuin University\\
Gakuen 2-1, Sanda 669-1337\\
Japan}\email{kyabuta3@kwansei.ac.jp}

\thanks{The first author is supported by Spanish Ministry of Science and Innovation through the Juan de la Cierva-Formaci\'{o}n 2018 (FJC2018-038526-I), through the ``Severo Ochoa Programme for Centres of Excellence in R\&D'' (CEX2019-000904-S), and through PID2019-107914GB-I00, and by the Spanish National Research Council through the ``Ayuda extraordinaria a Centros de Excelencia Severo Ochoa'' (20205CEX001). The second author was supported by Grants UBACyT 20020170100430BA (University of Buenos Aires), PIP 11220150100355 (CONICET) and PICT 2018-03399}

\date{December 29, 2021} 




\begin{abstract}
This paper is devoted to studying the Rubio de Francia extrapolation for multilinear compact operators. It allows one to extrapolate the compactness of $T$ from just one space to the full range of weighted spaces, whenever an $m$-linear operator $T$ is bounded on weighted Lebesgue spaces. This result is indeed established in terms of the multilinear Muckenhoupt weights $A_{\vec{p}, \vec{r}}$, and the limited range of the $L^p$ scale. To show extrapolation theorems above, by means of a new weighted Fr\'{e}chet-Kolmogorov theorem, we present the weighted interpolation for multilinear compact operators. To prove the latter, we also need to build a weighted interpolation theorem in mixed-norm Lebesgue spaces. As applications, we obtain the weighted compactness of commutators of many multilinear operators, including multilinear $\omega$-Calder\'{o}n-Zygmund operators, multilinear Fourier multipliers, bilinear rough singular integrals and bilinear Bochner-Riesz means. Beyond that, we establish the weighted compactness of higher order Calder\'{o}n commutators, and commutators of Riesz transforms related to Schr\"{o}dinger operators. 
\end{abstract}

\maketitle
\tableofcontents

\section{Introduction}\label{sec:intro}
The classical Rubio de Francia's extrapolation theorem \cite{RdF} states that if an operator $T$ satisfies 
\begin{equation}\label{eq:intro-RdF-1}
\begin{array}{c}
\|Tf\|_{L^{p_0}(w_0)} \leq C \|f\|_{L^{p_0}(w_0)}
\\[4pt]
\text{ for some } p_0 \in [1,\infty) \text{ and every } w_0 \in A_{p_0}, 
\end{array}
\end{equation}
then
\begin{equation}\label{eq:intro-RdF-2}
\begin{array}{c}
\|Tf\|_{L^p(w)} \leq C \|f\|_{L^p(w)}
\\[4pt]
\text{ for every } p \in (1,\infty) \text{ and every } w \in A_p. 
\end{array}
\end{equation}
Over the years, this result, along with its different versions, has become a fundamental piece to deal with many problems in harmonic analysis. For instance, one can obtain general $L^p$ estimates from an appropriate case $p=p_0$ and vector-valued weighted inequalities from the scalar-valued ones. The extrapolation theory on weighted Lebesgue spaces is systematically investigated in \cite{CMP}, which has been extended to the general function spaces in \cite{CMM} for the one-weight extrapolation, and in \cite{CO} for the two-weight case.  

Beyond the linear case, Grafakos and Martell \cite{GM} first established the Rubio de Francia extrapolation in the multivariable setting. Indeed, it was shown that if $T$ is bounded from $L^{p_1}(w_1) \times \cdots \times L^{p_m}(w_m)$ to $L^p(w_1^{p/p_1}\cdots w_m^{p/p_m})$ for some fixed exponents $\frac{1}{p}=\frac{1}{p_1}+\cdots+\frac{1}{p_m}$ with $1<p_1,\ldots,p_m <\infty$, and for all $(w_1,\ldots,w_m) \in A_{p_1} \times \cdots \times A_{p_m}$, then the same holds for all possible values of $p_j$. This result was enhanced by Cruz-Uribe and Martell \cite{CM} to the case  $p_j \in (\p_j^-, \p_j^+)$ and $w_j \in A_{p_j/\p_j^-} \cap RH_{(\p_j^+/p_j)'}$, where $1\le \p_j^-<\p_j^+\le \infty$, $j=1,\ldots,m$. Unfortunately, these two conclusions are given in each variable separately with its own Muckenhoupt class of weights and do not quite use the multivariable nature of the problem. In this direction, Li, Martell and Ombrosi \cite{LMO} introduced some new multilinear Muckenhoupt classes $A_{\vec{p}, \vec{r}}$ (cf. Definition \ref{def:Apr}), which is a generalization of the classes $A_{\vec{p}}$ in \cite{LOPTT} and contains some multivariable structure. It is worth mentioning that these classes $A_{\vec{p}, \vec{r}}$ appeared earlier in \cite{CDO} to obtain weighted norm inequalities for the bilinear Hilbert transform, and a particular case was studied in \cite{CTW} for the multilinear Calder\'{o}n-Zygmund operators with a H\"{o}rmander condition. As well as the $A_p$ classes characterize the $L^p$ boundedness of the Hardy-Littlewood maximal operator, the $A_{\vec{p}}$ classes characterize the boundedness of the multilinear Hardy-Littlewood maximal function $\M$ (cf. \eqref{eq:def-M}) from $L^{p_1}(w_1)\times \cdots \times L^{p_m}(w_m)$ to $L^p(w)$. The classes $A_{\vec{p}}$ are also the natural ones for multilinear Calder\'on-Zygmund operators, and for bilinear rough singular integrals with $\Omega \in L^{\infty}(\mathbb{S}^{2n-1})$, while the classes $A_{\vec{p}, \vec{r}}$ are related to operators with restricted ranges of boundedness such as multilinear Fourier multipliers, bilinear Hilbert transforms, and bilinear rough singular integrals with $\Omega \in L^q(\mathbb{S}^{2n-1})$ and $1<q<\infty$ (see Section \ref{sec:App}). Now we state the multilinear Rubio de Francias's extrapolation theorem from \cite{LMO, LMMOV} (see \cite{Nie} for a quantitative estimate) as follows. 
 
 \smallskip 
\noindent{\bf Theorem A.} 
Let $\F$ be a collection of $(m+1)$-tuples of non-negative functions and let $\vec{r}=(r_1, \ldots, r_{m+1})$ with $1 \le r_1, \ldots, r_{m+1}<\infty$. Assume that there exists $\vec{q}=(q_1, \ldots, q_m)$ with $1 \le q_1, \ldots, q_m <\infty$ and $\vec{r} \preceq \vec{q}$ such that for all $\vec{u}=(u_1, \ldots, u_m) \in A_{\vec{q}, \vec{r}}$, 
\begin{align}
\|f\|_{L^q(u^q)} \le C \prod_{i=1}^m \|f_i\|_{L^{q_i}(u_i^{q_i})}, \quad (f, f_1, \dots, f_m) \in \F, 
\end{align}
where $\frac1q=\frac{1}{q_1}+\cdots+\frac{1}{q_m}$ and $u=\prod_{i=1}^m u_i$. Then, for all $\vec{p}=(p_1, \dots, p_m)$ with $1 < p_1, \ldots, p_m <\infty$ and $\vec{r} \prec \vec{p}$, and for all $\vec{w}=(w_1, \ldots, w_m) \in A_{\vec{p}, \vec{r}}$, we have 
\begin{align}
\|f\|_{L^p(w^p)} \le C \prod_{i=1}^m \|f_i\|_{L^{p_i}(w_i^{p_i})}, \quad (f, f_1, \dots, f_m) \in \F, 
\end{align}
where $\frac1p=\frac{1}{p_1}+\cdots+\frac{1}{p_m}$ and $w=\prod_{i=1}^m w_i$.  

On the other hand, by means of extrapolation it is possible to improve the boundedness of an operator to its compactness. In this direction, Hyt\"onen and Lappas \cite{H} first established a ``compact version'' of Rubio de Francia's extrapolation theorem. More precisely, if $T$ is a linear operator such that \eqref{eq:intro-RdF-1} holds and $T$ is compact on $L^{p_0}(w_1)$ for some $w_1 \in A_{p_0}$, then $T$ is compact on $L^p(w)$ for all $p \in (1,\infty)$ and all $w \in A_p$. This conclusion improves \eqref{eq:intro-RdF-2}. Soon after, the same authors \cite{HL} generalized the preceding compact extrapolation to the off-diagonal and the limited range cases, which refine the results in \cite[Theorem~1]{HMS} and \cite[Theorem~4.9]{AM}, respectively. 

Motivated by the work above, the purpose of this paper is to study the Rubio de Francia's extrapolation for multilinear compact operators. To set the stage, let us give the definition of compactness of $m$-linear operators. Given normed spaces $X_1,\ldots,X_m$ and a quasi-normed space $Y$, an $m$-linear operator $T: X_1 \times \cdots \times X_m \to Y$ is said to be  compact if the set $\{T(x_1,\ldots,x_m): \|x_i\| \le 1, i=1,\ldots,m\}$ is relatively compact (or precompact) in $Y$. Writing $B_i$ for the closed unit ball in $X_i$, $i=1,\ldots,m$, the definition of compactness specifically requires that for every $\{(x_1^k,\ldots,x_m^k)\}_{k \ge 1} \subset B_1 \times \cdots \times B_m$, the sequence $\{T(x_1^k, \ldots, x_m^k)\}_{k \ge 1}$ has a convergent subsequence in $Y$. More definitions and notation are given in Section \ref{sec:weight}. 
 
We formulate the extrapolation theorem for multilinear compact operators as follows. 

\begin{theorem}\label{thm:Ap}
Let $T$ be an $m$-linear operator and let $\vec{r}=(r_1, \ldots, r_{m+1})$ with $1 \le r_1, \ldots, r_{m+1}<\infty$. Assume that there exist $q_* \in (0, \infty)$ and $\vec{q}=(q_1, \ldots, q_m)$ with $1 \le q_1, \ldots, q_m <\infty$ and $\vec{r} \preceq \vec{q}$ such that for all $\vec{u}=(u_1, \ldots, u_m) \in A_{\vec{q}, \vec{r}}$, 
\begin{align}\label{eq:Ap-1}
T \text{ is bounded from $L^{q_1}(u_1^{q_1}) \times \cdots \times L^{q_m}(u_m^{q_m})$ to $L^{q_*}(u^{q_*})$}, 
\end{align}
where $\frac{1}{q_*} \le \frac1q=\frac{1}{q_1}+\cdots+\frac{1}{q_m} \le \frac{1}{q_*}+\frac{1}{r'_{m+1}}$ and $u=\prod_{i=1}^m u_i$. Assume in addition that 
\begin{align}\label{eq:Ap-2}
T \text{ is compact from $L^{q_1}(v_1^{q_1}) \times \cdots \times L^{q_m}(v_m^{q_m})$ to $L^{q_*}(v^{q_*})$}  
\end{align} 
for some $\vec{v}=(v_1, \ldots, v_m) \in A_{\vec{q}, \vec{r}}$, where $v=\prod_{i=1}^m v_i$. Then 
\begin{equation}
T \text{ is compact from $L^{p_1}(w_1^{p_1}) \times \cdots \times L^{p_m}(w_m^{p_m})$ to $L^{p_*}(w^{p_*})$}
\end{equation} 
for all $p_* \in (0, \infty)$, for all $\vec{p}=(p_1, \dots, p_m)$ with $1 < p_1, \ldots, p_m <\infty$ so that  $\vec{r} \prec \vec{p}$ and $\frac1p-\frac{1}{p_*}=\frac1q-\frac{1}{q_*}$, and for all $\vec{w}=(w_1, \ldots, w_m) \in A_{\vec{p}, \vec{r}}$, where $\frac1p=\frac{1}{p_1}+\cdots+\frac{1}{p_m}$ and $w=\prod_{i=1}^m w_i$.  
\end{theorem}

We also establish the limited range extrapolation in the multilinear case. 
\begin{theorem}\label{thm:lim}
Let $T$ be an $m$-linear operator and let $1 \leq \p_i^{-} < \p_i^{+} \leq \infty$, $i=1,\ldots,m$. Assume that for each $i=1,\ldots,m$, there exists $q_i \in [\p_i^{-}, \p_i^{+}]$ such that for all $u_i^{q_i} \in A_{\frac{q_i}{\p_i^{-}}} \cap RH_{\big(\frac{\p_i^{+}}{q_i}\big)'}$, 
\begin{align}\label{eq:lim-1}
T \text{ is bounded from $L^{q_1}(u_1^{q_1}) \times \cdots \times L^{q_m}(u_m^{q_m})$ to $L^q(u^q)$}, 
\end{align}  
where $\frac1q=\frac{1}{q_1}+\cdots+\frac{1}{q_m}$ and $u=\prod_{i=1}^m u_i$. Assume in addition that 
\begin{align}\label{eq:lim-2}
T \text{ is compact from $L^{q_1}(v_1^{q_1}) \times \cdots \times L^{q_m}(v_m^{q_m})$ to $L^q(v^q)$} 
\end{align} 
for some $v_i^{q_i} \in A_{\frac{q_i}{\p_i^{-}}} \cap RH_{\big(\frac{\p_i^{+}}{q_i}\big)'}$, $i=1,\ldots,m$, where $v=\prod_{i=1}^m v_i$. Then 
\begin{align}
T\text{ is compact from $L^{p_1}(w_1^{p_1}) \times \cdots \times L^{p_m}(w_m^{p_m})$ to $L^p(w^p)$}
\end{align} 
for all exponents $p_i \in (\p_i^{-}, \p_i^{+})$ and for all weights $w_i^{p_i} \in A_{\frac{p_i}{\p_i^{-}}} \cap RH_{\big(\frac{\p_i^{+}}{p_i}\big)'}$, $i=1,\ldots,m$, where $\frac1p=\frac{1}{p_1}+\cdots+\frac{1}{p_m}$ and $w=\prod_{i=1}^m w_i$.
\end{theorem}

Theorems \ref{thm:Ap} and \ref{thm:lim} are vastly generalized. From Section \ref{sec:weight}, one can see that the weights class $A_{\vec{p}, \vec{r}}$ covers both $A_{\vec{p}}$ and $A_{\vec{p}, q}$, where they are formulated to investigate the  multilinear Calder\'on-Zygmund operators \cite{LOPTT} and fractional integrals \cite{M}, respectively. Additionally, picking $\p_i^-=1$ and $\p_i^+=\infty$, it follows that the class $A_{p_i/\p_i^-} \cap RH_{(\p_i^+/p_i)'}$ becomes $A_{p_i}$, $i=1,\ldots,m$. Thus, we extend the extrapolation theorems in \cite{CM, GM, LMO} to the corresponding compact version. On the other hand, in the $m$-linear setting, we successfully obtain the full range for $p$ and $\vec{p}=(p_1, \ldots, p_m)$. Inspired by \cite{H, HL}, to show Theorem \ref{thm:Ap}, we consider the target space $L^{p_i}(w^{p_i})$ as an interpolation space of another two different weighted spaces $L^{s_i}(u_i^{s_i})$ and $L^{q_i}(v_i^{q_i})$, meanwhile, there holds that $\vec{u}=(u_1, \ldots, u_m) \in A_{\vec{s}, \vec{r}}$ and $\vec{v}=(v_1, \ldots, v_m) \in A_{\vec{q}, \vec{r}}$, where $\vec{v}$ is the same as in the assumption \eqref{eq:Ap-2}. Then, the desired compactness from $L^{p_1}(w_1^{p_1}) \times \cdots \times L^{p_m}(w_m^{p_m})$ to $L^p(w^p)$ will follow from the interpolation for the $m$-linear compact operators on weighted spaces. Unfortunately, there are not good enough results to establish the $m$-linear compact interpolation since $L^p(w^p)$ is not a Banach space whenever $p \in (0, 1)$. That is the main reason why the approach in \cite{H, HL} can not be used to get the full range in the bilinear or the general multilinear case although a bilinear compact interpolation was given in \cite{CFM} for Banach couples. The proof that we present is a bit involved. In order to obtain an $m$-linear compact interpolation on weighted Lebesgue spaces (cf. Theorem \ref{thm:WMIP-4}), we first build a criterion for compactness in $L^p(w)$ with $p \in (0, \infty)$ by means of a new weighted Fr\'{e}chet-Kolmogorov theorem (cf. Theorem \ref{thm:FK-3}). To verify the criterion above, we also need to establish weighted interpolation theorems in both Lebesgue spaces (cf. Theorem \ref{thm:WMIP-1}) and mixed-norm Lebesgue spaces (cf. Theorem \ref{thm:WMIP-3}). The strategy of the proof of Theorem \ref{thm:lim} is similar. 

As a consequence of Theorems \ref{thm:Ap} and \ref{thm:lim}, we obtain compact extrapolation results for multilinear commutators, which allow us to present several applications for many singular integral operators. In the linear case, Uchiyama \cite{Uch} showed that the commutators of Calder\'on-Zygmund operators and pointwise multiplication with a symbol belonging to $\CMO$ are compact on $L^p(\Rn)$ with $1<p<\infty$. This result was extended to the bilinear setting in \cite{BT} and \cite{BDMT1}. Even more, B\'enyi et al \cite{BDMT2} proved the weighted compactness from $L^{p_1}(w_1) \times L^{p_2}(w_2)$ to $L^p(w)$ for $\frac1p=\frac{1}{p_1} + \frac{1}{p_2}$ with $1<p,p_1,p_2<\infty$ and $(w_1, w_2) \in A_p \times A_p$, where $w=w_1^{p/p_1} w_2^{p/p_2}$. Certainly, this is an incomplete result since the restriction on weights and exponents are not natural. We will see that in Section \ref{sec:App} our extrapolation (see Theorem \ref{thm:Tb} below) will deal with this problem. 

In order to present extrapolation theorems for compact commutators, let us introduce relevant notation and some definitions. We say that a locally integrable function $b \in \BMO$ if 
\begin{equation*}
\|b\|_{\BMO} :=\sup_{Q} \fint_{Q} |b(x)-b_Q| \, dx < \infty. 
\end{equation*}
where the supremum is taken over the collection of all cubes $Q \subset \Rn$ and $b_Q :=\fint_Q b\, dx$. Let $\CMO$ denote the closure of $\mathscr{C}_c^{\infty}(\Rn)$ in $\BMO$. Additionally, the space $\CMO$ is endowed with the norm of $\BMO$. Here $\mathscr{C}_c^{\infty}(\Rn)$ is the collection of $\mathscr{C}^{\infty}(\Rn)$ functions with compact supports.

Let T denote an $m$-linear operator from $X_1  \times \cdots \times X_m$ into $Y$, where $X_1,\ldots,X_m$ are some normed spaces and $Y$ is a quasi-normed space. For $(f_1, \ldots ,f_m) \in X_1 \times \cdots \times X_m$ and for a measurable vector $\b=(b_1, \ldots, b_m)$, and $1 \leq j \leq m$, we define, whenever it makes sense, the first order commutators 
\begin{align*}
[T, \b]_{e_j} (f_1, \ldots, f_m)=b_j T(f_1,\ldots, f_j,\ldots, f_m) - T(f_1, \ldots, b_j f_j, \ldots, f_m);
\end{align*}
we denoted by $e_j$ the basis element taking the value $1$ at component $j$ and $0$ in every other component, therefore expressing the fact that the commutator acts as a linear one in the $j$-th variable and leaving the rest of the entries of 
$(f_1, \ldots, f_m)$ untouched. Then, if $k \in \N_+$, we define 
\begin{align*}
[T, \b]_{ke_j}=[\cdots[[T, \b]_{e_j}, \b]_{e_j} \cdots, \b]_{e_j}, 
\end{align*}
where the commutator is performed $k$ times. Finally, if $\alpha=(\alpha_1, \ldots, \alpha_m) \in \N^m$ is a multi-index, we define 
\begin{align*}
[T, \b]_{\alpha}=[\cdots[[T, \b]_{\alpha_1 e_1}, \b]_{\alpha_2 e_2} \cdots, \b]_{\alpha_m e_m}.
\end{align*}

\begin{theorem}\label{thm:Tb}
Let $T$ be an $m$-linear operator and let $\vec{r}=(r_1, \ldots, r_{m+1})$ with $1 \le r_1, \ldots, r_{m+1}<\infty$. Let $\alpha \in \N^m$ be a multi-index and $\b=(b_1,\ldots, b_m) \in \BMO^m$. Assume that there exist $q_* \in (0, \infty)$ and $\vec{q}=(q_1, \ldots, q_m)$ with $1 \le q_1, \ldots, q_m <\infty$ and $\vec{r} \preceq \vec{q}$ such that for all $\vec{u}=(u_1, \ldots, u_m) \in A_{\vec{q}, \vec{r}}$, 
\begin{align}\label{eq:bT-1}
T \text{ is bounded from $L^{q_1}(u_1^{q_1}) \times \cdots \times L^{q_m}(u_m^{q_m})$ to $L^{q_*}(u^{q_*})$}, 
\end{align}
where $\frac{1}{q_*} \le \frac1q=\frac{1}{q_1}+\cdots+\frac{1}{q_m} \le \frac{1}{q_*}+\frac{1}{r'_{m+1}}$ and $u=\prod_{i=1}^m u_i$. Assume in addition that 
\begin{align}\label{eq:bT-2}
[T, \b]_{\alpha} \text{ is compact from $L^{q_1}(\Rn) \times \cdots \times L^{q_m}(\Rn)$ to $L^{q_*}(\Rn)$}. 
\end{align} 
Then 
\begin{align}
[T, \b]_{\alpha} \text{ is compact from $L^{p_1}(w_1^{p_1}) \times \cdots \times L^{p_m}(w_m^{p_m})$ to $L^{p_*}(w^{p_*})$} 
\end{align}
for all $p_* \in (0, \infty)$, for all $\vec{p}=(p_1, \dots, p_m)$ with $1 < p_1, \ldots, p_m <\infty$ so that $\vec{r} \prec \vec{p}$ and $\frac1p-\frac{1}{p_*}=\frac1q-\frac{1}{q_*}$, and for all $\vec{w}=(w_1, \ldots, w_m) \in A_{\vec{p}, \vec{r}}$, where $\frac1p=\frac{1}{p_1}+\cdots+\frac{1}{p_m}$ and $w=\prod_{i=1}^m w_i$.  
\end{theorem}

\begin{theorem}\label{thm:limTb}
Let $T$ be an $m$-linear operator and let $1 \leq \p_i^{-} < \p_i^{+} \leq \infty$, $i=1,\ldots,m$. Let $\alpha \in \N^m$ be a multi-index and $\b=(b_1,\ldots, b_m) \in \BMO^m$. Assume that for each $i=1,\ldots,m$, there exists $q_i \in [\p_i^{-}, \p_i^{+}]$ such that for all $u_i^{q_i} \in A_{\frac{q_i}{\p_i^{-}}} \cap RH_{\big(\frac{\p_i^{+}}{q_i}\big)'}$, 
\begin{align}\label{eq:limTb-1}
T \text{ is bounded from $L^{q_1}(u_1^{q_1}) \times \cdots \times L^{q_m}(u_m^{q_m})$ to $L^q(u^q)$}, 
\end{align}  
where $\frac1q=\frac{1}{q_1}+\cdots+\frac{1}{q_m}$ and $u=\prod_{i=1}^m u_i$. Assume in addition that 
\begin{align}\label{eq:limTb-2}
[T,\b]_{\alpha} \text{ is compact from $L^{q_1}(\Rn) \times \cdots \times L^{q_m}(\Rn)$ to $L^q(\Rn)$}. 
\end{align} 
Then 
\begin{align}
[T, \b]_{\alpha} \text{ is compact from $L^{p_1}(w_1^{p_1}) \times \cdots \times L^{p_m}(w_m^{p_i})$ to $L^p(w^p)$}
\end{align} 
for all exponents $p_i \in (\p_i^{-}, \p_i^{+})$ and for all weights $w_i^{p_i} \in A_{\frac{p_i}{\p_i^{-}}} \cap RH_{\big(\frac{\p_i^{+}}{p_i}\big)'}$, $i=1,\ldots,m$, where $\frac1p=\frac{1}{p_1}+\cdots+\frac{1}{p_m}$ and $w=\prod_{i=1}^m w_i$.
\end{theorem}

The rest of the paper is organized as follows. In Section \ref{sec:pre}, we give some definitions and properties about multilinear Muckenhoupt weights, and the weighted Fr\'{e}chet-Kolmogorov theorems to characterize the relative compactness of subsets in $L^p(w)$. Section \ref{sec:inter} is devoted to establishing the weighted interpolation theorems for multilinear compact operators, which will be the key point to demonstrate the compact extrapolation results aforementioned. In Section \ref{sec:com},  we present the proofs of our main theorems about extrapolation for compact operators. To conclude, in Section \ref{sec:App}, we include many applications of Theorems \ref{thm:Ap}--\ref{thm:limTb}.

\medskip
\noindent\textbf{Acknowledgements.} 
The authors would like to thank J.M. Martell for his helpful discussions about his work on multilinear extrapolation theorems, which enables us to improve Theorem \ref{thm:Ap} to the off-diagonal case. The authors were indebted to anonymous referees for his/her constructive suggestions.

\section{Preliminaries}\label{sec:pre} 
A measurable function $w$ on $\Rn$ is called a weight if $0<w(x)<\infty$ for a.e.~$x \in \Rn$. For $1<p<\infty$, we define the Muckenhoupt class $A_p$ as the collection of all weights $w$ on $\Rn$ satisfying 
\begin{equation*}
[w]_{A_p}:=\sup_{Q} \bigg(\fint_{Q} w\, dx \bigg) \bigg(\fint_{Q}w^{1-p'}\, dx \bigg)^{p-1}<\infty,
\end{equation*} 
where the supremum is taken over all cubes $Q \subset \Rn$. As for the case $p=1$, we say that $w\in A_1$ if  
\begin{equation*}
[w]_{A_1} :=\sup_{Q} \bigg(\fint_Q w\, dx\bigg) \esssup_Q w^{-1}<\infty.
\end{equation*}
Then, we define $A_{\infty} :=\bigcup_{p\geq 1}A_p$ and $[w]_{A_{\infty}}=\inf_{p>1} [w]_{A_p}$.

Given $1\le p\le q<\infty$, we say that $w \in A_{p,q}$ if it satisfies 
\begin{align*}
[w]_{A_{p,q}} := \sup_{Q} \bigg(\fint_{Q} w^q \, dx\bigg)^{\frac1q} \bigg(\fint_{Q} w^{-p'} dx\bigg)^{\frac{1}{p'}}<\infty.  
\end{align*}
Observe that 
\begin{align*}
w \in A_{p,q} &\iff w^q \in A_{1+\frac{q}{p'}} \iff w^{-p'} \in A_{1+\frac{p'}{q}}
\\
&\iff w^p \in A_p \quad\text{and}\quad w^q \in A_q.  
\end{align*}

For $s\in(1,\infty]$, we define the reverse H\"{o}lder class $RH_s$ as the collection of all weights $w$ such that 
\begin{equation*}
[w]_{RH_s} := \sup_{Q} \left(\fint_Q w^s\,dx\right)^{\frac1s} \left(\fint_Q w\,dx\right)^{-1}<\infty. 
\end{equation*}
When $s=\infty$, $(\fint_Q w^s\,dx)^{1/s}$ is understood as $(\esssup_{Q}w)$. Define $RH_1 := \bigcup\limits_{1<s \le \infty} RH_s$. Then we see that $RH_1=A_{\infty}$. It was proved in \cite{JN} that for all $p\in[1,\infty)$ and $s\in(1,\infty)$,
\begin{align}\label{eq:JN}
w\in A_p \cap RH_s \quad \Longleftrightarrow\quad w^s\in A_{\tau}, \quad \tau=s(p-1)+1.
\end{align}

Let us recall the sharp reverse H\"{o}lder's inequality from \cite{CGPSZ, HP12, LOP}. 
\begin{lemma}\label{lem:RH} 
For every $w \in A_p$ with $1 \le p \le \infty$,  
\begin{align}
\left(\fint_{Q} w^{r_w} dx\right)^{\frac{1}{r_w}} \le 2 \fint_Q w \, dx, 
\end{align}
for every cube $Q$, where 
\begin{equation*}
r_w=
\begin{cases}
1+\frac{1}{2^{n+1}[w]_{A_1}}, & p=1, \\
1+\frac{1}{2^{n+1+2p}[w]_{A_p}}, &p \in (1, \infty), \\
1+\frac{1}{2^{n+11}[w]_{A_{\infty}}}, &p=\infty. 
\end{cases}
\end{equation*}
\end{lemma} 

\subsection{Multilinear Muckenhoupt weights}\label{sec:weight}
The multilinear maximal operator is defined by 
\begin{equation}\label{eq:def-M}
\M(\vec{f})(x):= \sup_{Q \ni x} \prod_{i=1}^m \fint_Q |f_i(y_i)| dy_i, 
\end{equation}
where the supremum is taken over all cubes $Q$ containing $x$. 

We are going to present the definition of the multilinear Muckenhoupt classes $A_{\vec{p}, \vec{r}}$ introduced in \cite{LMO}.  Given $\vec{p}=(p_1, \ldots, p_m)$ with $1 \le p_1, \ldots, p_m \le \infty$ and $\vec{r}=(r_1, \ldots, r_{m+1})$ with $1 \le r_1, \ldots, r_{m+1} < \infty$, we say that $\vec{r} \preceq \vec{p}$ whenever
\begin{align*}
r_i  \le p_i,\, i=1, \ldots, m, \text{ and } r'_{m+1} \ge p, \text{ where } \frac1p:=\frac{1}{p_1}+\dots+\frac{1}{p_m}.
\end{align*}
Analogously, we say that $\vec{r} \prec \vec{p}$ if $\vec{r} \preceq \vec{p}$ and moreover $r_i<p_i$ for each $i=1, \ldots, m$, and $r'_{m+1}>p$.  

\begin{definition}\label{def:Apr}
Let $\vec{p}=(p_1,\ldots,p_m)$ with $1\leq p_1, \ldots, p_m<\infty$ and let $\vec{r}=(r_1, \ldots, r_{m+1})$ with $1 \le r_1, \ldots, r_{m+1} < \infty$ such that $\vec{r} \preceq \vec{p}$. Suppose that $\vec{w}=(w_1,\ldots,w_m)$ and each $w_i$ is a weight on $\Rn$. We say that $\vec{w} \in A_{\vec{p}, \vec{r}}$ if 
\begin{align}\label{eq:w-Apr}
[\vec{w}]_{A_{\vec{p}, \vec{r}}} 
:=\sup_Q \bigg(\fint_Q w^{\frac{r'_{m+1}p}{r'_{m+1}-p}}\, dx\bigg)^{\frac{1}{p}-\frac{1}{r'_{m+1}}} 
\prod_{i=1}^m \bigg(\fint_Q w_i^{\frac{r_i p_i}{r_i-p_i}}dx \bigg)^{\frac{1}{r_i}-\frac{1}{p_i}}< \infty,
\end{align} 
where $w=\prod_{i=1}^m w_i$ and the supremum is taken over all cubes $Q \subset \Rn$. When $p=r'_{m+1}$, the term corresponding to $w$ needs to be replaced by $\esssup_Q w$ and, analogously, when $p_i=r_i$, the term corresponding to $w_i$ should be $\esssup_Q w_i^{-1}$. When $r_{m+1}=1$, the term corresponding to $w$ needs to be replaced by $\big(\fint_Q w^p\, dx\big)^{1/p}$.
\end{definition}

Let us turn to a particular class of $A_{\vec{p}, \vec{r}}$ weights, called $A_{\vec{p},q}$ weights from \cite{LOPTT} and \cite{M}. Indeed, pick $\vec{r}=(1,\ldots,1,r_{m+1})$ with $\frac{1}{r'_{m+1}}=\frac1p-\frac1q$ in \eqref{eq:w-Apr}.  Then we see that $A_{\vec{p}, \vec{r}}$ agrees with $A_{\vec{p},q}$ below. 
\begin{definition}\label{def:Apq}
Let $0<p \le q<\infty$ and $\frac1p= \frac{1}{p_1} + \cdots+ \frac{1}{p_m}$ with $1<p_1, \ldots, p_m <\infty$. Suppose that $\vec{w} = (w_1, \ldots, w_m)$ and each $w_i$ is a nonnegative locally measurable function on $\Rn$. We say that $\vec{w} \in A_{\vec{p}, q}$ if 
\begin{align}\label{eq:w-Apq}
[\vec{w}]_{A_{\vec{p},q}} := \sup_Q \bigg(\fint_Q w^q \, dx \bigg)^{\frac1q}
\prod_{i=1}^m \bigg(\fint_Q w_i^{-p_i'} dx \bigg)^{\frac{1}{p_i'}} < \infty,
\end{align}
where $w=\prod_{i=1}^m w_i$ and the supremum is taken over all cubes $Q \subset \Rn$. When $p_i = 1$, $\big(\fint_Q w_i^{-p_i'}\big)^{1/{p_i'}}$ is understood as
$(\essinf_Q w_i)^{-1}$. 
\end{definition}

In the sequel we will just simply denote $A_{\vec{p}, p}$ by $A_{\vec{p}}$. Note that for $1 < p_1, \ldots, p_m < \infty$, in light of \eqref{eq:w-Apq}, $\vec{w} \in A_{\vec{p}}$ means that  
\begin{align}\label{eq:w-Ap}
[\vec{w}]_{A_{\vec{p}}} := [\vec{w}]_{A_{\vec{p},p}} = \sup_Q \bigg(\fint_Q w^p \, dx \bigg)^{\frac1p}
\prod_{i=1}^m \bigg(\fint_Q w_i^{-p_i'} dx \bigg)^{\frac{1}{p_i'}} < \infty,
\end{align}
where $w=\prod_{i=1}^m w_i$ and $\frac1p= \frac{1}{p_1} + \cdots+ \frac{1}{p_m}$. On the other hand, $A_{\vec{p}}$ agrees with $A_{\vec{p}, (1,\ldots,1)}$ in Definition \ref{def:Apr}. We would like to observe that our definition of the classes $A_{\vec{p}}$ and $A_{\vec{p}, \vec{r}}$ is slightly different to that in \cite{LOPTT} and \cite{LMO}. Essentially, they are the same since picking $w_i=v_i^{p_i}$ for every $i=1, \ldots, m$ in \eqref{eq:w-Apr} and \eqref{eq:w-Ap}, we see that $\vec{v}=(v_1, \ldots, v_m)$ belongs to $A_{\vec{p}}$ in \cite{LOPTT} and to $A_{\vec{p}, \vec{r}}$ in \cite{LMO}, respectively. This change enables us to state our results  uniformly and conveniently no matter the weights $\vec{w}$ belong to $A_{\vec{p}}$,  $A_{\vec{p}, q}$ or $A_{\vec{p}, \vec{r}}$. 

Given $\vec{p}=(p_1, \ldots, p_m)$ with $1 \le p_1, \ldots, p_m \le \infty$ and $\vec{r}=(r_1, \ldots, r_{m+1})$ with $1 \le r_1, \ldots, r_{m+1} < \infty$ such that $\vec{r} \preceq \vec{p}$, we set 
\begin{align}\label{eq:notation-1}
\frac{1}{r}:=\sum_{i=1}^{m+1}\frac{1}{r_i}, \quad \frac{1}{p_{m+1}} := 1-\frac1p, \quad 
\frac{1}{\delta_i} := \frac{1}{r_i} - \frac{1}{p_i}, \quad i=1,\ldots, m+1. 
\end{align}
and 
\begin{align}\label{eq:notation-2}
\frac{1}{\theta_i} := \frac1r-1-\frac{1}{\delta_i} 
=\bigg(\sum_{j=1}^{m+1}\frac{1}{\delta_j} \bigg)-\frac{1}{\delta_i}, \quad i=1,\ldots,m. 
\end{align}

A characterization of $A_{\vec{p}, \vec{r}}$ was given in \cite[Lemma~5.3]{LMO} as follows. 
\begin{lemma}\label{lem:Apr}
Let $\vec{p}=(p_1,\ldots,p_m)$ with $1 \le p_1,\ldots,p_m < \infty$ and $\vec{r}=(r_1,\ldots,r_{m+1})$ with $1\le r_1,\ldots,r_{m+1}<\infty$ such that $\vec{r} \preceq \vec{p}$. Then $\vec{w} \in A_{\vec{p}, \vec{r}}$ if and only if 
\begin{align}\label{eq:Aprw}
w^{\delta_{m+1}} \in A_{\frac{1-r}{r}\delta_{m+1}} \quad\text{and}\quad  
w_i^{\theta_i} \in A_{\frac{1-r}{r}\theta_i}, \quad i=1,\ldots,m. 
\end{align}
\end{lemma}

For the $A_{\vec{p}, q}$ class, the characterizations can be formulated in the following way. 
\begin{lemma}\label{lem:Apqw}
Let $0<p \leq q<\infty$ and $\frac1p=\frac{1}{p_1} + \cdots + \frac{1}{p_m}$ with $1 \leq p_1, \ldots, p_m< \infty$. Then 
\begin{enumerate}
\item[(a)] $\vec{w} \in A_{\vec{p}, q}$ if and only if 
\begin{align}\label{eq:Apqw-1}
w^q  \in A_{mq} \quad\text{and}\quad w_i^{-p_i'} \in A_{mp_i'}, \quad i=1,\ldots,m.
\end{align}
When $p_i = 1$, $w_i^{-p_i'}$ is understood as $w_i^{1/m} \in A_1$. \vspace{0.2cm}
\item[(b)] $\vec{w} \in A_{\vec{p}, q}$ if and only if 
\begin{align}\label{eq:Apqw-2}
w^q  \in A_{(m-\frac1p+\frac1q)q} \quad\text{and}\quad w_i^{-p_i'} \in A_{(m-\frac1p+\frac1q)p_i'}, \quad i=1,\ldots,m.
\end{align} 
When $p_i = 1$, $w_i^{-p_i'} \in A_{(m-\frac1p+\frac1q)p_i'}$ is understood as $w_i^{1/(m-\frac1p+\frac1q)} \in A_1$.
\end{enumerate} 
\end{lemma}

Indeed, \eqref{eq:Apqw-1} was proved in \cite[Theorem~3.6]{LOPTT} for $p=q$ and \cite[Theorem~3.4]{M} for $p<q$, while \eqref{eq:Apqw-2} is a consequence of \eqref{eq:Aprw}. To see the latter, we take $\vec{r}=(1,\ldots,1,r_{m+1})$ with $\frac{1}{r'_{m+1}}=\frac1p-\frac1q$ in \eqref{eq:Aprw}. Then, $\frac1r=m+\frac{1}{p'}+\frac1q$ and hence, $w^{\delta_{m+1}} \in A_{\frac{1-r}{r}\delta_{m+1}}$ becomes $w^q  \in A_{(m-\frac1p+\frac1q)q}$. In addition, $w_i^{\theta_i} \in A_{\frac{1-r}{r}\theta_i}$ becomes $w_i^{-p'_i[1-((m-\frac1p+\frac1q)p'_i)']} \in A_{((m-\frac1p+\frac1q)p'_i)'}$, which is equivalent to $w_i^{-p'_i} \in A_{(m-\frac1p+\frac1q)p'_i}$. This shows \eqref{eq:Apqw-2}. On the other hand, it is worth pointing out that the characterization \eqref{eq:Apqw-2} refines \cite[Theorem~3.7]{CWX} by removing the restriction $1\le p_1,\ldots,p_m<mn/\alpha$ and $\frac1p-\frac1q=\frac{\alpha}{n}$.

Beyond that, the $A_{\vec{p}, \vec{r}}$ class enjoys the following properties. 
\begin{lemma}\label{lem:Apw}
Let $\vec{p}=(p_1,\ldots,p_m)$ with $1 \le p_1,\ldots,p_m < \infty$ and $\vec{r}=(r_1,\ldots,r_{m+1})$ with $1\le r_1,\ldots,r_{m+1}<\infty$ such that $\vec{r} \prec \vec{p}$. Then the following statements hold:
\begin{list}{\textup{(\theenumi)}}{\usecounter{enumi}\leftmargin=1cm \labelwidth=1cm \itemsep=0.2cm 
			\topsep=.2cm \renewcommand{\theenumi}{\arabic{enumi}}}
\item\label{list:Apr-1} $A_{\vec{p}, \vec{s}} \subsetneq A_{\vec{p}, \vec{r}}$ for any $\vec{r} \prec \vec{s} \prec \vec{p}$. 
\item\label{list:Apr-2} $A_{\vec{p}, \vec{r}}=\bigcup_{\vec{r} \prec \vec{s} \prec \vec{p}} A_{\vec{p}, \vec{s}}=\bigcup_{1<t<t_0} A_{\vec{p}, \gamma_t(\vec{r})}$, where $t_0=\min_{1\le i \le m}\{p_i/r_i\}$ and $\gamma_t(\vec{r})=(tr_1, \ldots, tr_m, r_{m+1})$. 
\item\label{list:Apr-3} $A_{s_1, t_1} \times \cdots \times A_{s_m, t_m} \subsetneq A_{\vec{p}, \vec{r}}$ for all $\vec{s}=(s_1,\ldots,s_m) \preceq \vec{t}=(t_1,\ldots,t_m)$ with $\frac{1}{s_i}=1-\frac{1}{r_i}+\frac{1}{p_i}$ and $\frac{1}{t_1}+\cdots+\frac{1}{t_m}=\frac1p-\frac{1}{r'_{m+1}}$, where $\frac1p=\frac{1}{p_1}+\cdots+\frac{1}{p_m}$. 
\end{list}
\end{lemma}

\begin{proof}
We begin with showing \eqref{list:Apr-1}. Note that for any $\vec{r} \prec \vec{s} \preceq \vec{p}$, one has $\frac{r'_{m+1}}{r'_{m+1}-p} < \frac{s'_{m+1}}{s'_{m+1}-p}$ and $\frac{r_i}{p_i-r_i}<\frac{s_i}{p_i-s_i}$, $i=1,\ldots,m$. Then, this and Jensen's inequality give that $A_{\vec{p}, \vec{s}} \subset A_{\vec{p}, \vec{r}}$. In order to conclude \eqref{list:Apr-1}, it remains to find a vector of weights $\vec{w}$ such that $\vec{w} \in A_{\vec{p},\vec{r}}$ and $\vec{w} \not\in A_{\vec{p},\vec{s}}$. By definition, $\theta_i \leq \delta_{m+1}$ for each $i=1,\ldots,m$. Since the $A_p$ classes are increasing, we have $A_{\frac{1-s}{s}\theta_1} \subset A_{\frac{1-r}{r}\theta_1} \subset A_{\frac{1-r}{r}\delta_{m+1}}$. Pick $p_0:= \frac{1-s}{s}\theta_1$ and $w_0 =|x|^{n(p_0 - 1)}$. Then, it is easy to see that $w_0 \notin A_{\frac{1-s}{s}\theta_1}$ and $w_0 \in A_{\frac{1-r}{r}\theta_1}$. In addition, $w_1:= w_0^{1/\theta_1}$ satisfies that $w_1^{\theta_1} \in A_{\frac{1-r}{r}\theta_1}$, but $w_1^{\theta_1} \notin A_{\frac{1-s}{s}\theta_1}$. Even more, $w_1^{\delta_{m+1}} = |x|^{n(p_0 -1)\frac{\delta_{m+1}}{\theta_1}} \in A_{\frac{1-s}{s}\delta_{m+1}}$ and then $w_1^{\delta_{m+1}} \in  A_{\frac{1-r}{r}\delta_{m+1}}$.
Therefore, taking $\vec{w}:=(w_1,1,\dots,1)$, by Lemma \ref{lem:Apr} we conclude that $\vec{w} \in A_{\vec{p},\vec{r}}$, but $\vec{w} \notin A_{\vec{p},\vec{s}}$.

We next turn to \eqref{list:Apr-2}. We first demonstrate $A_{\vec{p}, \vec{r}}=\bigcup_{\vec{r} \prec \vec{s} \prec \vec{p}} A_{\vec{p}, \vec{s}}$. In view of \eqref{list:Apr-1}, it suffices to prove that for any $\vec{w} \in A_{\vec{p}, \vec{r}}$, there exists $\vec{r}\prec \vec{s} \prec \vec{p}$ such that $\vec{w} \in A_{\vec{p}, \vec{s}}$. Fix $\vec{w} \in A_{\vec{p}, \vec{r}}$. By Lemma \ref{lem:Apr}, one has  
\begin{align}\label{eq:wrr}
w^{\delta_{m+1}} \in A_{\frac{1-r}{r}\delta_{m+1}} \quad\text{and}\quad  
w_i^{\theta_i} \in A_{\frac{1-r}{r}\theta_i}, \quad i=1,\ldots,m. 
\end{align}
Recall that $v \in A_q$ with $1<q<\infty$ implies that $v^{\tau} \in A_{q/\kappa}$ for some $1<\kappa<q$ and $1<\tau<\infty$. Using this fact and \eqref{eq:wrr}, we obtain that 
\begin{align}\label{eq:wkap}
w_i^{\tau_i\theta_i} \in A_{\frac{1-r}{r\kappa_i}\theta_i}, \quad i=1,\ldots,m+1,  
\end{align}
for some $1<\kappa_i<\frac{1-r}{r}\theta_i$ and $1<\tau_i<\infty$, where $\theta_{m+1}:=\delta_{m+1}$. Let $\varepsilon \in (0, 1)$ chosen later. Define 
\begin{align*}
\frac{1}{s_i} := \frac{1-\varepsilon}{r_i} + \frac{\varepsilon}{p_i},\quad 
\frac{1}{\widetilde{\delta}_i} := \frac{1}{s_i} - \frac{1}{p_i}, \quad i=1,\ldots, m+1, 
\quad \widetilde{\theta}_{m+1} := \widetilde{\delta}_{m+1}, 
\end{align*}
and 
\begin{align*}
\frac{1}{s}:=\sum_{i=1}^{m+1}\frac{1}{s_i}, \quad 
\frac{1}{\widetilde{\theta}_i} := \frac1s-1-\frac{1}{\widetilde{\delta}_i} 
=\bigg(\sum_{j=1}^{m+1}\frac{1}{\widetilde{\delta}_j} \bigg)-\frac{1}{\widetilde{\delta}_i}, \quad i=1,\ldots,m. 
\end{align*}
Then we see that $\vec{s}$, $\widetilde{\delta}_i$ and $\widetilde{\theta}_i$ depend on $\varepsilon$, $\vec{r} \prec \vec{s} \preceq \vec{p}$ for every $\varepsilon \in (0, 1)$, and 
\begin{align*}
\frac{\widetilde{\theta}_i}{\theta_i} \to 1^{+} \quad\text{and} 
\quad \frac{\frac1r-1}{\frac1s-1} \to 1^{+}, \quad \text{as } \varepsilon \to 0. 
\end{align*}
This means that one can pick $\varepsilon \in (0, 1)$ small enough such that 
\begin{align}\label{eq:ttss}
\widetilde{\theta}_i \le \tau_i \theta_i \quad\text{and}\quad 
\frac{1-r}{r\kappa_i} \le \frac{1-s}{s}, \quad i=1,\ldots,m+1. 
\end{align}
From \eqref{eq:wkap} and \eqref{eq:ttss}, we have 
\begin{align}\label{eq:Ass}
w_i^{\widetilde{\theta}_i} \in A_{\frac{1-s}{s}\theta_i} \subset A_{\frac{1-s}{s}\widetilde{\theta}_i}, \quad i=1,\ldots,m+1.  
\end{align}
Therefore, it follows from \eqref{eq:Ass} and Lemma \ref{lem:Apr} that $\vec{w} \in A_{\vec{p}, \vec{s}}$. Likewise, one can get $A_{\vec{p},\vec{r}}=\bigcup_{1<t<t_0} A_{\vec{p}, \gamma_t(\vec{r})}$. 

Finally, let us demonstrate \eqref{list:Apr-3}. Fix $\vec{s}=(s_1,\ldots,s_m) \preceq \vec{t}=(t_1,\ldots,t_m)$ with $\frac{1}{s_i}=1-\frac{1}{r_i}+\frac{1}{p_i}$ and $\frac{1}{t_1}+\cdots+\frac{1}{t_m}=\frac1p-\frac{1}{r'_{m+1}}$. Let $\vec{w} \in A_{s_1, t_1} \times \cdots \times A_{s_m, t_m}$. Then H\"{o}lder's inequality gives that 
\begin{align*}
&\bigg(\fint_Q w^{\frac{r'_{m+1}p}{r'_{m+1}-p}}\, dx\bigg)^{\frac{1}{p}-\frac{1}{r'_{m+1}}} 
\prod_{i=1}^m \bigg(\fint_Q w_i^{\frac{r_i p_i}{r_i-p_i}}dx \bigg)^{\frac{1}{r_i}-\frac{1}{p_i}}
\\
&\quad\le \prod_{i=1}^m \bigg(\fint_Q w_i^{t_i}\, dx\bigg)^{\frac{1}{t_i}} 
\bigg(\fint_Q w_i^{-\frac{1}{s'_i}}dx \bigg)^{\frac{1}{s'_i}}
\le  \prod_{i=1}^m [w_i]_{A_{s_i, t_i}}, 
\end{align*} 
which implies $[\vec{w}]_{A_{\vec{p}, \vec{r}}} \le \prod_{i=1}^m [w_i]_{A_{s_i, t_i}}$ and so, $A_{s_1, t_1} \times \cdots \times A_{s_m, t_m} \subset A_{\vec{p}, \vec{r}}$. To show the strict containment, we construct an example such that $\vec{w} \in A_{\vec{p}, \vec{r}}$ and $\vec{w} \not\in A_{s_1,t_1} \times \cdots \times A_{s_m,t_m}$. We pick $w_1(x)=|x|^{-n/t_1}$. Then $w_1^{t_1} \not\in L^1_{\loc}(\Rn)$, but $w_1^{\delta_{m+1}}=|x|^{-nt/t_1} \in A_1$, where $\frac1t:=\frac{1}{t_1}+\cdots+\frac{1}{t_m}=\frac1p-\frac{1}{r'_{m+1}}=\frac{1}{\delta_{m+1}}$. Since $\theta_1<\delta_{m+1}$, we have $w_1^{\theta_1} \in A_1 \subset A_{\frac{1-r}{r}\theta_1}$. Hence, from Lemma \ref{lem:Apr}, we see that $\vec{w}:=(w_1,1,\ldots,1) \in A_{\vec{p}, \vec{r}}$, but $\vec{w} \not\in A_{s_1, t_1} \times \cdots \times A_{s_m, t_m}$. 
\end{proof}

\begin{lemma}\label{lem:lim-Apr}
Let $1\le \p_i^{-} <\p_i^{+} \le \infty$ and $p_i \in (\p_i^{-}, \p_i^{+})$, $i=1,\ldots,m$. If $w_i^{p_i} \in A_{\frac{p_i}{\p_i^{-}}} \cap RH_{\big(\frac{\p_i^{+}}{p_i}\big)'}$, $i=1,\ldots,m$, then $\vec{w}=(w_1,\ldots,w_m) \in A_{\vec{t}, \vec{r}}$, where $\vec{r}=(r_1,\ldots,r_m,1)$ $t_i=p_i(\p_i^{+}/p_i)'$, $r_i=t_i/\tau_i$, and $\tau_i=\big(\frac{\p_i^+}{p_i}\big)'(\frac{p_i}{\p_i^{-}}-1)+1$, $i=1,\ldots,m$.
\end{lemma}

\begin{proof}
Let $w_i^{p_i} \in A_{\frac{p_i}{\p_i^{-}}} \cap RH_{\big(\frac{\p_i^{+}}{p_i}\big)'}$, $i=1,\ldots,m$. Then by \eqref{eq:JN}, we see that $w_i^{t_i} \in A_{\tau_i}$, $i=1,\ldots,m$. Note that $r_i=t_i/\tau_i \ge 1$. Set $s'_i=t_i(\tau'_i-1)$. Then 
\begin{equation}\label{eq:rst}
\frac{1}{s_i}=1-\frac{1}{s'_i}=1-\frac{\tau_i}{t_i}+\frac{1}{t_i}=1-\frac{1}{r_i}+\frac{1}{t_i}. 
\end{equation}
On the other hand, by definition, 
\begin{align*}
[w_i^{t_i}]_{A_{\tau_i}} 
&= \sup_{Q} \bigg(\fint_Q w_i^{t_i}\, dx\bigg) \bigg(\fint_Q w_i^{-t_i(\tau'_i-1)}\bigg)^{\tau_i-1} 
\\
&=\sup_{Q} \bigg[\bigg(\fint_Q w_i^{t_i}\, dx\bigg)^{\frac{1}{t_i}} \bigg(\fint_Q w_i^{-s'_i}\bigg)^{\frac{1}{s'_i}}\bigg]^{t_i}
=[w_i]_{A_{s_i, t_i}}^{t_i},  
\end{align*}
which shows that $\vec{w}=(w_1,\ldots,w_m) \in A_{s_1, t_1} \times \cdots \times A_{s_m, t_m}$. This along with \eqref{eq:rst} and Lemma \ref{lem:Apw} (3) implies $\vec{w} \in A_{\vec{t}, \vec{r}}$. 
\end{proof}

\subsection{Characterizations of compactness}  
The weighted Fr\'{e}chet-Kolmogorov theorem below provides a way to characterize the relative compactness of a set in $L^p(w)$. In the unweighted setting, it was proved by Yosida \cite[p.~275]{Yosida} in the case $1\le p<\infty$, which is extended by Tsuji \cite{Tsuji} to the case $0<p<1$.  Hereafter, we always denote $\tau_hf(x)=f(x+h)$.  

\begin{theorem}\label{thm:FK-1}
Let $p \in (0, \infty)$, and let $w$ be a weight on $\Rn$ such that $w, w^{-\lambda} \in L^1_{\loc}(\Rn)$ for some $\lambda \in(0, \infty)$. 
\begin{enumerate}
\item[(a)] A subset $\mathcal{G} \subset L^p(w)$ is relatively compact if the following are satisfied: 
\begin{enumerate}
\item[(a-1)] $\sup\limits_{f \in \mathcal{G}} \|f\|_{L^p(w)}<\infty$, \vspace{0.2cm}
\item[(a-2)] $\lim\limits_{A\to\infty}\sup\limits_{f \in \mathcal{G}} \|f \mathbf{1}_{\{|x|>A\}}\|_{L^p(w)}=0$, \vspace{0.2cm}
\item[(a-3)] $\lim\limits_{|h|\to 0} \sup\limits_{f \in \mathcal{G}} \|\tau_hf - f\|_{L^p(w)}=0$. 
\end{enumerate}
\item[(b)] The conditions {\rm (a-1)} and {\rm (a-2)} are necessary, but {\rm (a-3)} is not. 
\item[(c)] If there exists $\delta>0$ such that $\tau_h w \lesssim w$ uniformly for any $|h|<\delta$, then the conditions {\rm (a-1)}, {\rm (a-2)} and {\rm (a-3)} are necessary. 
\end{enumerate}
\end{theorem}

\begin{proof}
We only focus on parts (b) and (c) since part (a) is contained in \cite[Lemma~4.2]{XYY}. Indeed, checking the proof of \cite[Lemma~4.2]{XYY}, one can find that it only requires $w, w^{1-p'_0} \in L^1_{\loc}(\Rn)$ for some $p_0 \in (1, \infty)$. Then taking $p_0=1+\frac{1}{\lambda}$, part (a) follows. To show part (b), let $\mathcal{G}$ be relatively compact in $L^p(w)$. Then $\mathcal{G}$ is bounded, and (a-1) holds. Let $\varepsilon>0$ be given. Then there exists a finite number of functions $f_1,\ldots,f_m \in L^p(w)$ such that, for each $f\in L^p(w)$ there is an $f_j$ with $\|f-f_j\|_{L^p(w)}\le \varepsilon$. Otherwise, we would have an infinite sequence $\{f_j\}\subset \mathcal{G}$ with $\|f_j-f_i\|_{L^p(w)}> \varepsilon$ for $i\ne j$, which is contrary to the relative compactness of $\mathcal{G}$. We then find simple functions (finitely-valued functions with compact support) $g_1,\ldots,g_m$ such that $\|f_j-g_j\|_{L^p(w)}\le \varepsilon$ $(j=1,2,\dots, m)$. Since each simple function $g_j(x)$ vanishes outside some sufficienty large ball $B(0,A)$, we have for any $f\in \mathcal{G}$, 
\begin{align*}
\|f \mathbf{1}_{\{|x|>A\}}\|_{L^p(w)} 
&\lesssim \|(f-g_j) \mathbf{1}_{\{|x|>A\}}\|_{L^p(w)} + \|g_j\mathbf{1}_{\{|x|>A\}}\|_{L^p(w)}
\\
&\lesssim \|f-f_j\|_{L^p(w)}+\|f_j-g_j\|_{L^p(w)}+0 \le 2\varepsilon.
\end{align*}
This proves (a-2). 

Next, we construct some examples to show that the condition (a-3) is not necessary. Let $w(x)=|x|^{1/2}$ and $f(x)=|x|^{-3/5} {\bf 1}_{\{|x| \le 1\}}$. Then, $w\in A_2(\R)$ and $f\in L^2(w)$. But, 
\begin{equation*}
\|f(\cdot+h)\|_{L^2(w)}=\infty\ \text{ for any } h\not=0.
\end{equation*}
Let $\mathcal{G}:=\{f\}$. Then $\mathcal{G}$ is a compact set in $L^2(w)$. However, 
\begin{equation*}
\int_{\R}|f(x+h)-f(x)|^2 w(x)dx=+\infty\ \text{ for any } h\not=0.
\end{equation*}
Thus $\mathcal{G}$ does not satisfy (a-3). Let us give another example. Let $1<p_0<p<\infty$ and $1/p<\alpha<p_0/p$. Set 
\begin{align*}
w(x)=|x|^{p_0-1} \quad\text{ and }\quad f(x)=|x|^{-\alpha}\mathbf{1}_{\{|x|\le1\}}.  
\end{align*} 
Then we get $p_0-1-p\alpha>-1$ and $p\alpha>1$, and hence 
\begin{align*}
w\in A_p(\R),\quad f\in L^p(w), \quad\text{ but}\quad \tau_h f \notin L^p(w),\quad \forall h\not=0.
\end{align*} 
Hence, letting $\mathcal{G}=\{f\}$, we see that $\mathcal{G}$ is a compact set in $L^p(w)$, but $\mathcal{G}$ does not satisfy (a-3).

To conclude part (c), it suffices to prove (a-3) is necessary under the additional assumption in (c). Let $\varepsilon>0$ and $f \in \mathcal{G}$. Since $\mathcal{G}$ is relatively compact, there exists a finite number of functions $\{f_j\}_{j=1}^m \subset L^p(w)$ such that for each $f \in \mathcal{G}$, there exists some $f_j$ such that $\|f-f_j\|_{L^p(w)}<\varepsilon$. Since $\mathscr{C}_c^{\infty}(\Rn)$ is dense in $L^p(w)$, there exists $g_j \in \mathscr{C}_c^{\infty}(\Rn)$ such that $\|f_j-g_j\|_{L^p(w)}<\varepsilon$. Additionally, there exists $\delta_0>0$  such that for any $|h|<\delta_0$, 
\begin{equation}\label{eq:tgg}
\|\tau_h g_j-g_j\|_{L^p(w)} < \varepsilon. 
\end{equation}
Now, since $\tau_h w \lesssim w$ for all $|h|<\delta$, 
\begin{equation}\label{eq:tftf}
\|\tau_h f - \tau_h f_j\|_{L^p(w)} =\|f-f_j\|_{L^p(\tau_{-h}w)} \lesssim \|f-f_j\|_{L^p(w)} < \varepsilon.  
\end{equation}
Similarly, 
\begin{equation}\label{eq:tftg}
\|\tau_h f_j - \tau_h g_j\|_{L^p(w)} \lesssim  \varepsilon.  
\end{equation}
Collecting \eqref{eq:tgg}, \eqref{eq:tftf} and \eqref{eq:tftg}, we get for any $|h|<\min\{\delta, \delta_0\}$, 
\begin{align*}
\|\tau_h f - f\|_{L^p(w)} &\le \|\tau_h f - \tau_h f_j\|_{L^p(w)} + \|\tau_h f_j - \tau_h g_j\|_{L^p(w)} 
\\
&\quad+ \|\tau_h g_j - g_j\|_{L^p(w)} + \|g_j - f_j\|_{L^p(w)} + \|f_j - f\|_{L^p(w)}
\\
& \lesssim  \varepsilon,  
\end{align*}
which gives that 
\[
\lim_{|h|\to 0} \|\tau_hf - f\|_{L^p(w)}=0,\quad\text{uniformly in } f \in \mathcal{G}.
\]
This completes the proof.  
\end{proof}

We present another characterization of the relative compactness of a subset in $L^p(w)$. 
\begin{theorem}\label{thm:FK-2}
Let $1<p<\infty$ and $w \in A_p$. Then a subset $\G \subseteq L^p(w)$ is relatively compact if and only if the following are satisfied:
\begin{list}{\rm (\theenumi)}{\usecounter{enumi}\leftmargin=1cm \labelwidth=1cm \itemsep=0.2cm \topsep=.2cm \renewcommand{\theenumi}{\arabic{enumi}}}
 
\item\label{list-1} $\sup\limits_{f \in \G} \|f\|_{L^p(w)} < \infty$, 
\item\label{list-2} $\lim\limits_{A \to \infty} \sup\limits_{f \in \G}\|f \mathbf{1}_{\{|x|>A\}}\|_{L^p(w)}=0$, 
\item\label{list-3} $\lim\limits_{r \to 0} \sup\limits_{f \in \G}\|f-f_{B(\cdot,r)}\|_{L^p(w)}=0$. 
\end{list}
\end{theorem}

\begin{proof}
The sufficiency is essentially contained in the proof of \cite[Lemma~4.1]{XYY}. We present the detailed proof for the convenience of the reader. Suppose that \eqref{list-1}, \eqref{list-2}, and \eqref{list-3} hold. Then by \eqref{list-2} and \eqref{list-3}, for any fixed $\varepsilon>0$, there exist $A>0$ and $\delta>0$ such that for $0<r<\delta$, 
\begin{align}\label{eq:compactset-1} 
\|f \mathbf{1}_{\{|x|>A\}}\|_{L^p(w)} < \varepsilon \quad\text{ and }\quad 
\|f-f_{B(\cdot, r)}\|_{L^p(w)}<\varepsilon \ \text{ for all } f \in \G.
\end{align}
Fix such an $r>0$. For any $x, y \in \Rn$, by H\"{o}lder's inequality, we have 
\begin{equation}\label{eq:bdd}
|f_{B(x, r)}| 
\le \bigg(\fint_{B(x, r)} |f(y)|^p w(y) \, dy \bigg)^{\frac1p}
\bigg(\fint_{B(x, r)} w(y)^{1-p'} dy \bigg)^{\frac1{p'}}, 
\end{equation}
and 
\begin{align}\label{eq:continue}
|f_{B(x, r)}-f_{B(y, r)}| 
& = \frac{1}{|B(0, r)|}\int_{\Rn}|\mathbf{1}_{B(x,r)} - \mathbf{1}_{B(y,r)}||f(z)| \,dz
\\ \nonumber 
& \le \frac{ \|f\|_{L^p(w)}}{|B(0, r)|}
\bigg(\int_{\Rn}|\mathbf{1}_{B(x,r)} - \mathbf{1}_{B(y,r)}|w(z)^{1-p'} \,dz \bigg)^{\frac1{p'}}. 
\end{align} 
Since $w \in A_p$ is equivalent to $w^{1-p'} \in A_{p'}$, it follows that $w^{1-p'} \in L^1_{\loc}(\Rn)$, which in turn implies that there exists $C_0>0$ such that 
 \begin{align}\label{eq:BA}
\int_{B(x, r)} w(z)^{1-p'} dz \le C_0 \quad\text{for all } |x| \le A. 
\end{align}
From \eqref{eq:bdd}--\eqref{eq:BA}, we see that $\{f_{B(x, r)}\}_{f \in \G}$ is equi-bounded and equi-continuous on the closed ball $B(0, A)$. Then, by Ascoli-Arzel\'a theorem, it is relatively compact and so totally bounded in $\mathscr{C}(B(0, A))$. Consequently, there exists a finite number of functions $\{f_j\}_{j=1}^N \subset \G$ such that 
\begin{equation*}
\inf_{j}\sup_{|x|\le A}|f_{B(x, r)}-f_{j, B(x, r)}| < \varepsilon \, w(B(0, A))^{-\frac1p} \ \text{ for all }f\in \G.
\end{equation*}
It follows that for each $f\in \G$ there exists $j \in \{1,\ldots, N\}$ such that 
\begin{equation}\label{eq:compactset-4}
\sup_{|x|\le A}|f_{B(x, r)} - f_{j, B(x, r)}| < \varepsilon \, w(B(0,A))^{-\frac1p}.
\end{equation}
Then invoking \eqref{eq:compactset-1} and \eqref{eq:compactset-4}, we obtain 
\begin{align*}
\|f-f_j\|_{L^p(w)}
&\le \|(f-f_j)\mathbf{1}_{\{|x|\le A\}}\|_{L^p(w)} + \|(f-f_j) \mathbf{1}_{\{|x|> A\}}\|_{L^p(w)}
\\
&\le \|(f-f_{B(x, r)}) \mathbf{1}_{\{|x| \le A\}}\|_{L^p(w)}
+ \|(f_{B(x, r)} - f_{j, B(x, r)}) \mathbf{1}_{\{|x|\le A\}}\|_{L^p(w)}
\\
&\quad+\|(f_{j,B(x, r)}-f_j) \mathbf{1}_{\{|x|\le A\}}\|_{L^p(w)}
+ \|f \mathbf{1}_{\{|x|>A\}}\|_{L^p(w)} + \|f_j \mathbf{1}_{\{|x|>A\}}\|_{L^p(w)} 
\\
&\le \|f-f_{B(x, r)}\|_{L^p(w)}
+ \varepsilon \, w(B(0, A))^{-\frac1p} \|\mathbf{1}_{\{|x|\le A\}}\|_{L^p(w)}
\\
&\quad+\|f_j - f_{j,B(x, r)}\|_{L^p(w)}
+ \|f \mathbf{1}_{\{|x|>A\}}\|_{L^p(w)} + \|f_j \mathbf{1}_{\{|x|>A\}}\|_{L^p(w)}
\\
&< 5\varepsilon,
\end{align*}
which implies that $\G$ is totally bounded. Therefore, $\G$ is relatively compact in $L^p(w)$.

Let us next prove the necessity. Let $\varepsilon>0$. Since $\G$ is relatively compact, it is totally bounded. Thus, there exists a finite number of functions $\{f_j\}_{j=1}^N \subset \G$ such that $\G \subseteq \bigcup_{k=1}^N B(f_k,\varepsilon)$. Let $f \in \G$ be an arbitrary function. Then for some $k \in \{1,\ldots,N\}$, 
\begin{align}\label{eq:fk-f}
\|f_k-f\|_{L^{p}(w)} < \varepsilon.
\end{align}
The condition \eqref{list-1} is satisfied because 
\begin{align*}
\|f\|_{L^{p}(w)} 
\leq \|f-f_k\|_{L^{p}(w)} + \|f_k\|_{L^{p}(w)}
<1+\max_{1 \leq k \leq N} \|f_k\|_{L^{p}(w)}. 
\end{align*}
Since $f_k \in L^{p}(w)$, there exists $A_k>0$ such that 
\begin{equation}\label{eq:fkAk}
\|f_k \mathbf{1}_{\{|x|>A_k\}}\|_{L^{p}(w)} < \varepsilon, \quad k=1,\ldots,N. 
\end{equation}
Set $A:=\max\{A_k: k=1,\ldots,N\}$. Then by \eqref{eq:fk-f} and \eqref{eq:fkAk},  
\begin{align*}
\|f \mathbf{1}_{\{|x|>A\}}\|_{L^{p}(w)} \leq \|f-f_k\|_{L^{p}(w)} + 
\|f_k \mathbf{1}_{\{|x|>A_k\}}\|_{L^{p}(w)}
< 2\varepsilon. 
\end{align*}
This shows \eqref{list-2} holds. Now with \eqref{eq:fk-f} in hand, we split 
\begin{align*}
\|f-f_{B(\cdot,r)}\|_{L^{p}(w)} 
\leq \|f-f_k\|_{L^{p}(w)} + \|f_k-(f_k)_{B(\cdot,r)}\|_{L^{p}(w)}
+ \|(f_k)_{B(\cdot,r)}-f_{B(\cdot,r)}\|_{L^{p}(w)}. 
\end{align*}
The first term is controlled by $\varepsilon$. Note that 
\begin{align}\label{eq:fkfk-1}
|f_k(x)-(f_k)_{B(x,r)}| \lesssim |f_k(x)| + Mf_k(x) \in L^p(w),  
\end{align} 
and $L^p(w) \subset L^1_{\loc}(\Rn)$. By this and Lebesgue differentiation theorem, 
\begin{align}\label{eq:fkfk-2}
(f_k)_{B(x,r)} \to f_k(x), \text{ as } r \to 0^+, \quad\text{a.e. } x \in \Rn. 
\end{align}
Then, in light of \eqref{eq:fkfk-1} and \eqref{eq:fkfk-2}, the Lebesgue domination convergence theorem gives that 
\begin{align*}
\|f_k-(f_k)_{B(\cdot,r)}\|_{L^{p}(w)} < \varepsilon,\quad \forall r \in (0, \delta), 
\end{align*}
for some $\delta>0$. As for the last term, one has 
\begin{align*}
|(f_k)_{B(x,r)}-f_{B(x,r)}| \leq \fint_{B(x,r)} |f_k(y)-f(y)| dy \leq M(f_k-f)(x). 
\end{align*}
Hence, we obtain 
\begin{align*}
\|(f_k)_{B(\cdot,r)}-f_{B(\cdot,r)}\|_{L^{p}(w)} 
\leq \|M(f_k-f)\|_{L^{p}(w)}  
\lesssim \|f_k-f\|_{L^{p}(w)}  
\lesssim \varepsilon. 
\end{align*}
Collecting these estimates, we deduce that for any $0<t<\delta$, 
\begin{align*}
\|f-f_{B(\cdot,r)}\|_{L^{p}(w)} \lesssim \varepsilon,\quad\text{uniformly in } f \in \G.  
\end{align*}
This concludes that \eqref{list-3} holds. 
\end{proof}

We will extend Theorem \ref{thm:FK-2} to the case $0<p \le 1$ as follows.  
\begin{theorem}\label{thm:FK-3}
Let $0<p<\infty$ and $w \in A_{p_0}$ with $1<p_0<\infty$. Then a subset $\G \subseteq L^p(w)$ is relatively compact if and only if the following are satisfied:
\begin{enumerate}
\item $\sup\limits_{f \in \G} \|f\|_{L^p(w)} < \infty$, 
\item $\lim\limits_{A \to \infty} \sup\limits_{f \in \G}\|f \mathbf{1}_{\{|x|>A\}}\|_{L^p(w)}=0$, 
\item ${\displaystyle \lim\limits_{r \to 0} \sup\limits_{f \in \G} \int_{\Rn} \bigg(\fint_{B(0, r)} |f(x)-f(x+y)|^{\frac{p}{p_0}} dy\bigg)^{p_0}w(x)dx=0}$. 
\end{enumerate}
\end{theorem}

\begin{proof}
Assume that (1), (2), and (3) hold. We first consider the case $p \ge p_0$. Observe that 
\begin{align*}
|f(x)-f_{B(x, r)}| \le \fint_{B(0, r)} |f(x)-f(x+y)| dy 
\le \bigg(\fint_{B(0, r)} |f(x)-f(x+y)|^{\frac{p}{p_0}} dy \bigg)^{\frac{p_0}{p}}. 
\end{align*}
This and (3) imply that 
\begin{equation}\label{eq:f-fB}
\lim\limits_{r \to 0} \sup\limits_{f \in \G}\|f-f_{B(\cdot,r)}\|_{L^p(w)}=0.
\end{equation} 
Note that $w \in A_{p_0} \subset A_p$. With (1), (2) and \eqref{eq:f-fB} in hand, by Theorem \ref{thm:FK-2}, we deduce that $\G$ is relatively compact in $L^p(w)$. 

Let us handle the case $p<p_0$. We first consider the case when $\G$ is a family of non-negative functions. Write $a:=p/p_0<1$. Then we see that 
\begin{align}\label{eq:fafa}
\big|f	^a(x) - (f^a)_{B(x, r)} \big| \le \fint_{B(0, r)} |f(x)-f(x+y)|^{\frac{p}{p_0}} dy, 
\end{align}
and, (1) and (2) are equivalent to 
\begin{align}\label{eq:fGfG}
\sup\limits_{f \in \G} \| f^a\|_{L^{p_0}(w)} < \infty \quad\text{and}\quad 
\lim\limits_{A \to \infty} \sup\limits_{f \in \G}\| f^a \mathbf{1}_{\{|x|>A\}}\|_{L^{p_0}(w)}=0. 
\end{align}
By \eqref{eq:fafa} and (3), there holds
\begin{equation}\label{eq:fa-fB}
\lim\limits_{r \to 0} \sup\limits_{f \in \G}\| f^a - (f^a)_{B(\cdot, r)}\|_{L^{p_0}(w)}=0.
\end{equation} 
Hence, from \eqref{eq:fGfG}, \eqref{eq:fa-fB}, $w \in A_{p_0}$ and Theorem \ref{thm:FK-2}, it follows that $\G^a:=\{f^a: f \in \G\}$ is relatively compact in $L^{p_0}(w)$. Now let $\{f_j\}$ be a sequence of functions in $\G$. Since $\G^a$ is relatively 
compact in $L^{p_0}(w)$, there exists a Cauchy subsequence of $\{f_j^a\}$, which we denote again by $\{f_j^a\}$ for simplicity. Then for any $\varepsilon>0$, there exists an integer $N$ such that for all $i, j\ge N$, 
\begin{equation}\label{eq:faij}
\int_{\Rn} \big|f_i^a(x) - f_j^a(x) \big|^{p_0}w(x)dx<\varepsilon ^{p_0}.
\end{equation}
For fixed $i, j \in \N$, we wet 
\begin{equation*}
E_\varepsilon :=\bigg\{x \in \Rn: \frac{f_i(x)+f_j(x)}{|f_i(x)-f_j(x)|}\le \frac{1}{\varepsilon} \bigg\}, \quad\forall\varepsilon>0.
\end{equation*}
By elementary calculation (see \cite{Tsuji}), for any $a \in (0,1)$
\begin{equation}\label{eq:sata}
|s^a-t^a|\le |s-t|^a \le \frac{1}{a}\bigg(\frac{s+t}{|s-t|}\bigg)^{1-a}|s^a-t^a|,\quad \text{ for all }s,t>0.
\end{equation}
Then, using $p_0 a=p$, \eqref{eq:faij} and \eqref{eq:sata}, we have
\begin{align*}
\int_{E_\varepsilon}|f_i(x)-f_j(x)|^{p}w(x)dx
&\le a^{-p_0}\varepsilon^{(a-1)p_0}
\int_{E_\varepsilon}|f_i^a(x)-f_j^a(x)|^{p_0}w(x)dx
\\
&\le a^{-p_0}\varepsilon^{(a-1)p_0}\varepsilon^{p_0}
=a^{-p_0}\varepsilon^{p}.
\end{align*}
On the other hand, \eqref{eq:sata} and (1) give 
\begin{align*}
\int_{E_\varepsilon^c} & |f_i(x)-f_j(x)|^{p}w(x)dx
\le \int_{E_\varepsilon^c}|\varepsilon(f_i(x)+f_j(x))|^{p}w(x)dx
\\
&\qquad\le \varepsilon^p \bigg(\int_{E_\varepsilon^c}|f_i(x)|^{p}w(x)dx
+\int_{E_\varepsilon^c}|f_j(x)|^{p}w(x)dx\bigg)
\le 2K^p\varepsilon^p, 
\end{align*}
where $K:=\sup\limits_{f \in \G} \|f\|_{L^p(w)}<\infty$. The two estimates above show that $\{f_j\}$ is a Cauchy sequence in 
$\G \subset L^p(w)$. Thus $\G$ is relatively compact in $L^p(w)$. 

To handle the general case, we denote 
\[
\G^+ := \{f^+: f \in \G\} \quad\text{and}\quad \G^- := \{f^-: f \in \G\}, 
\]
where 
\[
f^+(x) :=\max\{f(x), 0\} \quad\text{and}\quad 
f^-(x) := \max\{-f(x), 0 \}, \quad\forall f \in \G. 
\]
Then we see that for all $f \in \G$ and $x, y \in \Rn$, 
\begin{align*}
0 \le f^+(x) \le |f(x)|, \qquad |f^+(x) - f^+(x+y)| \le |f(x) - f(y)|, 
\\
0 \le f^-(x) \le |f(x)|, \qquad |f^-(x) - f^-(x+y)| \le |f(x) - f(y)|. 
\end{align*}
This means that 
\begin{align}\label{eq:GG}
\text{both $\G^+$ and $\G^-$ satisfy the conditions (1), (2), and (3)}.
\end{align} 
Let $\{f_j\}$ be an arbitrary sequence of functions in $\G$. By \eqref{eq:GG} and the conclusion in the preceding case, we conclude that for any $\varepsilon>0$ there exists $N_0 \in \N$ such that for all $i, j \ge N_0$, 
\begin{equation*}
\|f_i^+ - f_j^+\|_{L^p(w)} < \varepsilon \quad\text{and}\quad 
\|f_i^- - f_j^-\|_{L^p(w)} < \varepsilon, 
\end{equation*}
which implies 
\begin{align*}
\|f_i - f_j\|_{L^p(w)} 
\le \|f_i^+ - f_j^+\|_{L^p(w)} + \|f_i^- - f_j^-\|_{L^p(w)} 
< 2 \varepsilon. 
\end{align*}
Accordingly, $\{f_j\}$ is a Cauchy sequence in $\G \subset L^p(w)$, so $\G$ is relatively compact in $L^p(w)$. 

 Next, we show the necessity. Assume that $\G$ is relatively compact in $L^p(w)$. Since $w \in A_{p_0}$, $w \in L^1_{\loc}(\Rn)$ and $w^{1-p'_0} \in L^1_{\loc}(\Rn)$. Then together with Theorem \ref{thm:FK-1}, this gives (1) and (2) immediately. It remains to show (3). Let $\varepsilon>0$. Since $\G$ is relatively compact, there exists a finite number of functions $\{f_j\}_{j=1}^N \subset \G$ such that for any $g \in \G$, one can find $j \in \{1,\ldots,N\}$ satisfying $\|g-f_j\|_{L^p(w)}<\varepsilon$. 
Fix $f \in \G$. Then there is some $f_j \in \G$ such that 
\begin{align}\label{eq:ffe}
\|f-f_j\|_{L^p(w)} <\varepsilon. 
\end{align}
Observe that 
\begin{align}\label{eq:Ifr}
\mathcal{I}(f, r) &:= \int_{\Rn} \bigg(\fint_{B(0, r)} |f(x)-f(x+y)|^{\frac{p}{p_0}} dy\bigg)^{p_0}w(x)dx
\\ \nonumber
&\lesssim \int_{\Rn} \bigg(\fint_{B(0, r)} |f(x)-f_j(x)|^{\frac{p}{p_0}} dy\bigg)^{p_0}w(x)dx
\\ \nonumber
&\quad+ \int_{\Rn} \bigg(\fint_{B(0, r)} |f_j(x)-f_j(x+y)|^{\frac{p}{p_0}} dy\bigg)^{p_0}w(x)dx
\\ \nonumber
&\quad+ \int_{\Rn} \bigg(\fint_{B(0, r)} |f_j(x+y)-f(x+y)|^{\frac{p}{p_0}} dy\bigg)^{p_0}w(x)dx
\\ \nonumber
&=: \mathcal{I}_1 + \mathcal{I}_2 + \mathcal{I}_3. 
\end{align}
From \eqref{eq:ffe}, one has 
\begin{align}\label{eq:Ifr-1} 
\mathcal{I}_1 = \int_{\Rn} |f(x)-f_j(x)|^{p}w(x)dx < \varepsilon. 
\end{align}
For $\mathcal{I}_3$, we have 
\begin{align}\label{eq:Ifr-3} 
\mathcal{I}_3 \le \int_{\Rn} M(|f-f_j|^{\frac{p}{p_0}})(x)^{p_0} w(x)dx 
\lesssim \int_{\Rn} |f(x)-f_j(x)|^{p}w(x)dx < \varepsilon, 
\end{align}
where we used that $w \in A_{p_0}$ and \eqref{eq:ffe}. To deal with $\mathcal{I}_2$, we see that $w \in L^1_{\loc}(\Rn)$, and hence, $\mathscr{C}_c^{\infty}(\Rn)$ is dense in $L^p(w)$ for any $p \in (0, \infty)$. So, we can find $g_j \in \mathscr{C}_c^{\infty}(\Rn)$ such that 
\begin{equation}\label{eq:fge}
\|f_j-g_j\|_{L^p(w)} < \varepsilon. 
\end{equation}
We may assume that there exist $r_0, A_0>0$ such that $\supp(g_j) \subset B(0, A_0)$ and 
\begin{align}\label{eq:gjgj}
\sup_{|y| \le r_0} \|g_j(\cdot) - g_j(\cdot+y)\|_{L^{\infty}(\Rn)} < \varepsilon. 
\end{align}
Using \eqref{eq:fge}, \eqref{eq:gjgj}, we obtain that for any $0<r<r_0$, 
\begin{align}\label{eq:Ifr-2}
\mathcal{I}_2  &\le \int_{\Rn} \bigg(\fint_{B(0, r)} |f_j(x)-g_j(x)|^{\frac{p}{p_0}} dy\bigg)^{p_0}w(x)dx
\\ \nonumber 
&\qquad+ \int_{\Rn} \bigg(\fint_{B(0, r)} |g_j(x)-g_j(x+y)|^{\frac{p}{p_0}} dy\bigg)^{p_0}w(x)dx
\\ \nonumber
&\qquad+ \int_{\Rn} \bigg(\fint_{B(0, r)} |g_j(x+y)-f_j(x+y)|^{\frac{p}{p_0}} dy\bigg)^{p_0}w(x)dx
\\ \nonumber
&\le  \int_{\Rn} |f_j-g_j|^p w\, dx 
+ \sup_{|y| \le r_0} \|g_j(\cdot) - g_j(\cdot+y)\|_{L^{\infty}(\Rn)}^p w(B(0,A+r))
\\ \nonumber
&\qquad+ \int_{\Rn} M(|g_j-f_j|^{\frac{p}{p_0}})(x)^{p_0} w(x)dx 
\\ \nonumber
&\lesssim \varepsilon^p + \varepsilon^p w(B(0,A+r_0)) + \|f_j-g_j\|_{L^p(w)}^p 
\lesssim \varepsilon^p. 
\end{align}
Collecting \eqref{eq:Ifr}, \eqref{eq:Ifr-1}, \eqref{eq:Ifr-3} and \eqref{eq:Ifr-2}, we conclude that for any $0<r<r_0$, 
\begin{align*}
\mathcal{I}(f, r) \lesssim \varepsilon + \varepsilon^p, 
\end{align*}
where the implicit constant is independent of $f$ and $r$. This proves (3) and completes the proof.  
\end{proof}

The following result will provide us great convenience in practice. 
\begin{lemma}\label{lem:bTTj}
Let $\frac1p=\frac{1}{p_1}+\cdots+\frac{1}{p_m}$ with $1<p_1,\ldots,p_m<\infty$, and fix $k\in \{1,\ldots,m\}$. Assume that an m-linear operator $T$ satisfies the following: 
\begin{enumerate}
\item[(i)]  $\|[b, T]_{e_k}\|_{L^{p_1}(\Rn) \times \cdots \times L^{p_m}(\Rn) \to L^p(\Rn)} \lesssim \|b\|_{\BMO}$ for any $b \in \BMO$; 
\vspace{0.2cm}
\item[(ii)]  $T=\sum_{j \ge 0} T_j$, where $T_j$ is also an m-linear operator such that  \vspace{0.2cm} 
\begin{enumerate}
\item[(ii-1)] $\|T_j\|_{L^{p_1}(\Rn)\times \cdots \times L^{p_m}(\Rn) \to L^p(\Rn)} \lesssim 2^{-\delta j}$ for each $j \ge 0$, where $\delta>0$ is a fixed number. 
\item[(ii-2)] For any $b \in \CMO$, $[b, T_j]_{e_k}$ is compact from $L^{p_1}(\Rn)\times \cdots \times L^{p_m}(\Rn)$ to $L^p(\Rn)$ for each $j \ge 0$. 
\end{enumerate}
\end{enumerate}
Then, $[b, T]_{e_k}$ is compact from $L^{p_1}(\Rn) \times \cdots \times L^{p_m}(\Rn)$ to $L^p(\Rn)$ for any $b \in \CMO$. 
\end{lemma}

\begin{proof}
For any $N,M\in\N$ with $N<M$, by (ii-1), we have
\begin{equation*}
\Big\|\sum_{j \le N}T_j(\vec{f})- \sum_{j \le M}T_j(\vec{f})\Big\|_{L^p(\Rn)}
\le \sum_{N<j \le M}2^{-\delta j}\prod_{j=1}^{m}\|f_j\|_{L^{p_j}(\Rn)}. 
\end{equation*} 
Letting $M\to\infty$, we get
\begin{equation*}
\Big\|T(\vec{f})- \sum_{j\le N}T_j (\vec{f})\Big\|_{L^p(\Rn)}
\le \sum_{j>N}2^{-\delta j}\prod_{j=1}^{m}\|f_{j}\|_{L^{p_j}(\Rn)},
\end{equation*} 
which implies
\begin{equation}\label{eq:TTj}
\Big\|T- \sum_{j\le N}T_j \Big\|_{L^{p_1}(\Rn)\times\cdots\times L^{p_m}(\Rn)\to L^p(\Rn)} \le \sum_{j>N}2^{-\delta j}.
\end{equation}
Now for $b\in C_c^\infty(\Rn)$ and $f_{j}\in L^{p_{j}}(\Rn)$, 
\begin{align*}
&\Big\|[b,T]_{e_k}(\vec{f})- \sum_{j\le N}[b,T_j]_{e_k}(\vec{f})\Big\|_{L^p(\Rn)}
\\
&\le \Big\|b \Big(T- \sum_{j\le N}T_j \Big)(\vec{f})\Big\|_{L^p(\Rn)}
+ \Big\| \Big(T- \sum_{j\le N}T_j \Big)(f_1,\ldots,bf_k,\ldots,f_m) \Big\|_{L^p(\Rn)}\\
&\le 2\|b\|_{L^{\infty}(\Rn)} \Big\|T- \sum_{j\le N}T_j\Big\|_{L^{p_1}(\Rn) \times \cdots \times L^{p_m}(\Rn)\to L^p(\Rn)} \prod_{j=1}^{m}\|f_j\|_{L^{p_j}(\Rn)}\\
&\le 2\|b\|_{L^{\infty}(\Rn)} \sum_{j>N} 2^{-\delta j}\prod_{j=1}^{m}\|f_j\|_{L^{p_j}(\Rn)}, 
\end{align*}
where \eqref{eq:TTj} was used in the last inequality. Hence, for $b\in \mathscr{C}_c^\infty(\Rn)$, 
\begin{equation*}
\Big\|[b,T]_{e_k}- \sum_{j\le N}[b,T_j]_{e_k} \Big\|_{L^{p_1}(\Rn)\times\cdots\times L^{p_m}(\Rn) \to L^p(\Rn)} \to 0,\quad  \text{as }N\to\infty.
\end{equation*}
From (ii-2), we see that $[b,T]_{e_k}$ is compact from $L^{p_1}(\Rn) \times\cdots\times L^{p_m}(\Rn)$ to $L^p(\Rn)$ whenever $b\in \mathscr{C}_c^\infty(\Rn)$. 

Next, let $b \in \CMO$ and take $b_j\in \mathscr{C}_c^\infty(\Rn)$ so that $\lim_{j\to\infty}\|b-b_j\|_{\BMO}=0$. Then using (i), 
\begin{equation}\label{eq:bjlim}
\|[b_j,T]_{e_{k}}-[b,T]_{e_k}\|_{L^{p_1}(\Rn)\times\cdots\times L^{p_m}(\Rn)\to L^p(\Rn)}  \lesssim \|b_j-b\|_{\BMO}.
\end{equation}
Since $[b_j,T]_{e_{k}}$ is compact from $L^{p_1}(\Rn)\times\cdots\times L^{p_m}(\Rn)$ to $L^p(\Rn)$, this and \eqref{eq:bjlim} yield that $[b,T]_{e_{k}}$ is compact from $L^{p_1}(\Rn)\times\cdots\times L^{p_m}(\Rn)$ to $L^p(\Rn)$.
\end{proof}

\section{Interpolation for multilinear operators}\label{sec:inter} 

In this section, we will study the weighted interpolation for multilinear 
operators. We first generalize the results in \cite{CZ, S} 
to the weighted case. 

\begin{theorem}\label{thm:WMIP-1}
Suppose that $(\Sigma_0, \mu_0), \ldots, (\Sigma_m, \mu_m)$ are 
measure spaces, and $\mathscr{S}_j$ is the collection of all simple functions 
on $\Sigma_j$, $j=1,\dots,m$. Denote by $\mathfrak{M}(\Sigma_0)$ the set of 
all measurable functions on $\Sigma_0$. 
Let $T: \mathscr{S}=\mathscr{S}_1 \times \cdots \times \mathscr{S}_m \to 
\mathfrak{M}(\Sigma_0)$ be an $m$-linear operator. 
Let $0<p_0, q_0< \infty$, $1\le p_j, q_j\le \infty$ $(j=1,\dots,m)$, and let 
$w_j, v_j$ be weights on $\Sigma_j$ $(j=0,\dots,m)$. 
Assume that there exist $M_1, M_2 \in (0, \infty)$ such that 
\begin{align}
&\|T\|_{L^{p_1}(\Sigma_1,\, w_1^{p_1}) \times \cdots 
\times L^{p_m}(\Sigma_m,\, w_m^{p_m}) \to L^{p_0}(\Sigma_0,\,  w_0^{p_0})} 
\le M_1,  
\\
&\|T\|_{L^{q_1}(\Sigma_1,\, v_1^{q_1}) \times \cdots \times 
L^{q_m}(\Sigma_m,\,  v_m^{q_m}) \to L^{q_0}(\Sigma_0,\,  v_0^{q_0})} \le M_2,  
\end{align}
Then, we have 
\begin{equation}
\|T\|_{L^{r_1}(\Sigma_1,\,  u_1^{r_1}) \times \cdots \times 
L^{r_m}(\Sigma_m,\, u_m^{r_m}) \to L^{r_0}(\Sigma_0,\,  u_0^{r_0})} 
\le M_1^{1-\theta} M_2^{\theta},   
\end{equation}
for all exponents satisfying 
\begin{equation}\label{eq:exp}
0<\theta<1,\quad \frac{1}{r_j}=\frac{1-\theta}{p_j}+\frac{\theta}{q_j} 
\quad\text{and}\quad u_j=w_j^{1-\theta} v_j^{\theta},\quad j=0,\dots,m.
\end{equation}
\end{theorem}

Obviously, Theorem \ref{thm:WMIP-1} is a consequence of Lemma \ref{lem:WMIP} and Lemma \ref{lem:dense} below. 
\begin{lemma}\label{lem:WMIP}
Suppose that $(\Sigma_0, \mu_0), \ldots, (\Sigma_m, \mu_m)$ are measure spaces, and $\mathscr{S}_j$ is the collection of all simple functions on $\Sigma_j$, $j=1,\dots,m$. Denote by $\mathfrak{M}(\Sigma_0)$ the set of all measurable functions on $\Sigma_0$. Let $T: \mathscr{S}=\mathscr{S}_1 \times \cdots \times \mathscr{S}_m \to \mathfrak{M}(\Sigma_0)$ be an $m$-linear operator. Let $0<p_0, q_0< \infty$, $1\le p_j, q_j\le \infty$ $(j=1,\dots,m)$, and let $w_j, v_j$ be weights on $\Sigma_j$ $(j=0,\dots,m)$. Assume that there exist $M_1, M_2 \in (0, \infty)$ such that 
\begin{align}\label{eq:WMIP-1} 
\|T(\vec{f})\, w_0\|_{L^{p_0}(\Sigma_0,\, \mu_0)} \le M_1\prod_{j=1}^{m}\|f_j\, w_j\|_{L^{p_j}(\Sigma_j,\, \mu_j)}, 
\end{align}
for all $\vec{f}=(f_1,\ldots,f_m) \in \mathscr{S}$ with $\|f_j\,w_j\|_{L^{p_j}(\Sigma_j,\, \mu_j)}<\infty$, $j=1,\ldots,m$, and 
\begin{align}\label{eq:WMIP-2} 
\|T(\vec{f})\, v_0\|_{L^{q_0}(\Sigma_0,\, \mu_0)} \le M_2\prod_{j=1}^{m}\|f_j\, v_j\|_{L^{q_j}(\Sigma_j,\, \mu_j)},  
\end{align}
for all $\vec{f}=(f_1,\ldots,f_m) \in \mathscr{S}$ with $\|f_j\,v_j\|_{L^{q_j}(\Sigma_j,\, \mu_j)}<\infty$, $j=1,\ldots,m$. Then, for all exponents satisfying \eqref{eq:exp},  
\begin{equation}\label{eq:WMIP-3}
\|T(\vec{f})\, u_0\|_{L^{r_0}(\Sigma_0,\, \mu_0)} \le M_1^{1-\theta} M_2^{\theta} \prod_{j=1}^{m}\|f_j\, u_j\|_{L^{r_j}(\Sigma_j,\, \mu_j)}, 
\end{equation}
for all $\vec{f}=(f_1,\ldots,f_m) \in \mathscr{S}$ with $\|f_j\,w_j\|_{L^{p_j}(\Sigma_j,\, \mu_j)}<\infty$ and $\|f_j\,v_j\|_{L^{q_j}(\Sigma_j,\, \mu_j)}<\infty$, $j=1,\ldots,m$. 
\end{lemma}

\begin{proof}
We begin with a claim that given $\mu_j$-measurable sets $F_j \subset \Sigma_j$ with $\mu_j(F_j)<\infty$, $j=1,\ldots,m$, under the assumptions in Lemma \ref{lem:WMIP}, for any fixed $\varepsilon>0$ and simple functions $w'_j, v'_j, u'_j$ on $\Sigma_j$ $(j=0,\ldots,m)$ satisfying $w_j \le w'_j$, $v_j \le v'_j$ on the set $F'_j:=\{x \in F_j: \varepsilon \le w_j(x), v_j(x) \le 1/\varepsilon\}$, $w'_j(x)=v'_j(x)=0$ on $\Sigma_j \setminus F'_j$ $(j=1,\ldots,m)$, $w'_0 \le w_0$, $v'_0 \le v_0$, and $u'_j=(w'_j)^{1-\theta}(v'_j)^{\theta}$ $(j=0,\ldots,m)$, it holds 
\begin{equation}\label{eq:Tfu-1}
\|T(\vec{f})\, u'_0\|_{L^{r_0}(\Sigma_0,\, \mu_0)}
\le M_1^{1-\theta} M_2^{\theta}\prod_{j=1}^{m}\|f_j\, u'_j\|_{L^{r_j}(\Sigma_j,\, \mu_j)}, 
\end{equation}
for any simple functions $f_j$ with $f_j=0$ on $\Sigma_j \setminus F'_j$, $j=1,\ldots,m$. 

We momentarily assume \eqref{eq:Tfu-1} holds. Letting $w'_j \to w_j$ and $v'_j \to v_j$ on $F'_j$ $(j=1,\ldots,m)$, and by Lebesgue's dominated convergence theorem, we obtain from \eqref{eq:Tfu-1} that  
\begin{equation}\label{eq:Tfu-2}
\|T(\vec{f})\, u'_0\|_{L^{r_0}(\Sigma_0,\, \mu_0)}
\le M_1^{1-\theta} M_2^{\theta}\prod_{j=1}^{m}\|f_j\, u_j\|_{L^{r_j}(\Sigma_j,\, \mu_j)}, 
\end{equation}
for any simple functions $f_j$ with $f_j=0$ on $\Sigma_j \setminus F'_j$, $j=1,\ldots,m$. Then using \eqref{eq:Tfu-2}, letting $w'_0 \to w_0$ and $v'_0 \to v_0$ increasingly, and by Fatou's lemma, we get 
\begin{equation}\label{eq:Tfu-3}
\|T(\vec{f})\, u_0\|_{L^{r_0}(\Sigma_0,\, \mu_0)}
\le M_1^{1-\theta} M_2^{\theta}\prod_{j=1}^{m}\|f_j\, u_j\|_{L^{r_j}(\Sigma_j,\, \mu_j)}, 
\end{equation}
for any simple functions $f_j$ with $f_j=0$ on $\Sigma_j \setminus F'_j$, $j=1,\ldots,m$. 

We are going to conclude \eqref{eq:WMIP-3} by means of \eqref{eq:Tfu-3}. Let $f_j$ be a simple function on $\Sigma_j$ satisfying $f_j w_j \in L^{p_j}(\Sigma_j, \mu_j)$ and $f_j v_j \in L^{q_j}(\Sigma_j, \mu_j)$, $j=1,\ldots,m$. Then there are measurable sets $F_j \subset \Sigma_j$ with $\mu_j(F_j)<\infty$ such that $f_j=0$ on $\Sigma_j \setminus F_j$, $j=1,\ldots,m$. Note that H\"{o}lder's inequality gives that 
\begin{align*}
\|f_ju_j\|_{L^{r_j}(\Sigma_j, \mu_j)} \le \|f_jw_j\|_{L^{p_j}(\Sigma_j,\, \mu_j)}^{1-\theta} \|f_jv_j\|_{L^{q_j}(\Sigma_j,\,  \mu_j)}^{\theta}. 
\end{align*}
Denote $F_{j,k} :=\{x \in F_j: 1/k \le w_j(x), v_j(x) \le k\}$ and $f_{j,k}=f_j {\bf 1}_{F_{j,k}}$, $j=1\ldots,m$. Then $f_{j,k}$ is a simple function in $\Sigma_j$ and $f_{j,k}=0$ on $\Sigma_j \setminus F_{j,k}$. By Lebesgue's dominated convergence theorem, we see that $f_{j,k} \to f_j$ in $L^{p_j}(\Sigma_j, w_j^{p_j})$, $L^{q_j}(\Sigma_j, v_j^{q_j})$ and $L^{r_j}(\Sigma_j, u_j^{r_j})$ for each $j=1,\ldots,m$. Hence, \eqref{eq:WMIP-1} gives that $T(f_{1,k}, \ldots, f_{m,k}) w_0$ tends to $T(\vec{f}) w_0$ in $L^{p_0}(\Sigma_0, \mu_0)$. On the other hand, from \eqref{eq:Tfu-3}, we see that $\{T(f_{1,k},\ldots,f_{m,k}) u_0\}_{k \ge 1}$ is a Cauchy sequence in $L^{r_0}(\Sigma_0, \mu_0)$. These two facts yield that $T(f_{1,k},\ldots,f_{m,k}) u_0$ tends to $T(\vec{f}) u_0$ in $L^{r_0}(\Sigma_0, \mu_0)$, which implies 
\begin{equation*}
\|T(\vec{f})\, u_0\|_{L^{r_0}(\Sigma_0,\, \mu_0)}
\le M_1^{1-\theta} M_2^{\theta}\prod_{j=1}^{m}\|f_j\, u_j\|_{L^{r_j}(\Sigma_j,\, \mu_j)}. 
\end{equation*}
This coincides with \eqref{eq:WMIP-3}. 

Now, we proceed to demonstrate \eqref{eq:Tfu-1}. For the sake of simplicity, we use $w_j, v_j$ and $u_j$ instead of $w_j', v_j'$ and $u_j'$, respectively. Pick $k\in\N$ so that $k>\max\{\frac{1}{p_0}, \frac{1}{q_0}\}$, which gives that $kr_0>1$. Hence we have
\begin{align}\label{eq:WMIP-5} 
\|T(\vec{f})\, u_0\|^{1/k}_{L^{r_0}(\Sigma_0,\, \mu_0)} 
= \sup_{g}\int_{\Sigma_0} |T(\vec{f})u_0|^{1/k} g\, d\mu_0,
\end{align}
where $g$ is nonnegative simple functions on $\Sigma_0$ satisfying $\|g\|_{L^{(kr_0)'}(\Sigma_0, \mu_0)}=1$. Let us fix $\vec{f}=(f_1,\ldots,f_m)$ and $g$. We may assume $\|f_j\,u_j\|_{L^{r_j}(\Sigma_j, \mu_j)}<\infty$ for each $j=1,\ldots,m$. Write $\widetilde{f}_j=f_j u_j$ and $\widetilde{f}_j=|\widetilde{f}_j|e^{is_j}$, $j=1,\ldots,m$. Set 
\begin{equation*}
A_1 :=\prod_{j=1}^{m}\|\widetilde{f}_j\|_{L^{r_j}(\Sigma_j,\, \mu_j)}^{r_j/p_j} \quad \text{and}\quad 
A_2 :=\prod_{j=1}^{m}\|\widetilde{f}_j\|_{L^{r_j}(\Sigma_j,\, \mu_j)}^{r_j/q_j}
\end{equation*} 
Define  for $\ell\in \N$ 
\begin{equation}\label{eq:WMIP-8}
\Phi_\ell(z) :=\int_{\Sigma_0} |U_\ell(z)|^{\frac1k}\, d\mu_0,
\end{equation}
where  
\begin{align*}
& U_\ell(z) := e^{k(z^2-1)/\ell}
(A_1M_1)^{z-1}(A_2M_2)^{-z}T(\vec{F_z})w^{1-z}_0v_0^{z} G_z^k, 
\\
& G_z := g^{\frac{1-1/(kr_0(z))}{1-1/(kr_0)}}, \quad 
\frac{1}{r_j(z)} := \frac{1-z}{p_j} + \frac{z}{q_j}, \, j=0, 1, \dots, m. 
\\
&F_{z,j} :=|\widetilde{f}_j|^{\frac{r_j}{r_j(z)}} e^{is_j}w^{z-1}_jv_j^{-z}, \, j=1,\dots,m, 
\end{align*}
We claim that 
\begin{equation}\label{eq:Phi-ell}
\Phi_\ell(z) \le 1, \quad\forall \ell \in \N,  
\end{equation}
which in turn implies 
\begin{align*}
\Phi_\infty(z) := \lim_{\ell\to\infty}\Phi_\ell(z) \le 1, 
\end{align*} 
and further, 
\begin{equation*}
\|T(\widetilde{f}_1u_1^{-1},\dots, \widetilde{f}_mu_m^{-1})\, u_0\|_{L^{r_0}(\Sigma_0,\, \mu_0)}
\le M_1^{1-\theta} M_2^{\theta}\prod_{j=1}^{m}\|\widetilde{f}_j\|_{L^{r_j}(\Sigma_j,\, \mu_j)}. 
\end{equation*}
The latter is equivalent to \eqref{eq:Tfu-1}.

It remains to show \eqref{eq:Phi-ell}. We easily see that $U_\ell(z)$ is holomorphic in the strip $S:=\{z\in \C: 0<\Re z<1\}$ and hence $|U_\ell(z)|^{1/k}$ is subharmonic in $S$. It is continuous on $\overline{S}$. For any circle $\{z\in \C: |z-z_0|<r\}$ in $S$, we have
\begin{align}\label{eq:sub}
\frac{1}{2\pi r}\int_{0}^{2\pi}\Phi_\ell(re^{i t}-z_0)dt
&=\int_{\Sigma_0}\frac{1}{2\pi r}
\int_{0}^{2\pi}|U_\ell(re^{i t}-z_0)|^{\frac1k}\, dt\, d\mu_0
\\ \nonumber
&\ge \int_{\Sigma_0}|U_\ell(re^{z_0})|^{\frac1k}\, d\mu_0=\Phi_\ell(z_0),
\end{align}
and so $\Phi_\ell(z)$ is subharmonic in $S$. We also see that $\Phi_\ell(z)$ is continuous on $\overline{S}$. Next, we would like to get that it is bounded on $\overline{S}$. Fix $z \in S$. If we write $h_j:=e^{is_j} w_j^{z-1}v_j^z$, $j=1,\ldots,m$, then 
\begin{align*}
|T(\vec{F_z}) w^{1-z}_0 v_0^z G_z^k|^{\frac1k} 
\lesssim \sum_{l_0, l_1, \ldots, l_m} 
|T(h_1 {\bf 1}_{I_{1, l_1}},\dots, h_m {\bf 1}_{I_{m, l_m}}) |^{\frac1k} 
{\bf 1}_{I_{0, l_0}}.
\end{align*}
Therefore, together with H\"older's inequality and \eqref{eq:WMIP-1}, 
\begin{align*}
\Phi_\ell(z) & \lesssim e^{-|\Im z|^2/\ell}\sum_{l_0, \ldots, l_m}
\|T(h_1 {\bf 1}_{I_{1, l_1}},\dots, h_m 
{\bf 1}_{I_{m, l_m}}) w_0\|_{L^{p_0}(\Sigma_0,\, \mu_0)}^{\frac1k}\,  
\mu_0(I_{0, l_0})^{\frac{1}{(k p_0)'}}
\\
&\lesssim e^{-|\Im z|^2/\ell}\sum_{l_0,\ldots, l_m}  \prod_{j=1}^{m} 
\|h_j {\bf 1}_{I_{j, l_j}}w_j\|_{L^{p_j}(\Sigma_j,\, \mu_j)}^{\frac1k}
\mu_0 (I_{0, l_0})^{\frac{1}{(kp_0)'}} 
\\
&\lesssim e^{-|\Im z|^2/\ell}\sum_{l_0,\ldots,l_m}
\prod_{j=1}^{m} \mu_j(I_{j, l_j})^{\frac{1}{kp_j}}
\mu_0 (I_{0, l_0})^{\frac{1}{(kp_0)'}}\lesssim e^{-|\Im z|^2/\ell}<\infty,
\end{align*}
which shows that $\Phi_\ell(z)$ is bounded on $\overline{S}$. Also, for each $\ell \in \N$, 
\begin{equation}\label{eq:WMIP-8-1}
\lim_{|\Im z|\to\infty} \Phi_\ell(z) = 0 \quad \text{uniformly for }0\le \Re z\le 1.
\end{equation}

Let us consider $z=x+iy$ with $\Re(z)=0$. Then, $\Re(r_j(z))=p_j$ for each $j=0,\dots,m$. Note that 
\begin{equation}\label{eq:Giy}
\|G_{iy}\|_{L^{(kp_0)'}(\Sigma_0,\, \mu_0)}
=\|g^{\frac{1-1/(kp_0)}{1-1/(kr_0)}}\|_{L^{(kp_0)'}(\Sigma_0,\, \mu_0)}
=\|g\|_{L^{(kr_0)'}(\Sigma_0,\, \mu_0)}^{\frac{(kr_0)'}{(kp_0)'}}
=1.
\end{equation}
Thus, by H\"older inequality, \eqref{eq:WMIP-1} and \eqref{eq:Giy}, we obtain 
\begin{align}\label{eq:Phi-1}
|\Phi_\ell(iy)| 
& \le e^{-|\Im z|^2/\ell}(A_1M_1)^{-1/k}
\|T(\vec{F}_{iy})w^{1-iy}_0 v_0^{iy}\|_{L^{p_0}(\Sigma_0,\, \mu_0)}^{1/k} 
\|G_{iy}\|_{L^{(kp_0)'}(\Sigma_0,\, \mu_0)}
\nonumber\\
&\le e^{-|\Im z|^2/\ell}(A_1M_1)^{-1/k}
\|T(\vec{F}_{iy})w_0\|_{L^{p_0}(\Sigma_0,\, \mu_0)}^{1/k} \|G_{iy}\|_{L^{(kp_0)'}(\Sigma_0,\, \mu_0)}
\nonumber\\
&\le e^{-|\Im z|^2/\ell}(A_1M_1)^{-1/k}
M_1^{1/k} \prod_{j=1}^{m} \|F_{iy,j} w_j\|^{1/k}_{L^{p_j}(\Sigma_j,\, \mu_j)} 
\nonumber\\
&=e^{-|\Im z|^2/\ell}A_1^{-1/k} 
\prod_{j=1}^{m} \|\,|\widetilde{f}_{j}|^{r_j/p_j}|\,\|^{1/k}_{L^{p_j}(\Sigma_j,\, \mu_j)} 
\nonumber\\
&=e^{-|\Im z|^2/\ell}A_1^{-1/k} 
\prod_{j=1}^{m} \|\widetilde{f}_{j}\|_{L^{r_j}(\Sigma_j,\, \mu_j)}^{r_j/(kp_j)}
\le 1.  
\end{align}

Next, we treat the case $\Re(z)=1$. In this case, we have $\Re(r_j(z))=q_j$ for each $j=0,\ldots,m$. Since 
\begin{equation*}
\|G_{1+iy}\|_{L^{(kq_0)'}(\Sigma_0,\, \mu_0)} 
=\|g^{\frac{1-1/(kq_0)}{1-1/(kr_0)}}\|_{L^{(kq_0)'}(\Sigma_0,\, \mu_0)}
=\|g\|_{L^{(kr_0)'}(\Sigma_0,\, \mu_0)}^{\frac{(kr_0)'}{(kq_0)'}}
=1,
\end{equation*}
H\"older inequality and \eqref{eq:WMIP-2} imply 
\begin{align}\label{eq:Phi-2}
\Phi_\ell(1+iy) & \le e^{-|\Im z|^2/\ell}(A_2M_2)^{-1/k}
\|T(\vec{F}_{1+iy}) w^{-iy}_0 v_0^{1+iy}\|_{L^{q_0}(\Sigma_0,\, \mu_0)}^{1/k} 
\|G_{iy}\|_{L^{(kq_0)'}(\Sigma_0,\, \mu_0)}
\nonumber\\
&\le e^{-|\Im z|^2/\ell}(A_2M_2)^{-1/k}
\|T(\vec{F}_{1+iy})v_0\|_{L^{q_0}(\Sigma_0,\, \mu_0)}^{1/k} 
\|G_{iy}\|_{L^{(kq_0)'}(\Sigma_0,\, \mu_0)}
\nonumber\\
&\le e^{-|\Im z|^2/\ell}A_2^{-1/k}
\prod_{j=1}^{m} \|F_{1+iy, j} v_j\|_{L^{q_j}(\Sigma_j,\, \mu_j)}^{1/k}
\nonumber\\
&=e^{-|\Im z|^2/\ell}A_2^{-1/k}
\prod_{j=1}^{m} \|\,|\widetilde{f}_j|^{r_j/q_j}\,\|^{1/k}_{L^{q_j}(\Sigma_j,\, \mu_j)}
\nonumber\\
&=e^{-|\Im z|^2/\ell}A_2^{-1/k}
\prod_{j=1}^{m} \|\widetilde{f}_j\|_{L^{r_j}(\Sigma_j,\, \mu_j)}^{r_j/(kq_j)}\le 1. 
\end{align}
Consequently, \eqref{eq:Phi-ell} follows from \eqref{eq:WMIP-8-1}, \eqref{eq:Phi-1}, \eqref{eq:Phi-2}, and the subharmonicity of $\Phi_\ell(z)$. This completes the proof of Lemma \ref{lem:WMIP}. 
\end{proof}

\begin{lemma}\label{lem:dense}
Let $w$ and $v$ be weights on $(\Sigma, \mu)$, and let $1\le p, q<\infty$. Denote 
\[
\mathfrak{S}_{p,q} :=\big\{\text{simple functions } a \in L^p(\Sigma, w^p) \cap L^q(\Sigma, v^q) \big\}. 
\]
Then 
\begin{align}\label{eq:Spq}
\mathfrak{S}_{p,q} \text{ is dense in } L^r(\Sigma, u^r),
\end{align}
whenever $\theta \in (0, 1)$, $u=w^{1-\theta} v^{\theta}$ and $\frac{1}{r}=\frac{1-\theta}{p}+\frac{\theta}{q}$. 
\end{lemma}

\begin{proof}
We first deal with a particular case: for any weight $\sigma$ on $(\Sigma, \mu)$ and for any $1 \le s<\infty$, 
\begin{align}\label{eq:SSr}
\mathfrak{S}_s := \big\{\text{simple functions } a \in L^s(\Sigma, \sigma^s) \big\} \text{ is dense in } L^s(\Sigma, \sigma^s),
\end{align}
Indeed, for $f \in L^s(\Sigma, \sigma^s)$, we assume that $f \ge 0$ $\mu$-a.e..  Let $\varepsilon>0$. Then there exists a simple function $a(x)=\sum_{i=1}^{\ell_0} a_i {\bf 1}_{E_i}(x)$ such that $a\le f\sigma$ and $\|f\sigma-a\|_{L^s(\Sigma, \mu)}<\varepsilon/2^{1/s}$, where $a_i>0$, $\{E_i\}_{i=1}^{\ell_0}$ is a disjoint family and $0<\mu(E_i)<\infty$. Set $E=\bigcup_{i=1}^{\ell_0} E_i$. Observe that 
\[
\varepsilon^s/2 > \|f\sigma-a\|_{L^r(\Sigma,\, \mu)}^s 
= \int_{E} |f\sigma-a|^s\, d\mu + \int_{\Sigma \setminus E} |f\sigma|^s\, d\mu, 
\]
and hence, 
\begin{equation}\label{eq:fuE}
\|f\sigma\|_{L^s(\Sigma \setminus E,\, \mu)}^s < \varepsilon^s/2. 
\end{equation}
On the other hand, there exist simple functions $b_j(x)=\sum_{i=1}^{\ell_j}b_{j,i}{\bf 1}_{F_{j,i}}(x)$ such that $\supp(b_j) \subset E$ and $\lim\limits_{j \to \infty}b_j(x)=f(x)$ for all $x \in E$. Then 
\[
\lim_{j\to \infty} \|(f-b_j)\sigma\|_{L^s(E,\, \mu)}=0, 
\]
which implies that there exists $j_0 \in \N$ so that 
\begin{align}\label{eq:fbu}
\|(f-b_{j_0})\sigma\|_{L^s(E,\, \mu)}<\varepsilon^s/2.  
\end{align}
Therefore, it follows from that 
\begin{equation*}
\|f-b_{j_0}\|_{L^s(\Sigma,\, \sigma^s)}^s=\int_{\Sigma \setminus E} |f\sigma|^s\, d\mu + \int_{E}|(f-b_{j_0})\sigma|^s\, d\mu
< \varepsilon^s/2 + \varepsilon^s/2=\varepsilon^s. 
\end{equation*}
This shows \eqref{eq:SSr}. 

We next turn to the proof of \eqref{eq:Spq}. By \eqref{eq:SSr}, it suffices 
to show that for any $E \subset \Sigma$ with $\mu(E)<\infty$ and $u \in 
L^r(E, \mu)$, and for any $\varepsilon>0$, there exists a simple function $a$ 
such that 
\begin{equation}\label{eq:Eaa}
a \in L^p(\Sigma, w^p) \cap L^q(\Sigma, v^q) \quad\text{and}\quad 
\|\mathbf{1}_E-a\|_{L^r(\Sigma,\, u^r)}<\varepsilon. 
\end{equation}
Let $\varepsilon>0$. Since $u \in L^r(E, \mu)$, there exists $\delta>0$ 
such that 
\begin{equation}\label{eq:EF}
\forall F \subset E: \mu(F)<\delta \quad\Longrightarrow\quad 
\|u\|_{L^r(F,\, \mu)}<\varepsilon. 
\end{equation}
Note that $0<w<\infty$ $\mu$-a.e. and $\mu(E)<\infty$. Then there exists 
$K_1>0$ such that $\mu(\{x \in E: w(x)^p>K_1\})<\delta/2$. Similarly, there 
exists $K_2>0$ such that $\mu(\{x \in E: v(x)^q>K_2\})<\delta/2$. Set 
\[
F_0:=\{x \in E: w(x)^p>K_1\} \cup \{x \in E: v(x)^q>K_2\}.
\] 
Then $\mu(F_0)<\delta$ and $\|u\|_{L^r(F_0, \mu)}<\varepsilon$ by 
\eqref{eq:EF}. By definition, we have $w \in L^p(E\setminus F_0, \mu)$ and 
$v \in L^q(E\setminus F_0, \mu)$. 
Picking $a(x)=\mathbf{1}_{E\setminus F_0}(x)$, we see that 
$a \in L^p(\Sigma, w^p) \cap L^q(\Sigma, v^q)$ and 
\[
\|{\bf 1}_E -a\|_{L^r(\Sigma,\, u^r)}=\|{\bf 1}_{F_0}\|_{L^r(\Sigma,\, u^r)}
=\|u\|_{L^r(F_0,\, \mu)}<\varepsilon. 
\]
This proves \eqref{eq:Eaa} and completes the proof.  
\end{proof}

With Theorem \ref{thm:WMIP-1} in hand, we will establish the interpolation for multilinear compact operators. 

\begin{theorem}\label{thm:WMIP-2}
Suppose that $(\Sigma_1, \mu_1), \ldots, (\Sigma_m, \mu_m)$ are measure spaces, and $\mathscr{S}_j$ is the collection of all simple functions on $\Sigma_j$, $j=1,\dots,m$. Denote by $\mathfrak{M}(\Rn)$ the set of all measurable functions on $\Rn$. Let $T: \mathscr{S}=\mathscr{S}_1 \times \cdots \times \mathscr{S}_m \to \mathfrak{M}(\Rn)$ be an $m$-linear operator. Let $0<p_0, q_0< \infty$ and $1\le p_j, q_j\le \infty$ $(j=1,\dots,m)$. Assume that 
\begin{align}
\label{eq:WMIP-21} &T \text{ is bounded from $L^{p_1}(\Sigma_1) \times \cdots \times L^{p_m}(\Sigma_m)$ to $L^{p_0}(\Rn)$},  
\\ 
\label{eq:WMIP-22} &T \text{ is compact from $L^{q_1}(\Sigma_1) \times \cdots \times L^{q_m}(\Sigma_m)$ to $L^{q_0}(\Rn)$}.  
\end{align} 
Then, $T$ is also a compact operator from $L^{r_1}(\Sigma_1) \times \cdots \times L^{r_m}(\Sigma_m)$ to $L^{r_0}(\Rn)$ for all exponents satisfying 
\begin{equation*}
0<\theta<1 \quad\text{and}\quad \frac{1}{r_j}=\frac{1-\theta}{p_j}+\frac{\theta}{q_j},\quad j=0,\dots,m. 
\end{equation*}
\end{theorem}

\begin{proof}
It follows from \eqref{eq:WMIP-21} that there exists $M_1<\infty$ such that 
\begin{align}\label{eq:WMIP-21'}
\|T(\vec{f})\|_{L^{p_0}(\Rn)} \le M_1 \prod_{j=1}^m \|f_j\|_{L^{p_j}(\Sigma_j)}. 
\end{align}
From \eqref{eq:WMIP-22} and Theorem \ref{thm:FK-1}, we have the following: 
\begin{align}
\label{eq:th-1} &\|T(\vec{f})\|_{L^{q_0}(\Rn)} \le M_2 \prod_{j=1}^m \|f_j\|_{L^{q_j}(\Sigma_j)}, 
\\
\label{eq:th-2} & \lim_{A \to \infty} \|T(\vec{f}) {\bf 1}_{\{|x|>A\}}\|_{L^{q_0}(\Rn)} \bigg/ \prod_{j=1}^m \|f_j\|_{L^{q_j}(\Sigma_j)} =0, 
\\ 
\label{eq:th-3} & \lim_{|h| \to 0} \|\tau_h(T\vec{f}) - T(\vec{f})\|_{L^{q_0}(\Rn)} \bigg/ \prod_{j=1}^m \|f_j\|_{L^{q_j}(\Sigma_j)} =0. 
\end{align}
By \eqref{eq:WMIP-21'} and \eqref{eq:th-1}, Theorem \ref{thm:WMIP-1} yields that 
\begin{align}\label{eq:th-4}
\|T(\vec{f})\|_{L^{r_0}(\Rn)} \le M_1^{1-\theta} M_2^{\theta} \prod_{j=1}^m \|f_j\|_{L^{r_j}(\Sigma_j)}. 
\end{align}
Additionally, it follows from \eqref{eq:th-2} that for any $\varepsilon>0$, there exists $A_{\varepsilon}>0$ such that for all $A>A_{\varepsilon}$, 
\begin{equation}\label{eq:th-5}
\|T(\vec{f}) {\bf 1}_{\{|x|>A\}}\|_{L^{q_0}(\Rn)} <\varepsilon \prod_{j=1}^m \|f_j\|_{L^{q_j}(\Sigma_j)}. 
\end{equation}
Then, \eqref{eq:WMIP-21'}, \eqref{eq:th-5} and Theorem \ref{thm:WMIP-1} applied to $T(\vec{f}) {\bf 1}_{\{|x|>A\}}$ imply that 
\begin{equation*}
\|T(\vec{f}) {\bf 1}_{\{|x|>A\}}\|_{L^{r_0}(\Rn)} 
< M_1^{1-\theta} \varepsilon^{\theta} \prod_{j=1}^m \|f_j\|_{L^{r_j}(\Sigma_j)}, 
\end{equation*}
which gives that 
\begin{equation}\label{eq:th-6}
\lim_{A\to \infty} \|T(\vec{f}) {\bf 1}_{\{|x|>A\}}\|_{L^{r_0}(\Rn)}=0 
\end{equation}
uniformly for all $\vec{f}$ such that $f_j \in L^{r_j}(\Sigma_j)$ with $\|f_j\|_{L^{r_j}(\Sigma_j)} \le 1$, $j=1,\ldots,m$. On the other hand, by \eqref{eq:WMIP-21'}
\begin{align}\label{eq:th-7}
\|\tau_h(T\vec{f}) -T(\vec{f})\|_{L^{p_0}(\Rn)} \le 2M_1 \prod_{j=1}^m \|f_j\|_{L^{p_j}(\Sigma_j)}. 
\end{align}
The equation \eqref{eq:th-3} gives that for any $\varepsilon>0$, there exists $\eta>0$ such that for all $|h|<\eta$, 
\begin{align}\label{eq:th-8}
\|\tau_h(T\vec{f})-T(\vec{f})\|_{L^{q_0}(\Rn)} \le \varepsilon \prod_{j=1}^m \|f_j\|_{L^{q_j}(\Sigma_j)}. 
\end{align}
Since $\tau_hT-T$ is also an $m$-linear operator, \eqref{eq:th-7}, \eqref{eq:th-8} and Theorem \ref{thm:WMIP-1} applied to the operator $\tau_h T -T$ lead that for all $|h|<\eta$, 
\begin{align*}
\|\tau_h(T\vec{f})-T(\vec{f})\|_{L^{r_0}(\Rn)} 
\le (2M_1)^{1-\theta} \varepsilon^{\theta} \prod_{j=1}^m \|f_j\|_{L^{r_j}(\Sigma_j)}. 
\end{align*}
This means that 
\begin{align}\label{eq:th-9}
\lim_{|h| \to 0} \|\tau_h(T\vec{f})-T(\vec{f})\|_{L^{r_0}(\Rn)} =0, 
\end{align}
uniformly for all $\vec{f}$ such that $f_j \in L^{r_j}(\Sigma_j)$ with $\|f_j\|_{L^{r_j}(\Sigma_j)} \le 1$, $j=1,\ldots,m$. Now gathering \eqref{eq:th-4}, \eqref{eq:th-6} and \eqref{eq:th-9}, we by Theorem \ref{thm:FK-1} conclude that $T$ is a compact operator from $L^{r_1}(\Sigma_1) \times \cdots \times L^{r_m}(\Sigma_m)$ to $L^{r_0}(\Rn)$. 
\end{proof}

Next, we are going to establish the weighted version of Theorem \ref{thm:WMIP-2}. Unfortunately, the approach used above is invalid in the weighted setting. To overcome this difficulty, we present a variation of Theorem \ref{thm:WMIP-1}. 

\begin{theorem}\label{thm:WMIP-3}
Suppose that $(\widetilde{\Sigma}_0, \widetilde{\mu}_0)$, $(\Sigma_0, \mu_0)$, $(\Sigma_1, \mu_1), \ldots, (\Sigma_m, \mu_m)$ are measure spaces, and $\mathscr{S}_j$ is the collection of all simple functions on $\Sigma_j$, $j=1,\dots,m$. Denote by $\mathfrak{M}(\widetilde{\Sigma}_0 \times \Sigma_0)$ the set of all measurable functions on $\widetilde{\Sigma}_0 \times \Sigma_0$. Let $T: \mathscr{S}=\mathscr{S}_1 \times \cdots \times \mathscr{S}_m \to \mathfrak{M}(\widetilde{\Sigma}_0 \times \Sigma_0)$ be an $m$-linear operator. Let $0<\widetilde{p}_0, \widetilde{q}_0, p_0, q_0 < \infty$, $1\le p_j, q_j\le \infty$ $(j=1,\dots,m)$, and let $w_j, v_j$ be weights on $\Sigma_j$, $(j=1,\dots,m)$, and $w_0, v_0$ be weights on $\Sigma_0$. Assume that there exist $M_1, M_2 \in (0, \infty)$ such that 
\begin{align}\label{eq:WMIP-31} 
\bigg[\int_{\Sigma_0} \bigg(\int_{\widetilde{\Sigma}_0} 
|T(\vec{f})(x,y)|^{\widetilde{p}_0}\, d\widetilde{\mu}_0(y) \bigg)^{\frac{p_0}
{\widetilde{p}_0}} w_0(x)^{p_0} d\mu_0(x) \bigg]^{\frac{1}{p_0}} 
\le M_1 \prod_{j=1}^m \|f_j\|_{L^{p_j}(\Sigma_j, w_j^{p_j})}  
\end{align}
for all $\vec{f}=(f_1,\ldots,f_m) \in \mathscr{S}$ with $\|f_j\,w_j\|_{L^{p_j}(\Sigma_j,\, \mu_j)}<\infty$, $j=1,\ldots,m$, and 
\begin{align}\label{eq:WMIP-32} 
\bigg[\int_{\Sigma_0} \bigg(\int_{\widetilde{\Sigma}_0} |T(\vec{f})(x,y)|^{\widetilde{q}_0}\, d\widetilde{\mu}_0(y)  \bigg)^{\frac{q_0}{\widetilde{q}_0}} v_0(x)^{q_0} d\mu_0(x) \bigg]^{\frac{1}{q_0}} 
 \le M_2 \prod_{j=1}^m \|f_j\|_{L^{p_j}(\Sigma_j, v_j^{q_j})} 
\end{align}
for all $\vec{f}=(f_1,\ldots,f_m) \in \mathscr{S}$ with $\|f_j\,v_j\|_{L^{q_j}(\Sigma_j,\, \mu_j)}<\infty$, $j=1,\ldots,m$. Then, 
for all exponents satisfying $0<\theta<1$, and 
\begin{equation}\label{eq:exp-3}
\frac{1}{\widetilde{r}_0}=\frac{1-\theta}{\widetilde{p}_0}+\frac{\theta}{\widetilde{q}_0} ,\quad 
\frac{1}{r_j}=\frac{1-\theta}{p_j}+\frac{\theta}{q_j},  
\quad u_j=w_j^{1-\theta} v_j^{\theta},\quad j=0,\dots,m, 
\end{equation} 
we have 
\begin{equation}\label{eq:WMIP-33}
\bigg[\int_{\Sigma_0} \bigg(\int_{\widetilde{\Sigma}_0} |T(\vec{f})(x,y)|^{\widetilde{r}_0}\, d\widetilde{\mu}_0(y) \bigg)^{\frac{r_0}{\widetilde{r}_0}} u_0(x)^{r_0} d\mu_0(x) \bigg]^{\frac{1}{r_0}}
\le M_1^{1-\theta}M_2^{\theta}\prod_{j=1}^{m}\|f_j\|_{L^{r_j}(\Sigma_j, u_j^{r_j})}  
\end{equation}
for all $\vec{f}=(f_1,\ldots,f_m) \in \mathscr{S}$ with $\|f_j\,w_j\|_{L^{p_j}(\Sigma_j,\, \mu_j)}<\infty$ and $\|f_j\,v_j\|_{L^{q_j}(\Sigma_j,\, \mu_j)}<\infty$, $j=1,\ldots,m$. 
\end{theorem}

\begin{proof}
The proof is similar to that of Theorem \ref{thm:WMIP-1}. We modify it by following the ideas in the proof of an interpolation theorem in mixed $L^p$ spaces in \cite{BP}. We begin with \eqref{eq:WMIP-5}.  Pick $k \in \mathbb{N}$ so that $k> \max\{ \frac{1}{\widetilde{p}_0}, \frac{1}{\widetilde{q}_0}, \frac{1}{p_0}, \frac{1}{q_0}\}$, which implies that $k>\max\{\frac{1}{\widetilde{r}_0}, \frac{1}{r_0}\}$. By \cite[Theorem~1]{BP}, we have 
\begin{align}\label{eq:Tf-dual}
&\bigg[\int_{\Sigma_0} \bigg(\int_{\widetilde{\Sigma}_0} |T(\vec{f})(x, y)|^{\widetilde{r}_0}\, d\widetilde{\mu}_0(y)  \bigg)^{\frac{r_0}{\widetilde{r}_0}}\, u_0(x)^{r_0} 
d\mu_0(x)\bigg]^{\frac{1}{kr_0}}
\nonumber\\ 
&\quad = \sup_{g} \int_{\Sigma_0} \int_{\widetilde{\Sigma}_0} |T(\vec{f})(x,y) u_0(x)|^{\frac{1}{k}} g(x,y) \, d\widetilde{\mu}_0(y)\, d\mu_0(x),
\end{align}
where the supremum is taken over all nonnegative simple functions $g$ on $\widetilde{\Sigma}_0 \times \Sigma_0$ satisfying $\|g\|_{L^{((k\widetilde{r}_0)', (kr_0)' )}(\widetilde{\Sigma}_0, \Sigma_0)}=1$. Fix $\vec{f}=(f_1,\ldots,f_m)$ and $g$. We may assume $\|f_j\,u_j\|_{L^{r_j}(\Sigma_j, \mu_j)}<\infty$ for each $j=1,\ldots,m$. Write $\widetilde{f}_j=f_j u_j$ and $\widetilde{f}_j=|\widetilde{f}_j|e^{is_j}$ for each $j=1,\ldots,m$. Set 
\begin{equation}\label{eq:Lrnorm}
A_1 :=\prod_{j=1}^{m}\|\widetilde{f}_j\|_{L^{r_j}(\Sigma_j,\, \mu_j)}^{r_j/p_j} \quad \text{and}\quad 
A_2 :=\prod_{j=1}^{m}\|\widetilde{f}_j\|_{L^{r_j}(\Sigma_j,\, \mu_j)}^{r_j/q_j} . 
\end{equation} 
Define  for $\ell\in \N$ 
\begin{equation}\label{eq:Phiz}
\Phi_\ell(z) 
:= \int_{\Sigma_0} \int_{\widetilde{\Sigma}_0} |U_\ell(z)|^{\frac1k}\, d\widetilde{\mu}_0 \, d\mu_0,
\end{equation}
where 
\begin{align*}
& U_{\ell}(z) := e^{k(z^2-1)/\ell}(A_1 M_1)^{z-1}(A_2 M_2)^{-z} T(\vec{F}_z)w_0(x)^{1-z}v_0(x)^z G^k_z, 
\\ 
& F_{z,j} :=|\widetilde{f}_j|^{\frac{r_j}{r_j(z)}} e^{is_j}w^{z-1}_jv_j^{-z}, \quad 
\frac{1}{r_j(z)} := \frac{1-z}{p_j} + \frac{z}{q_j}, \, j=1,\dots,m, 
\\
& \frac{1}{\widetilde{r}_0(z)} := \frac{1-z}{\widetilde{p}_0} +\frac{z}{\widetilde{q}_0}, 
\quad \frac{1}{r_0(z)} = \frac{1-z}{p_0} +\frac{z}{q_0},  
\\ 
& G_z := g^{\frac{(k\widetilde{r}_0)'}{(k\widetilde{r}_0(z))'}} \Big(\|g(\cdot, y)\|_{L^{(k\widetilde{r}_0)'} (\widetilde{\Sigma}_0)}\Big)^{\frac{(kr_0)'}{(kr_0(z))'} - \frac{(k\widetilde{r}_0)'}{(k\widetilde{r}_0(z))'}}. 
\end{align*}
Analogously to \eqref{eq:Phi-ell}, it suffices to show 
\begin{align}\label{eq:Phi-ell-2}
\Phi_{\ell}(z) \le 1, \quad\forall \ell \in \N. 
\end{align} 
Then one has $\lim\limits_{\ell \to \infty} \Phi_{\ell}(z) \le 1$, which along with \eqref{eq:Tf-dual} and \eqref{eq:Phiz} implies \eqref{eq:WMIP-33}. 

In order to demonstrate \eqref{eq:Phi-ell-2}, applying the same arguments as in \eqref{eq:sub} and \eqref{eq:WMIP-8-1}, one can verify that $\Phi_{\ell}(z)$ is subharmonic in $S$ and continuous on $\overline{S}$. Furthermore, for each $\ell \in \mathbb{N}$, 
\begin{equation}\label{eq:Imz}
\lim_{|\Im z|\to\infty} \Phi_\ell(z) = 0 \, \text{ uniformly for }0\le \Re z\le 1.
\end{equation}
As shown in \cite[p.~315]{BP}, we have 
\begin{equation}\label{eq:GGiy}
\|G_{iy}\|_{L^{((k\widetilde{p}_0)',(kp_0)')}}=1 \quad \text{and} \quad \|G_{1+iy}\|_{L^{((k\widetilde{q}_0)',(kq_0)')}}=1.
\end{equation}
Now, we need to see what happens to $\Phi_\ell(iy)$ and $\Phi_\ell(1+iy)$. By H\"{o}lder's inequality, \eqref{eq:Lrnorm} and \eqref{eq:GGiy}, we deduce that 
\begin{align}\label{eq:iy}
\Phi_\ell(iy) 
& \leq e^{-|\Im(z)|^2/\ell}(A_1 M_1)^{-\frac1k} \int_{\widetilde{\Sigma}_0} \int_{\Sigma_0} 
|T(\vec{F}_{iy}) w_0(x) G^k_{iy}|^{\frac1k} \,d\widetilde{\mu}_0\, d\mu_0
\\  \nonumber
&\leq e^{-|\Im(z)|^2/\ell}(A_1 M_1)^{-\frac1k} \|G_{iy}\|_{L^{((k\widetilde{p}_0)',(kp_0)')}}
\\  \nonumber
&\qquad\times \bigg[ \int_{\Sigma_0} \bigg( \int_{\widetilde{\Sigma}_0} 
\big(|T(\vec{F}_{iy})w_0(x)|^{\frac{1}{k}} \big)^{k\widetilde{p}_0} \, 
d\widetilde{\mu}_0 \bigg)^{\frac{kp_0}{k\widetilde{p}_0}} \, d\mu_0
\bigg]^{\frac{1}{kp_0}} 
\\  \nonumber
&\le (A_1 M_1)^{-\frac1k} \bigg[ \int_{\Sigma_0} \bigg( \int_{\widetilde{\Sigma}_0} |T(\vec{F}_{iy})|^{\widetilde{p}_0} \, dy \bigg)^{\frac{p_0}{\widetilde{p}_0}} w_0(x)^{p_0}\, dx\bigg]^{\frac{1}{kp_0}} 
\\  \nonumber
&\leq A_1^{-\frac1k} \prod_{j=1}^m \|F_{iy,j} w_j\|_{L^{p_j}(\Sigma_j,\mu_j)}^{\frac{1}{k}} 
= A_1^{-\frac1k} \prod_{j=1}^m \||\widetilde{f}_j|^{r_j/p_j}\|_{L^{p_j}(\Sigma_j,\mu_j)}^{\frac{1}{k}} 
=1. 
\end{align}
Analogously, 
\begin{align}\label{eq:iiy}
\Phi_\ell(1+iy) \le 1. 
\end{align}
Therefore, \eqref{eq:Phi-ell-2} is a consequence of the subharmonicity of $\Phi_{\ell}(z)$, \eqref{eq:Imz}, \eqref{eq:iy} and \eqref{eq:iiy}. The proof is complete. 
\end{proof}

Now let us see how to derive a weighted interpolation for $m$-linear compact operators from Theorem \ref{thm:WMIP-3}. 

\begin{theorem}\label{thm:WMIP-4}
Suppose that $(\Sigma_1, \mu_1)\ldots (\Sigma_m, \mu_m)$ are measure spaces, and $\mathscr{S}_j$ is the collection of all simple functions on $\Sigma_j$, $j=1,\ldots,m$. Denote by $\mathfrak{M}(\mathbb{R}^n)$ the set of all measurable functions on $\mathbb{R}^n$. Let $T: \mathscr{S}=  \mathscr{S}_1 \times \cdots \times  \mathscr{S}_m \rightarrow \mathfrak{M}(\mathbb{R}^n)$ be an $m$-linear operator. Let $0< p_0,q_0<\infty$, $1\leq p_j,q_j\leq \infty$, $j=1,\ldots,m$ and let $w_0^{p_0}, v_0^{q_0} \in A_\infty (\mathbb{R}^n)$ and $w_j,v_j$ be weights on $\Sigma_j$. 
Assume that 
\begin{align}
\label{eq:WMIP-41} &T \text{ is bounded from $L^{p_1}(\Sigma_1, w_1^{p_1}) \times \cdots \times L^{p_m}(\Sigma_m, w_m^{p_m})$ to $L^p(\Rn, w_0^{p_0})$}, 
\\
\label{eq:WMIP-42} &T \text{ is compact from $L^{q_1}(\Sigma_1, v_1^{q_1})\times \cdots \times L^{q_m}(\Sigma_m, v_m^{q_m})$ to $L^{q_0}(\Rn, v_0^{q_0})$}. 
\end{align}
Then, $T$ can be extended as a compact operator from $L^{r_1}(\Sigma_1, u_1^{r_1})\times \cdots \times L^{r_m}(\Sigma_m, u_m^{r_m})$ to $L^{r_0}(\Rn, u_0^{r_0})$ for all exponents satisfying \eqref{eq:exp}.
\end{theorem}

\begin{proof}
Since $w_0^{p_0}, v_0^{q_0} \in A_{\infty}$, there exists $r \in (1, \infty)$ such that $w_0^{p_0}, v_0^{q_0} \in A_r$. Given $\rho >0$, let us consider
\[
\mathcal{N}(f,\rho) := \bigg[\int_{\Rn} \bigg( \fint_{B(0,\rho)} |f(x)- f(x+y)|^{\frac{p_0}{r}} \, dy\bigg)^{r} w_0^{p_0}(x)\, dx\bigg]^{\frac{1}{p_0}}.
\] 
The fact $w_0^{p_0} \in A_r$ implies 
\begin{align}\label{eq:Iff}
\mathcal{N}(f, \rho) &\lesssim \bigg(\int_{\Rn} |f|^{p_0} w_0^{p_0} \, dx \bigg)^{\frac{1}{p_0}} 
+ \bigg(\int_{\Rn} M(|f|^{\frac{p_0}{r}})^{r} w_0^{p_0}\, dx \bigg)^{\frac{1}{p_0}} 
\lesssim \|f\|_{L^{p_0}(w_0^{p_0})}.
\end{align}
In what follows, we always denote $\mathbb{T}(\vec{f})(x, y):=T(\vec{f})(x)- T(\vec{f})(x+y)$. Note that $\mathbb{T}$ is an $m$-linear operator. Then, \eqref{eq:Iff} and \eqref{eq:WMIP-41} yield that for any $\rho>0$,
\begin{equation}\label{eq:WMIP-43}
\bigg[\int_{\Rn} \bigg(\fint_{B(0,\rho)} |\mathbb{T}(\vec{f})(x, y)|^{\frac{p_0}{r}} dy\bigg)^{r} w_0^{p_0}(x)\, dx\bigg]^{\frac{1}{p_0}} \leq M_1 \prod_{j=1}^m \|f_j\|_{L^{p_j}(w_j^{p_j})}.
\end{equation}
On the other hand, from \eqref{eq:WMIP-42} and Theorem \ref{thm:FK-3} we have  
\begin{align}
\label{eq:WMIP-44} &\|T(\vec{f})\|_{L^{q_0}(v_0^{q_0})} \leq  M_2 \prod_{j=1}^m \|f_j\|_{L^{q_j}(v_j^{q_j})}, 
\\
\label{eq:WMIP-45} \lim_{A \to \infty} & \|T(\vec{f})\mathbf{1}_{\{|x|>A\}}\|_{L^{q_0}(v_0^{q_0})} 
\bigg/\prod_{j=1}^m \|f_j\|_{L^{q_j}(v_j^{q_j})} =0, 
\\ 
\label{eq:WMIP-46} \lim_{\rho \to \infty}  \bigg[\int_{\Rn} \bigg(\fint_{B(0,\rho)} & |\mathbb{T}(\vec{f})(x, y)|^{\frac{q_0}{r}} dy\bigg)^r v_0(x)^{q_0} dx\bigg]^{\frac{1}{q_0}} \bigg/\prod_{j=1}^m \|f_j\|_{L^{q_j}(v_j^{q_j})} =0.
\end{align} 
By \eqref{eq:WMIP-41} with the bound $\widetilde{M}_1$, \eqref{eq:WMIP-44} and Theorem \ref{thm:WMIP-1},  there holds 
\begin{equation}\label{eq:WMIP-47}
\|T(\vec{f})\|_{L^{r_0}(u_0^{r_0})} \leq \widetilde{M}_1^{1-\theta}M_2^{\theta} \prod_{j=1}^m \|f_j\|_{L^{r_j}(u_j^{r_j})}.
\end{equation}
From \eqref{eq:WMIP-45}, we obtain that for any $\varepsilon>0$ there exists $A_\varepsilon$ such that for all $A >A_\varepsilon$, 
\begin{equation}\label{eq:WMIP-48}
\|T(\vec{f})\mathbf{1}_{\{|x|>A\}}\|_{L^{q_0}(v_0^{q_0})} 
< \varepsilon \prod_{j=1}^m \|f_j\|_{L^{q_j}(v_j^{q_j})}. 
\end{equation}
Thus, \eqref{eq:WMIP-41} with the bound $\widetilde{M}_1$, \eqref{eq:WMIP-48} and Theorem \ref{thm:WMIP-1} applied to $T(\vec{f}) \mathbf{1}_{\{|x|>A\}}$ give 
\[
\|T(\vec{f})\mathbf{1}_{\{|x|>A\}}\|_{L^{r_0}(u_0^{r_0})} 
\leq \widetilde{M}_1^{1-\theta} \varepsilon^{\theta} \prod_{j=1}^m \|f_j\|_{L^{r_j}(u_j^{r_j})}, 
\]
which asserts  
\begin{equation}\label{eq:WMIP-49}
\lim_{A \to \infty} \|T(\vec{f})\mathbf{1}_{\{|x|>A\}}\|_{L^{r_0}(u_0^{r_0})} 
\bigg/\prod_{j=1}^m \|f_j\|_{L^{r_j}(u_j^{r_j})}=0.
\end{equation}
Additionally,  invoking \eqref{eq:WMIP-46}, we have that for any $\varepsilon>0$ there exists $\rho_0=\rho_0(\varepsilon)>0$ such that for all $0<\rho<\rho_0$,
\begin{align}\label{eq:T_epsilon_2}
& \bigg[\int_{\Rn} \bigg(\fint_{B(0,\rho)} |\mathbb{T}(\vec{f})(x, y)|^{\frac{q_0}{r}} \, dy\bigg)^r v_0(x)^{q_0}   \,dx\bigg]^{\frac{1}{q_0}} 
\leq  \varepsilon \prod_{j=1}^m \|f_j\|_{L^{q_j}(v_j^{q_j})} .
\end{align}
Hence, Theorem \ref{thm:WMIP-3} applied to \eqref{eq:WMIP-43} and \eqref{eq:T_epsilon_2} leads 
\[
\bigg[\int_{\Rn} \bigg(\fint_{B(0,\rho)} |\mathbb{T}(\vec{f})(x, y)|^{\frac{r_0}{r}} \, dy\bigg)^r u_0(x)^{r_0} dx\bigg]^{\frac{1}{r_0}} 
\leq M_1^{1-\theta} \varepsilon^{\theta} \prod_{j=1}^m \|f_j\|_{L^{r_j}(u_j^{r_j})},
\]
 which shows that
 \begin{equation}\label{eq:T_rho}
 \lim_{\rho \to \infty} \bigg[\int_{\Rn} \bigg(\fint_{B(0,\rho)} |\mathbb{T}(\vec{f})(x, y)|^{\frac{r_0}{r}} \, dy\bigg)^r u_0(x)^{r_0} dx\bigg]^{\frac{1}{r_0}} \bigg/\prod_{j=1}^m \|f_j\|_{L^{r_j}(u_j^{r_j})} =0. 
 \end{equation}
Therefore, the desired result follows at once from \eqref{eq:WMIP-47}, \eqref{eq:WMIP-49} and \eqref{eq:T_rho} and Theorem  \ref{thm:FK-3}.
\end{proof}

Finally, we obtain the weighted interpolation for multilinear compact operators when the weights belong to $A_{\vec{p}, \vec{r}}$ classes and the limited range case. To state our results conveniently, we will use $[L^p(w^p), L^q(v^q)]_{\theta}$ to denote the space $L^r(u^r)$ whenever $u=w^{1-\theta} v^{\theta}$, $\frac1r=\frac{1-\theta}{p}+\frac{\theta}{q}$ and $0<p,q<1$.

\begin{corollary}\label{cor:interApr}
Fix $\vec{r}=(r_1, \ldots, r_{m+1})$ with $1 \le r_1, \ldots, r_{m+1}<\infty$. Let $0<p_*, q_* < \infty$, $\vec{r} \preceq \vec{p}$ with $\frac{1}{p_*} \le \frac1p=\frac{1}{p_1}+\cdots+\frac{1}{p_m} \le \frac{1}{p_*}+\frac{1}{r'_{m+1}}$, $\vec{r} \preceq \vec{q}$ with $\frac{1}{q_*} \le \frac1q=\frac{1}{q_1}+\cdots+\frac{1}{q_m} \le \frac{1}{q_*}+\frac{1}{r'_{m+1}}$, and let $\vec{w} \in A_{\vec{p}, \vec{r}}$ and $\vec{v} \in A_{\vec{q}, \vec{r}}$. Assume that $T$ is an m-linear operator such that 
\begin{align}
&T \text{ is bounded from $L^{p_1}(w_1^{p_1}) \times \cdots \times L^{p_m}(w_m^{p_m})$ to $L^{p_*}(w^{p_*})$},  
\\
&T \text{ is compact from $L^{q_1}(v_1^{q_1}) \times \cdots \times L^{q_m}(v_m^{q_m})$ to $L^{q_*}(v^{q_*})$},  
\end{align}
where $w=\prod_{i=1}^m w_i$ and $v=\prod_{i=1}^m v_i$. Then $T$ is compact from $[L^{p_1}(w_1^{p_1}), L^{q_1}(v_1^{q_1})]_{\theta} \times \cdots \times [L^{p_m}(w_m^{p_m}), L^{q_m}(v_m^{q_m})]_{\theta}$ to $[L^{p_*}(w^{p_*}), L^{q_*}(v^{q_*})]_{\theta}$ for any $0<\theta<1$.  
\end{corollary}

\begin{proof}
Let $\vec{w} \in A_{\vec{p}, \vec{r}}$ and $\vec{v} \in A_{\vec{q}, \vec{r}}$. We use the same notation as in \eqref{eq:notation-1} and \eqref{eq:notation-2}. It follows from Lemma \ref{lem:Apr} that $w^{\delta_{m+1}} \in A_{\frac{1-r}{r}\delta_{m+1}}$. By definition, we see that 
\begin{align}
\frac{1}{\delta_{m+1}} = \frac{1}{r_{m+1}} - \frac{1}{p_{m+1}}=\frac1p-\frac{1}{r'_{m+1}} \le \frac{1}{p_*}. 
\end{align}
That is, $p_* \le \delta_{m+1}$. This implies that 
\begin{align}\label{eq:wpAi} 
w^{p_*} \in A_{\frac{1-r}{r}\delta_{m+1}} \subset A_{\infty}.
\end{align} 
Similarly, one has $v^{q_*} \in A_{\infty}$. Therefore, Corollary \ref{cor:interApr} is a consequence of Theorem \ref{thm:WMIP-4}. 
\end{proof}

\begin{corollary}\label{cor:interlim}
Let $1 \leq \p_i^{-} < \p_i^{+} \leq \infty$, $p_i, q_i \in [\p_i^{-}, \p_i^{+}]$, and let $w_i^{p_i} \in A_{\frac{p_i}{\p_i^{-}}} \cap RH_{\big(\frac{\p_i^{+}}{p_i}\big)'}$ and $v_i^{q_i} \in A_{\frac{q_i}{\p_i^{-}}} \cap RH_{\big(\frac{\p_i^{+}}{q_i}\big)'}$, $i=1,\ldots,m$. Assume that $T$ is an m-linear operator such that 
\begin{align}
&T \text{ is bounded from $L^{p_1}(w_1^{p_1}) \times \cdots \times L^{p_m}(w_m^{p_m})$ to $L^p(w^p)$},  
\\
&T \text{ is compact from $L^{q_1}(v_1^{q_1}) \times \cdots \times L^{q_m}(v_m^{q_m})$ to $L^q(v^q)$},  
\end{align}
where $w=\prod_{i=1}^m w_i$ and $v=\prod_{i=1}^m v_i$. Then $T$ is compact from $[L^{p_1}(w_1^{p_1}), L^{q_1}(v_1^{q_1})]_{\theta} \times \cdots \times [L^{p_m}(w_m^{p_m}), L^{q_m}(v_m^{q_m})]_{\theta}$ to $[L^p(w^p), L^q(v^q)]_{\theta}$ for any $0<\theta<1$.  
\end{corollary}

\begin{proof}
Let $w_i^{p_i} \in A_{\frac{p_i}{\p_i^{-}}} \cap RH_{\big(\frac{\p_i^{+}}{p_i}\big)'}$ and $v_i^{q_i} \in A_{\frac{q_i}{\p_i^{-}}} \cap RH_{\big(\frac{\p_i^{+}}{q_i}\big)'}$, $i=1,\ldots,m$. By Lemma \ref{lem:lim-Apr}, there holds $\vec{w}=(w_1,\ldots,w_m) \in A_{\vec{t}, \vec{r}}$, where $\vec{t}$ and $\vec{r}$ are defined in Lemma \ref{lem:lim-Apr}. In view of \eqref{eq:wpAi}, we obtain 
\begin{align}\label{eq:wtAi}
w^t \in A_{\infty},\quad\text{where}\quad \frac{1}{t}=\frac{1}{t_1}+\cdots+\frac{1}{t_m}. 
\end{align}
Observe that $t_i=p_i(\p_i^{+}/p_i)' \ge p_i$ for each $i=1,\ldots,m$, which implies $p \le t$. This and \eqref{eq:wtAi} yield 
$w^p \in A_{\infty}$. Analogously, $v^q \in A_{\infty}$. Hence, Corollary \ref{cor:interlim} immediately follows from these  and Theorem \ref{thm:WMIP-4}. 
\end{proof}

In Section \ref{sec:com}, we will use Corollaries \ref{cor:interApr} and \ref{cor:interlim} to show Theorems  \ref{thm:Ap} and \ref{thm:lim}.

\section{Extrapolation of compactness}\label{sec:com}
The goal of this section is to present the proof of Theorems \ref{thm:Ap}--\ref{thm:limTb}. For this purpose, we establish a fundamental result about $A_p$ weights below, which generalizes the main points in weighted interpolation theorems involving $A_{\vec{p}, \vec{r}}$ and limited range weights, see Lemmas \ref{lem:int-Lp} and \ref{lem:int-lim}. 

\begin{lemma}\label{lem:AAA}
Fix $1<\gamma_i, \widetilde{\gamma}_i, \eta_i, \widetilde{\eta}_i<\infty$ such that $\frac{\eta_i}{\gamma_i}=\frac{\widetilde{\eta}_i}{\widetilde{\gamma}_i}$, $i=1,\ldots,m$. Assume that $w_i^{\gamma_i} \in A_{\eta_i}$ and $v_i^{\widetilde{\gamma}_i} \in A_{\widetilde{\eta}_i}$ for each $i=1,\ldots,m$. Then there exists $\theta \in (0, 1)$ such that \begin{align}\label{eq:ur}
u_i^{\widehat{\gamma}_i} \in A_{\widehat{\eta}_i},\quad i=1,\ldots,m, 
\end{align}
where 
\begin{align}\label{eq:gaga}
w_i=u_i^{1-\theta} v_i^{\theta},\quad 
\frac{1}{\gamma_i}=\frac{1-\theta}{\widehat{\gamma}_i}+\frac{\theta}{\widetilde{\gamma}_i},\quad  
\frac{1}{\eta_i}=\frac{1-\theta}{\widehat{\eta}_i} + \frac{\theta}{\widetilde{\eta}_i}, \quad i=1,\ldots,m. 
\end{align}
\end{lemma}

\begin{proof}
Let $w_i^{\gamma_i} \in A_{\eta_i}$ and $v_i^{\widetilde{\gamma}_i} \in A_{\widetilde{\eta}_i}$, $i=1,\ldots,m$.  In view of Lemma \ref{lem:RH}, there exist $\tau_i, \widetilde{\tau}_i \in (1, \infty)$  such that 
\begin{align}\label{eq:RH}
\bigg(\fint_{Q} w_i^{\gamma_i \tau_i} dx\bigg)^{\frac{1}{\tau_i}} 
\lesssim \fint_{Q} w_i^{\gamma_i} dx \quad\text{and}\quad 
\bigg(\fint_{Q} v_i^{\widetilde{\gamma}_i \widetilde{\tau}_i} dx\bigg)^{\frac{1}{\widetilde{\tau}_i}} 
\lesssim \fint_{Q} v_i^{\widetilde{\gamma}_i} dx, 
\end{align} 
for every cube $Q \subset \Rn$. Given $\theta \in (0, 1)$, we define $u_i$, $\widehat{\gamma}_i$ and $\widehat{\eta}_i$ as in \eqref{eq:gaga}, and pick 
\begin{align*}
\alpha_i=\alpha_i(\theta) := \theta\eta_i/\widetilde{\eta}'_i \quad\text{and}\quad 
\beta_i=\beta_i(\theta) := \theta \eta'_i/\widetilde{\eta}_i, \quad i=1,\ldots,m.  
\end{align*}
Then one can verify that 
\begin{align}
\label{eq:ki} \kappa_i=\kappa_i(\theta) &:=\frac{\widehat{\gamma}_i (1+\alpha_i)}{\gamma_i(1-\theta)}
=\frac{\widehat{\eta}_i \theta(1+\alpha_i)}{\widetilde{\eta}'_i (1-\theta)\alpha_i} 
=\frac{\widehat{\gamma}_i \big(\widetilde{\eta}'_i+\theta\eta_i\big)}{\gamma_i \widetilde{\eta}'_i(1-\theta)}, 
\\
\label{eq:kki} \widetilde{\kappa}_i=\widetilde{\kappa}_i (\theta)
&:=\frac{\widehat{\eta}'_i (1+\beta_i)}{\eta_i' (1-\theta)} 
=\frac{\widehat{\eta}'_i \theta(1+\beta_i)}{\widetilde{\eta}_i(1-\theta)\beta_i}  
= \frac{\widehat{\eta}'_i \big(\widetilde{\eta}_i+\theta \eta_i'\big)}{\eta_i' \widetilde{\eta}_i (1-\theta)}. 
\end{align}
From \eqref{eq:gaga}, we see that $\widehat{\gamma}_i=\widehat{\gamma}_i(\theta)$ depends only on $\theta$ and $\widehat{\gamma}_i(0^+)=\gamma_i$. Together with \eqref{eq:ki} and \eqref{eq:kki}, the latter in turn gives that $\kappa_i (0^+)=\widehat{\gamma}_i(0^+)/\gamma_i=1$ and $\widetilde{\kappa}_i(0^+)=(\widehat{\eta}_i(0^+))'/\eta_i'=1$. Hence, by continuity, one has 
\begin{align}\label{eq:kkk}
\kappa_i=\kappa_i(\theta) <\tau_i \quad\text{and}\quad 
\widetilde{\kappa}_i=\widetilde{\kappa}_i(\theta) < \widetilde{\tau}_i, \quad i=1,\ldots,m, 
\end{align}
if $\theta \in (0, 1)$ is small enough. Hereafter, we fix $\theta \in (0, 1)$ sufficiently small such that \eqref{eq:kkk} holds. 

By our assumption and \eqref{eq:gaga}, there holds 
\begin{equation}\label{eq:erer}
\frac{\eta_i}{\gamma_i}=\frac{\widetilde{\eta}_i}{\widetilde{\gamma}_i} 
=\frac{\widehat{\eta}_i}{\widehat{\gamma}_i},\quad i=1,\ldots,m. 
\end{equation}
Now, using $w_i=u_i^{1-\theta} v_i^{\theta}$, H\"{o}lder's inequality, \eqref{eq:ki}, \eqref{eq:RH} and \eqref{eq:kkk}, we conclude that 
\begin{align}\label{eq:ui-1}
\fint_{Q} u_i^{\widehat{\gamma}_i}\, dx 
&=\fint_{Q} w_i^{\frac{\widehat{\gamma}_i}{1-\theta}} v_i^{-\frac{\theta \widehat{\gamma}_i}{1-\theta}} dx
=\fint_{Q} (w_i^{\gamma_i})^{\frac{\widehat{\gamma}_i}{\gamma_i(1-\theta)}} \big(v_i^{\widetilde{\gamma}_i (1-\widetilde{\eta}'_i)}\big)^{\frac{\widehat{\eta}_i \theta}{\widetilde{\eta}'_i (1-\theta)}} dx
\\ \nonumber
&\le \bigg(\fint_{Q} (w_i^{\gamma_i})^{\frac{\widehat{\gamma}_i (1+\alpha_i)}{\gamma_i(1-\theta)}} dx\bigg)^{\frac{1}{1+\alpha_i}}
\bigg(\fint_{Q} \big(v_i^{\widetilde{\gamma}_i (1-\widetilde{\eta}'_i)}\big)^{\frac{\widehat{\eta}_i \theta (1+\alpha_i)}{\widetilde{\eta}'_i (1-\theta) \alpha_i}} dx\bigg)^{\frac{\alpha_i}{1+\alpha_i}} 
\\ \nonumber
&=\bigg(\fint_{Q} w_i^{\gamma_i \kappa_i} \, dx\bigg)^{\frac{\widetilde{\eta}'_i}{\widetilde{\eta}'_i + \theta \eta_i}} 
\bigg(\fint_{Q} v_i^{\widetilde{\gamma}_i (1-\widetilde{\eta}'_i) \kappa_i} dx\bigg)^{\frac{\theta \eta_i}{\widetilde{\eta}'_i+\theta \eta_i}} 
\\ \nonumber
&\lesssim \bigg(\fint_{Q} w_i^{\gamma_i} dx\bigg)^{\frac{\kappa_i \widetilde{\eta}'_i}{\widetilde{\eta}'_i + \theta \eta_i}} 
\bigg(\fint_{Q} v_i^{\widetilde{\gamma}_i (1-\widetilde{\eta}'_i)} dx\bigg)^{\frac{\kappa_i \theta \eta_i}{\widetilde{\eta}'_i+\theta \eta_i}} 
\\ \nonumber
&= \bigg(\fint_{Q} w_i^{\gamma_i} \, dx\bigg)^{\frac{\widehat{\gamma}_i}{\gamma_i(1-\theta)}} 
\bigg(\fint_{Q} v_i^{\widetilde{\gamma}_i (1-\widetilde{\eta}'_i)} dx\bigg)^{\frac{\widehat{\eta}_i \theta}{\widetilde{\eta}'_i(1-\theta)}}.  
\end{align}
Analogously, we have 
\begin{align}\label{eq:ui-2}
\fint_{Q} u_i^{\widehat{\gamma}_i (1-\widehat{\eta}'_i)} dx
&=\fint_{Q} (w_i^{\gamma_i(1-\eta_i')})^{\frac{\widehat{\eta}'_i}{\eta_i'(1-\theta)}} (v_i^{\widetilde{\gamma}_i})^{\frac{\widehat{\eta}_i \theta}{\widetilde{\eta}_i(1-\theta)}} dx  
\\ \nonumber
&\le \fint_{Q} (w_i^{\gamma_i(1-\eta_i')})^{\frac{\widehat{\eta}'_i (1+\beta_i)}{\eta_i' (1-\theta)}} dx \bigg)^{\frac{1}{1+\beta_i}}
\bigg(\fint_{Q}(v_i^{\widetilde{\gamma}_i})^{\frac{\widehat{\eta}_i \theta(1+\beta_i)}{\widetilde{\eta}_i(1-\theta)\beta_i}} dx \bigg)^{\frac{\beta_i}{1+\beta_i}}
\\ \nonumber
&=\bigg(\fint_{Q}w_i^{\gamma_i(1-\eta_i')\widetilde{\kappa}_i} dx\bigg)^{\frac{\widetilde{\eta}_i}{\widetilde{\eta}_i+\theta \eta_i'}}  
\bigg(\fint_{Q} v_i^{\widetilde{\gamma}_i \widetilde{\kappa}_i} dx\bigg)^{\frac{\theta \eta_i'}{\widetilde{\eta}_i+\theta \eta_i'}}   
\\ \nonumber
&\lesssim \bigg(\fint_{Q}w_i^{\gamma_i(1-\eta_i')} dx\bigg)^{\frac{\widetilde{\kappa}_i \widetilde{\eta}_i}{\widetilde{\eta}_i+\theta \eta_i'}}  
\bigg(\fint_{Q} v_i^{\widetilde{\gamma}_i} dx\bigg)^{\frac{\widetilde{\kappa}_i \theta \eta_i'}{\widetilde{\eta}_i+\theta \eta_i'}}  
\\ \nonumber
&=\bigg(\fint_{Q}w_i^{\gamma_i(1-\eta_i')} dx\bigg)^{\frac{\widehat{\eta}'_i}{\eta_i'(1-\theta)}}   
\bigg(\fint_{Q} v_i^{\widetilde{\gamma}_i} dx\bigg)^{\frac{\widehat{\eta}'_i \theta}{\widetilde{\eta}_i (1-\theta)}}.    
\end{align}
Gathering \eqref{eq:erer}, \eqref{eq:ui-1} and \eqref{eq:ui-2}, we obtain 
\begin{align*}
\bigg(\fint_{Q} & u_i^{\widehat{\gamma}_i}\, dx\bigg) 
\bigg(\fint_{Q} u_i^{\widehat{\gamma}_i (1-\widehat{\eta}_i')} dx\bigg)^{\widehat{\eta}_i-1}
\\
&\lesssim \bigg(\fint_{Q} w_i^{\gamma_i} \, dx\bigg)^{\frac{\widehat{\gamma}_i}{\gamma_i(1-\theta)}} 
\bigg(\fint_{Q}w_i^{\gamma_i(1-\eta_i')} dx\bigg)^{\frac{\widehat{\eta}_i}{\eta_i'(1-\theta)}}   
\\
&\qquad\times \bigg(\fint_{Q} v_i^{\widetilde{\gamma}_i} dx\bigg)^{\frac{\widehat{\eta}_i \theta}{\widetilde{\eta}_i (1-\theta)}}   
\bigg(\fint_{Q} v_i^{\widetilde{\gamma}_i (1-\widetilde{\eta}_i')} dx\bigg)^{\frac{\widehat{\eta}_i \theta}{\widetilde{\eta}_i' (1-\theta)}} 
\\
&= \bigg\{\bigg(\fint_{Q} w_i^{\gamma_i} \, dx\bigg)
\bigg(\fint_{Q}w_i^{\gamma_i(1-\eta_i')} dx\bigg)^{\eta_i-1} \bigg\}^{\frac{\widehat{\gamma}_i}{\gamma_i(1-\theta)}} 
\\
&\qquad\times \bigg\{\bigg(\fint_{Q} v_i^{\widetilde{\gamma}_i} dx\bigg)
\bigg(\fint_{Q} v_i^{\widetilde{\gamma}_i (1-\widetilde{\eta}'_i)} dx\bigg)^{\widetilde{\eta}_i-1} \bigg\}^{\frac{\widehat{\gamma}_i \theta}{\widetilde{\gamma}_i (1-\theta)}}   
\\
&\le [w_i^{\gamma_i}]_{A_{\eta_i}}^{\frac{\widehat{\gamma}_i}{\gamma_i(1-\theta)}}  
[v_i^{\widetilde{\gamma}_i}]_{A_{\widetilde{\eta}_i}}^{\frac{\widehat{\gamma}_i \theta}{\widetilde{\gamma}_i(1-\theta)}}  
=[w_i^{\gamma_i}]_{A_{\eta_i}}^{\frac{\widetilde{\gamma}_i}{\widetilde{\gamma}_i-\theta \gamma_i}}  
[v_i^{\widetilde{\gamma}_i}]_{A_{\widetilde{\eta}_i}}^{\frac{\theta \gamma_i}{\widetilde{\gamma}_i -\theta \gamma_i}},  
\end{align*}
where we used \eqref{eq:gaga} in the last step. This gives that $u_i^{\widehat{\gamma}_i} \in A_{\widehat{\eta}_i}$ for each $i=1,\ldots,m$, and hence shows \eqref{eq:ur}.  
\end{proof}

We recall an interpolation theorem due to Stein-Weiss \cite{SW}. 
\begin{lemma}\label{lem:SW}
Let $1 \le p_0, p_1<\infty$ and let $w_0$, $w_1$ be two weights. Then for any $\theta \in (0, 1)$, 
\begin{align*}
[L^{p_0}(w_0^{p_0}), L^{p_1}(w_1^{p_1})]_{\theta} = L^p(w^p),
\end{align*}
where $\frac1p=\frac{1-\theta}{p_0}+\frac{\theta}{p_1}$ and $w=w_0^{1-\theta} w_1^{\theta}$.
\end{lemma}

For convenience, in what follows, the notation $[L^p(w^p), L^q(v^q)]_{\theta}$ will denote the space $L^r(u^r)$ whenever $u=w^{1-\theta} v^{\theta}$, $\frac1r=\frac{1-\theta}{p}+\frac{\theta}{q}$ and $0<p,q<1$.

\begin{lemma}\label{lem:int-Lp} 
Let $\vec{r}=(r_1, \ldots, r_{m+1})$ with $1 \le r_1, \ldots, r_{m+1}<\infty$, and let $\vec{p}=(p_1,\ldots,p_m)$ with $\vec{r} \preceq \vec{p}$ and $\vec{q}=(q_1,\ldots,q_m)$ with $\vec{r} \preceq \vec{q}$. If $\vec{w} \in A_{\vec{p}, \vec{r}}$ and $\vec{v} \in A_{\vec{q}, \vec{r}}$, then there exist $\theta \in (0, 1)$,  $\vec{s}=(s_1,\ldots,s_m)$ with $\vec{r} \preceq \vec{s}$, and $\vec{u} \in A_{\vec{s}, \vec{r}}$ such that 
\begin{align*}
L^{p_*}(w^{p_*}) = [L^{s_*}(u^{s_*}), L^{q_*}(v^{q_*})]_{\theta} \quad\text{and}\quad 
L^{p_i}(w_i^{p_i}) = [L^{s_i}(u_i^{s_i}), L^{q_i}(v_i^{q_i})]_{\theta}, 
\end{align*}
for each $i=1,\ldots,m$, where $\frac1p-\frac{1}{p_*}=\frac1q-\frac{1}{q_*}=\frac1s-\frac{1}{s_*}$, $\frac{1}{p}=\sum_{i=1}^m \frac{1}{p_i}$, $\frac{1}{q}=\sum_{i=1}^m \frac{1}{q_i}$, $\frac{1}{s}=\sum_{i=1}^m \frac{1}{s_i}$, $w=\prod_{i=1}^m w_i$, $u=\prod_{i=1}^m u_i$ and $v=\prod_{i=1}^m v_i$. 
\end{lemma}

\begin{proof}
Let $\vec{w} \in A_{\vec{p}, \vec{r}}$ and $\vec{v} \in A_{\vec{q}, \vec{r}}$. We claim that there exist $\theta \in (0, 1)$, $\vec{s}=(s_1,\ldots,s_m)$ with $\vec{r} \preceq \vec{s}$, and $\vec{u} \in A_{\vec{s}, \vec{r}}$ such that 
\begin{align}\label{eq:pw}
\frac{1}{p_i}=\frac{1-\theta}{s_i}+\frac{\theta}{q_i} 
\quad\text{ and }\quad w_i=u_i^{1-\theta} v_i^{\theta},\quad i=1,\ldots,m. 
\end{align}
Once \eqref{eq:pw} is proved, it follows from Lemma \ref{lem:SW} that 
\begin{align*}
L^{p_i}(w_i^{p_i}) = [L^{s_i}(u_i^{s_i}), L^{q_i}(v_i^{q_i})]_{\theta},\quad i=1,\ldots,m. 
\end{align*}
In addition, from \eqref{eq:pw}, we see that 
\begin{align}\label{eq:pp}
\frac1p=\frac{1-\theta}{s}+\frac{\theta}{q}, \qquad \frac{1}{p_*}=\frac{1-\theta}{s_*}+\frac{\theta}{q_*}, 
\end{align}
and 
\begin{align}\label{eq:ww}
w=\prod_{i=1}^m w_i = \bigg(\prod_{i=1}^m u_i\bigg)^{1-\theta} 
\bigg(\prod_{i=1}^m v_i\bigg)^{\theta} = u^{1-\theta} v^{\theta}. 
\end{align} 
Therefore, \eqref{eq:pp} and \eqref{eq:ww} imply 
\begin{align*}
L^{p_*}(w^{p_*}) = [L^{s_*}(u^{s_*}), L^q(v^{q_*})]_{\theta}. 
\end{align*}

It remains to show our claim \eqref{eq:pw}. To proceed, we let $\vec{w} \in A_{\vec{p}, \vec{r}}$ and $\vec{v} \in A_{\vec{q}, \vec{r}}$. Set $\frac{1}{p_{m+1}}:=1-\frac1p$, $\frac{1}{q_{m+1}} :=1-\frac1q$, 
\begin{align}\label{eq:drp}
\frac{1}{r}:=\sum_{i=1}^{m+1}\frac{1}{r_i},\quad 
\frac{1}{\delta_i} := \frac{1}{r_i} - \frac{1}{p_i}, \quad 
\frac{1}{\widetilde{\delta}_i} := \frac{1}{r_i} - \frac{1}{q_i}, \quad i=1,\ldots, m+1,  
\end{align}
and 
\begin{align}\label{eq:titi}
\frac{1}{\theta_i} := \frac1r-1-\frac{1}{\delta_i}, \quad
\frac{1}{\widetilde{\theta}_i} := \frac1r-1-\frac{1}{\widetilde{\delta}_i}, \quad i=1,\ldots,m. 
\end{align}
For convenience, denote $\theta_{m+1}:=\delta_{m+1}$, $\widetilde{\theta}_{m+1}:=\widetilde{\delta}_{m+1}$, $w_{m+1}:=w$ and $v_{m+1}:=v$. Then, it follows from Lemma \ref{lem:Apr} that 
\begin{align*}
w_i^{\theta_i} \in A_{\frac{1-r}{r}\theta_i} =: A_{\eta_i} \quad\text{and}\quad 
v_i^{\widetilde{\theta}_i} \in A_{\frac{1-r}{r} \widetilde{\theta}_i} =:A_{\widetilde{\eta}_i}, \quad i=1,\ldots,m+1.  
\end{align*}  
By Lemma \ref{lem:AAA}, there exists $\theta \in (0, 1)$ such that $u_i^{\widehat{\theta}_i} \in A_{\widehat{\eta}_i}$, $i=1,\ldots,m+1$, where 
\begin{align}\label{eq:ttt}
w_i=u_i^{1-\theta} v_i^{\theta},\quad 
\frac{1}{\theta_i}=\frac{1-\theta}{\widehat{\theta}_i}+\frac{\theta}{\widetilde{\theta}_i},\quad 
\frac{1}{\eta_i}=\frac{1-\theta}{\widehat{\eta}_i} + \frac{\theta}{\widetilde{\eta}_i},\quad i=1,\ldots,m+1. 
\end{align}
Using \eqref{eq:ttt}, $w=\prod_{i=1}^m w_i$, $v=\prod_{i=1}^m v_i$, $\eta_i=\frac{1-r}{r}\theta_i$ and $\widetilde{\eta}_i=\frac{1-r}{r}\widetilde{\theta}_i$, we obtain 
\begin{equation}\label{eq:ueta} 
u_{m+1}=u=\prod_{i=1}^m u_i \quad\text{and}\quad  
\widehat{\eta}_i=\frac{1-r}{r}\widehat{\theta}_i,\quad i=1,\ldots,m+1.
\end{equation} 
This gives that 
\begin{align}\label{eq:uii}
u_i^{\widehat{\theta}_i} \in A_{(\frac1r-1)\widehat{\theta}_i},\quad i=1,\ldots,m+1. 
\end{align}
Pick $s_i$ such that 
\begin{align}\label{eq:risi} 
\frac{1}{r_i} - \frac{1}{s_i}=\frac{1}{\widehat{\delta}_i},  \quad i=1,\ldots,m+1, 
\end{align}
where 
\begin{align}\label{eq:dm}
\frac{1}{\widehat{\delta}_i}:= \frac1r-1-\frac{1}{\widehat{\theta}_i}, \quad i=1,\ldots,m,\quad 
\text{and}\quad \widehat{\delta}_{m+1}:=\widehat{\theta}_{m+1}. 
\end{align}
Inserting \eqref{eq:titi} and \eqref{eq:dm} into the second term in \eqref{eq:ttt}, we obtain that $\frac{1}{\delta_i} = \frac{1-\theta}{\widehat{\delta}_i}+\frac{\theta}{\widetilde{\delta}_i}$, which together with \eqref{eq:drp} and \eqref{eq:risi} gives that 
\begin{align}\label{eq:pw-1}
\frac{1}{p_i}=\frac{1-\theta}{s_i}+\frac{\theta}{q_i}, \quad i=1,\ldots,m. 
\end{align}
Additionally, from \eqref{eq:uii} and \eqref{eq:dm}, one has 
\begin{align}\label{eq:uu}
u^{\widehat{\delta}_{m+1}} \in A_{(\frac1r-1)\widehat{\delta}_{m+1}} \quad\text{and}\quad 
u_i^{\widehat{\theta}_i} \in A_{(\frac1r-1)\widehat{\theta}_i}, \quad i=1,\ldots,m. 
\end{align}
As a consequence, Lemma \ref{lem:Apr} and \eqref{eq:uu} imply at once that $\vec{u} \in A_{\vec{s}, \vec{r}}$. This shows \eqref{eq:pw} and completes the proof.   
\end{proof}

\begin{lemma}\label{lem:int-lim}
Let $1 \leq \p_i^{-} < \p_i^{+} \leq \infty$ and $p_i, q_i \in (\p_i^{-}, \p_i^{+})$, $i=1,\ldots,m$. If $w_i^{p_i} \in A_{\frac{p_i}{\p_i^{-}}} \cap RH_{\big(\frac{\p_i^{+}}{p_i}\big)'}$ and $v_i^{q_i} \in A_{\frac{q_i}{\p_i^{-}}} \cap RH_{\big(\frac{\p_i^{+}}{q_i}\big)'}$, $i=1,\ldots,m$, then there exist $s_i \in (\p_i^{-}, \p_i^{+})$ and $\theta \in (0, 1)$ such that $u_i^{s_i} \in A_{\frac{s_i}{\p_i^{-}}} \cap RH_{\big(\frac{\p_i^{+}}{s_i}\big)'}$, 
\begin{align*}
L^p(w^p) = [L^s(u^s), L^q(v^q)]_{\theta} \quad\text{ and }\quad 
L^{p_i}(w_i^{p_i}) = [L^{s_i}(u_i^{s_i}), L^{q_i}(v_i^{q_i})]_{\theta}, 
\end{align*}
where $\frac1p=\frac{1}{p_1}+\cdots+\frac{1}{p_m}$, $\frac{1}{q}=\frac{1}{q_1}+\cdots+\frac{1}{q_m}$, $\frac1s=\frac{1}{s_1}+\cdots+\frac{1}{s_m}$, $w=\prod_{i=1}^m w_i$, $v=\prod_{i=1}^m v_i$ and $u=\prod_{i=1}^m u_i$.
\end{lemma}

\begin{proof}
Let $w_i^{p_i} \in A_{\frac{p_i}{\p_i^{-}}} \cap RH_{\big(\frac{\p_i^{+}}{p_i}\big)'}$ and $v_i^{q_i} \in A_{\frac{q_i}{\p_i^{-}}} \cap RH_{\big(\frac{\p_i^{+}}{q_i}\big)'}$, $i=1,\ldots,m$. As we did in the proof of Lemma \ref{lem:int-Lp}, it suffices to show that there exist $s_i \in (\p_i^{-}, \p_i^{+})$, $u_i^{s_i} \in A_{\frac{s_i}{\p_i^{-}}} \cap RH_{\big(\frac{\p_i^{+}}{s_i}\big)'}$ and $\theta \in (0, 1)$ such that 
\begin{align}\label{eq:pswu}
\frac{1}{p_i}=\frac{1-\theta}{s_i}+\frac{\theta}{q_i} 
\quad\text{ and }\quad w_i=u_i^{1-\theta} v_i^{\theta},\quad i=1,\ldots,m. 
\end{align}

Denote 
\begin{align}
\label{eq:gap} \gamma_i &:=p_i(\p_i^+/p_i)',\quad 
\eta_i:=\bigg(\frac{\p_i^+}{p_i}\bigg)'\bigg(\frac{p_i}{\p_i^-}-1\bigg)+1, \quad i=1,\ldots,m, 
\\
\label{eq:gaq} \widetilde{\gamma}_i &:=q_i(\p_i^+/q_i)',\quad 
\widetilde{\eta}_i:=\bigg(\frac{\p_i^+}{q_i}\bigg)'\bigg(\frac{q_i}{\p_i^-}-1\bigg)+1, \quad i=1,\ldots,m. 
\end{align}
Then it follows from \eqref{eq:JN} that $w_i^{\gamma_i} \in A_{\eta_i}$ and $v_i^{\widetilde{\gamma}_i} \in A_{\widetilde{\eta}_i}$, $i=1,\ldots,m$. Observe that 
\begin{equation}\label{eq:etag}
\frac{\eta_i}{\gamma_i}=\frac{\widetilde{\eta}_i}{\widetilde{\gamma}_i}
=\frac{1}{\p_i^-} - \frac{1}{\p_i^+},\quad i=1,\ldots,m. 
\end{equation}
Thus, by Lemma \ref{lem:AAA}, there exists $\theta \in (0, 1)$ such that $u_i^{\widehat{\gamma}_i} \in A_{\widehat{\eta}_i}$, $i=1,\ldots,m$, where 
\begin{align}\label{eq:con-1}
w_i=u_i^{1-\theta} v_i^{\theta},\quad 
\frac{1}{\gamma_i}=\frac{1-\theta}{\widehat{\gamma}_i}+\frac{\theta}{\widetilde{\gamma}_i},\quad 
\frac{1}{\eta_i}=\frac{1-\theta}{\widehat{\eta}_i} + \frac{\theta}{\widetilde{\eta}_i},\quad i=1,\ldots,m. 
\end{align}
Pick $s_i \in (\p_i^-, \p_i^+)$ such that 
\begin{equation}\label{eq:srp} 
\frac{1}{\widehat{\gamma}_i}=\frac{1}{s_i}-\frac{1}{\p_i^+},\quad i=1,\ldots,m.
\end{equation} 
Inserting \eqref{eq:srp} into the second term in \eqref{eq:con-1}, and using \eqref{eq:gap} and \eqref{eq:gaq}, we deduce that 
\[
\frac{1}{p_i}-\frac{1}{\p_i^+} = (1-\theta) \bigg(\frac{1}{s_i} - \frac{1}{\p_i^+}\bigg) + \theta\bigg(\frac{1}{q_i}-\frac{1}{\p_i^+}\bigg), 
\]
and hence, 
\begin{equation}\label{eq:con-2}
\frac{1}{p_i}=\frac{1-\theta}{s_i}+\frac{\theta}{q_i},\quad i=1,\ldots,m. 
\end{equation}
Furthermore, from \eqref{eq:etag}, \eqref{eq:con-1} and \eqref{eq:srp}, we have 
\begin{align}\label{eq:wide}
\widehat{\eta}_i = \widehat{\gamma}_i \bigg(\frac{1}{\p_i^-} - \frac{1}{\p_i^+}\bigg)
=\bigg(\frac{\p_i^+}{s_i}\bigg)' \bigg(\frac{s_i}{\p_i^-}-1\bigg)+1, \quad i=1,\ldots,m. 
\end{align}
Using \eqref{eq:srp}, \eqref{eq:wide} and \eqref{eq:JN}, we see that $u_i^{\widehat{\gamma}_i} \in A_{\widehat{\eta}_i}$ is equivalent to  
\begin{equation}\label{eq:con-3}
u_i^{s_i}  \in A_{\frac{s_i}{\p_i^-}} \cap RH_{\big(\frac{\p_i^+}{s_i}\big)'},\quad i=1,\ldots,m. 
\end{equation}
Therefore, \eqref{eq:pswu} follows from the first one in \eqref{eq:con-1}, \eqref{eq:con-2} and \eqref{eq:con-3}. 
\end{proof}

In order to show Theorem \ref{thm:Ap}, we first need to establish an off-diagonal extrapolation for the boundedness below, which is a generalization of Theorem {\bf A}. 
\begin{theorem}\label{thm:Apq}
Let $\F$ be a collection of $(m+1)$-tuples of non-negative functions. Let $\vec{r}=(r_1, \ldots, r_{m+1})$ with $1 \le r_1, \ldots, r_{m+1}<\infty$. Assume that there exist $p_{*} \in (0, \infty)$ and $\vec{p}=(p_1, \ldots, p_m)$ with $\vec{r} \preceq \vec{p}$ and $\frac{1}{p_*} \le \frac1p=\frac{1}{p_1}+\cdots+\frac{1}{p_m}$ such that for all $\vec{w}=(w_1, \ldots, w_m) \in A_{\vec{p}, \vec{r}}$, 
\begin{align}
\|f\|_{L^{p_*}(w^{p_*})} \le C \prod_{i=1}^m \|f_i\|_{L^{p_i}(w_i^{p_i})}, \quad (f, f_1, \dots, f_m) \in \F, 
\end{align}
where $w=\prod_{i=1}^m w_i$. Then, for all $q_* \in (0, \infty)$, for all $\vec{q}=(q_1,\ldots,q_m)$ with $\vec{r} \prec \vec{q}$, $\frac1q=\frac{1}{q_1}+\cdots+\frac{1}{q_m}$ and $\frac1q-\frac{1}{q_*}=\frac1p-\frac{1}{p_*}$, and for all $\vec{v}=(v_1, \ldots, v_m) \in A_{\vec{q}, \vec{r}}$, we have
\begin{align}
\|f\|_{L^{q_*}(v^{q_*})} \le C \prod_{i=1}^m \|f_i\|_{L^{q_i}(v_i^{q_i})}, \quad (f, f_1, \dots, f_m) \in \F, 
\end{align}
where $v=\prod_{i=1}^m v_i$. 
\end{theorem}

\begin{proof}
Note that \cite[Lemma 3.2]{LMO} presents a general characterization of $A_{\vec{p}, \vec{r}}$. It is still valid in the current scenario. Keeping the same notation, one can verify that for $\frac{1}{\gamma} := \frac{1}{r'_{m+1}}+\frac{1}{p_*}-\frac1p$, 
\begin{align}
\label{eq:fW-1} \|f\|_{L^{p_*}(w^{p_*})} 
&=\Big\|\Big(f\widehat{w}^{-\frac{1}{\gamma}}\Big)^{r_m}W\Big\|_{L^{\frac{p_*}{r_m}}(\widehat{w})}^{\frac{1}{r_m}}, 
\\
\label{eq:fW-2} \|f\|_{L^{p_m}(w_m^{p_m})} 
&=\Big\|\Big(f\widehat{w}^{-\frac{1}{r_m}}\Big)^{r_m}W\Big\|_{L^{\frac{p_m}{r_m}}(\widehat{w})}^{\frac{1}{r_m}}. 
\end{align}
Following the strategy in \cite{LMO} and replacing \cite[(3.7)]{LMO} by \eqref{eq:fW-1}, one can conclude  Theorem \ref{thm:Apq}. The details are left to the reader.  
\end{proof}

Next, let us turn to prove our main theorems. 
\begin{proof}[\textbf{Proof of Theorem $\ref{thm:Ap}$.}] 
Let $\vec{p}=(p_1, \dots, p_m)$ with $\vec{r} \prec \vec{p}$ and $\vec{w}=(w_1, \ldots, w_m) \in A_{\vec{p}, \vec{r}}$. Recall that $\vec{v}=(v_1,\ldots,v_m) \in A_{\vec{q}, \vec{r}}$. Then Lemma \ref{lem:int-Lp} gives that for every $i=1, \ldots, m$,  
\begin{align}\label{eq:Lpsq}
L^{p_*}(w^{p_*}) = [L^{s_*}(u^{s_*}), L^{q_*}(v^{q_*})]_{\theta}, \quad 
L^{p_i}(w_i^{p_i}) = [L^{s_i}(u_i^{s_i}), L^{q_i}(v_i^{q_i})]_{\theta},
\end{align}
for some $\theta \in (0, 1)$, $\vec{s}=(s_1,\ldots,s_m)$ with $\vec{r} \prec \vec{s}$ and $\vec{u} \in A_{\vec{s}, \vec{r}}$. Here, all the exponents are the same as in Lemma \ref{lem:int-Lp}. 

On the other hand, by Theorem \ref{thm:Apq}, the assumption \eqref{eq:Ap-1} implies that 
\begin{align}\label{eq:Apmu}
T \text{ is bounded from $L^{t_1}(\mu_1^{t_1}) \times \cdots \times L^{t_m}(\mu_m^{t_m})$ to $L^{t_*}(\mu^{t_*})$}, 
\end{align}
for all $t_* \in (0, \infty)$, for all $\vec{t}=(t_1, \dots, t_m)$ with $\vec{r} \prec \vec{t}$ and $\frac1t-\frac{1}{t_*}=\frac1q-\frac{1}{q_*}$, and for all $\vec{\mu} \in A_{\vec{t}, \vec{r}}$, where $\frac1t=\frac{1}{t_1}+\cdots+\frac{1}{t_m}$ and $\mu=\prod_{i=1}^m \mu_i$. Hence, \eqref{eq:Apmu} applied to $\vec{u} \in A_{\vec{s}, \vec{r}}$ yields 
\begin{align}\label{eq:Lsu}
T \text{ is bounded from $L^{s_1}(u_1^{s_1}) \times \cdots \times L^{s_m}(u_m^{s_m})$ to $L^{s_*}(u^{s_*})$}. 
\end{align}
In addition, recalling \eqref{eq:Ap-2}, we have 
\begin{align}\label{eq:Lqv}
T \text{ is compact from $L^{q_1}(v_1^{q_1}) \times \cdots \times L^{q_m}(v_m^{q_m})$ to $L^{q_*}(v^{q_*})$}.  
\end{align} 
Observe that $\frac{1}{s_*} \le \frac1s \le \frac{1}{s_*} +\frac{1}{r'_{m+1}}$ and $\frac{1}{q_*} \le \frac1q \le \frac{1}{q_*} +\frac{1}{r'_{m+1}}$. Consequently, from \eqref{eq:Lpsq}, \eqref{eq:Lsu}, \eqref{eq:Lqv} and Corollary \ref{cor:interApr}, we deduce that $T$ is compact from $L^{p_1}(w_1^{p_1}) \times \cdots \times L^{p_m}(w_m^{p_m})$ to $L^{p_*}(w^{p_*})$. The proof is complete. 
\end{proof} 

\begin{proof}[\textbf{Proof of Theorem $\ref{thm:lim}$.}] 
Let $p_i \in (\p_i^{-}, \p_i^{+})$ and $w_i^{p_i} \in A_{\frac{p_i}{\p_i^{-}}} \cap RH_{\big(\frac{\p_i^{+}}{p_i}\big)'}$, $i=1,\ldots,m$. Recall that $v_i^{q_i} \in A_{\frac{q_i}{\p_i^{-}}} \cap RH_{\big(\frac{\p_i^{+}}{q_i}\big)'}$, $i=1,\ldots,m$. Then Lemma \ref{lem:int-lim} gives that for each $i=1, \ldots, m$,  
\begin{align}\label{eq:limLpq}
L^p(w^p) = [L^s(u^s), L^q(v^q)]_{\theta}, \quad \, 
L^{p_i}(w_i^{p_i}) = [L^{s_i}(u_i^{s_i}), L^{q_i}(v_i^{q_i})]_{\theta},
\end{align}
for some $\theta \in (0, 1)$, $\vec{s}=(s_1,\ldots,s_m)$ with $s_i \in (\p_{-}, \p_{+})$ and $u_i^{s_i} \in A_{\frac{s_i}{\p_i^{-}}} \cap RH_{\big(\frac{\p_i^{+}}{s_i}\big)'}$.  

In view of \cite[Theorem~1.3]{CM}, the assumption \eqref{eq:lim-1} yields that 
\begin{align}\label{eq:limmu}
T \text{ is bounded from $L^{q_1}(\mu_1^{q_1}) \times \cdots \times L^{q_m}(\mu_m^{q_m})$ to $L^q(\mu^q)$}, 
\end{align}
for all $q_i \in (\p_i^{-}, \p_i^{+})$ and for all $\mu_i^{q_i} \in A_{\frac{q_i}{\p_i^{-}}} \cap RH_{\big(\frac{\p_i^{+}}{q_i}\big)'}$, $i=1,\ldots,m$, where $\frac1q=\frac{1}{q_1}+\cdots+\frac{1}{q_m}$ and $\mu=\prod_{i=1}^m \mu_i$. From \eqref{eq:limmu} and $u_i^{s_i} \in A_{\frac{s_i}{\p_i^{-}}} \cap RH_{\big(\frac{\p_i^{+}}{s_i}\big)'}$, $i=1,\ldots,m$, we obtain that 
\begin{align}\label{eq:limLsu}
T \text{ is bounded from $L^{s_1}(u_1^{s_1}) \times \cdots \times L^{s_m}(u_m^{s_m})$ to $L^s(u^s)$}. 
\end{align}
Moreover, \eqref{eq:lim-2} states that 
\begin{align}\label{eq:limLqv}
T \text{ is compact from $L^{q_1}(v_1^{q_1}) \times \cdots \times L^{q_m}(v_m^{q_m})$ to $L^q(v^q)$}.  
\end{align} 
Therefore, by \eqref{eq:limLpq}, \eqref{eq:limLsu}, \eqref{eq:limLqv} and Corollary \ref{cor:interlim}, $T$ is compact from $L^{p_1}(w_1^{p_1}) \times \cdots \times L^{p_m}(w_m^{p_m})$ to $L^p(w^p)$. This shows Theorem \ref{thm:lim}. 
\end{proof} 

\begin{proof}[\textbf{Proof of Theorem $\ref{thm:Tb}$.}] 
Let $T$ be an $m$-linear operator. Let $\vec{q}=(q_1, \ldots, q_m)$ with $\vec{r} \preceq \vec{q}$ be the same as in  \eqref{eq:bT-1}. By the same argument as in the proof of \cite[Proposition~5.1]{LMO}, the hypothesis \eqref{eq:bT-1} implies that 
\begin{align}\label{eq:bT-3}
[T, \b]_{\alpha} \text{ is bounded from $L^{q_1}(u_1^{q_1}) \times \cdots \times L^{q_m}(u_m^{q_m})$ to $L^{q_*}(u^{q_*})$}, 
\end{align}
for all $\vec{u}=(u_1, \ldots, u_m) \in A_{\vec{q}, \vec{r}}$, where $u=\prod_{i=1}^m u_i$. Then, \eqref{eq:bT-3} and \eqref{eq:bT-2} respectively verifies \eqref{eq:Ap-1} and \eqref{eq:Ap-2} with $\vec{v}=(1,\ldots,1)$ for $[T, \b]_{\alpha}$ in place of $T$. Invoking Theorem \ref{thm:Ap}, we conclude Theorem  \ref{thm:Tb}. 
\end{proof} 

\begin{proof}[\textbf{Proof of Theorem $\ref{thm:limTb}$.}] 
Let $T$ be an $m$-linear operator. Let $\vec{q}=(q_1, \ldots, q_m)$ with $q_i \in [\p_i^{-}, \p_i^{+}]$ be the same as in  \eqref{eq:limTb-1}. In view of \cite[Theorem~4.3]{BMMST}, the hypothesis \eqref{eq:limTb-1} gives that 
\begin{align}\label{eq:limTb-3}
[T, \b]_{\alpha} \text{ is bounded from $L^{q_1}(u_1^{q_1}) \times \cdots \times L^{q_m}(u_m^{q_m})$ to $L^q(u^q)$}, 
\end{align}
for all $u_i^{q_i} \in A_{\frac{q_i}{\p_i^{-}}} \cap RH_{\big(\frac{\p_i^{+}}{q_i}\big)'}$, $i=1,\ldots,m$, where $\frac1q=\frac{1}{q_1}+\cdots+\frac{1}{q_m}$ and $u=\prod_{i=1}^m u_i$. Hence, \eqref{eq:limTb-3} and \eqref{eq:limTb-2} respectively verifies \eqref{eq:lim-1} and \eqref{eq:lim-2} with $\vec{v}=(1,\ldots,1)$ for $[T, \b]_{\alpha}$ instead of $T$. As a consequence, Theorem \ref{thm:limTb} follows from Theorem \ref{thm:lim}. 
\end{proof} 

\section{Applications}\label{sec:App}
In this section, we present some applications of compact extrapolation theorems obtained above. More specifically, we will establish the compactness of commutators for several kinds of multilinear operators on the weighted Lebesgue spaces. 

\subsection{Multilinear $\omega$-Calder\'{o}n-Zygmund operators}
Let $\omega: [0, \infty) \to [0, \infty)$ be a modulus of continuity, which means that $\omega$ is increasing, subadditive and $\omega(0)=0$. We say that a function $K:\R^{n(m+1)} \setminus \{x=y_1=\cdots=y_m\} \to \C$ is an $\omega$-Calder\'{o}n-Zygmund kernel, if there exists a constant $A>0$ such that 
\begin{align*}
|K(x,\vec{y})| & \le  \frac{A}{\big(\sum_{j=1}^{m}|x-y_j|\big)^{mn}},
\\ 
|K(x,\vec{y}) - K(x',\vec{y})| & \le \frac{A}{\big(\sum_{j=1}^{m}|x-y_j|\big)^{mn}} 
\omega\bigg(\frac{|x-x'|}{\sum_{j=1}^{m}|x-y_j|}\bigg),
\end{align*}
whenever $|x-x'| \leq \frac12 \max\limits_{1\leq j \leq m}|x-y_j|$, and for each $i=1,\ldots,m$, 
\begin{align*}
|K(x,\vec{y}) - K(x,y_1,\ldots,y'_i,\ldots,y_m)| & \le \frac{A}{\big(\sum_{j=1}^{m}|x-y_j|\big)^{mn}} 
\omega\bigg(\frac{|y_i-y'_i|}{\sum_{j=1}^{m}|x-y_j|}\bigg),
\end{align*}
whenever
$|y_i-y_i'| \leq \frac12 \max\limits_{1\leq j \leq m}|x-y_j|$.

An $m$-linear operator $T: \S(\Rn) \times \cdots \times \S(\Rn) \to \S'(\Rn)$ is called an $\omega$-Calder\'{o}n-Zygmund operator if there exists an $\omega$-Calder\'{o}n-Zygmund kernel $K$ such that 
\begin{align*}
T(\vec{f})(x) =\int_{(\Rn)^m} K(x, \vec{y}) f_1(y_1)\cdots f_m(y_m) d\vec{y},  
\end{align*}
whenever $x \not\in \bigcap_{i=1}^m \supp(f_i)$ and $\vec{f}=(f_1,\ldots,f_m) \in \mathscr{C}_c^{\infty}(\Rn) \times \cdots \times \mathscr{C}_c^{\infty}(\Rn)$, and $T$ can be boundedly extended from $L^{q_1}(\Rn) \times \cdots \times L^{q_m}(\Rn)$ to $L^q(\Rn)$ for some $\frac1q=\frac{1}{q_1}+\cdots+\frac{1}{q_m}$ with $1<q_1,\ldots,q_m<\infty$. 

For a modulus of continuity $\omega$, we say that $\omega$ satisfies the Dini condition (or, $\omega \in \text{Dini}$) if it verifies 
\begin{align*}
\|\omega\|_{{\rm Dini}}:=\int_{0}^{1}\omega(t) \frac{dt}{t}<\infty.
\end{align*} 
An example of Dini condition is $\omega(t)=t^{\delta}$ with $\delta>0$. In this case, an $\omega$-Calder\'{o}n-Zygmund operator $T$ is called a (standard) Calder\'{o}n-Zygmund operator, which was studied by Grafakos and Torres \cite{GT}. For the general $\omega$, the linear $\omega$-CZO was introduced by the third author in \cite{Y}, while it was extended by Maldonado and Naibo \cite{MN} to the bilinear case. 

Now we state the main result of this subsection as follows. 
\begin{theorem}\label{thm:CZO}
Let $T$ be an m-linear $\omega$-Calder\'{o}n-Zygmund operator with $\omega \in \text{Dini}$. If $b \in \CMO$, then for each $j=1,\ldots,m$, $[T, b]_{e_j}$ is compact from $L^{p_1}(w_1^{p_1}) \times \cdots \times L^{p_m}(w_m^{p_m})$ to $L^p(w^p)$ for all $\vec{p}=(p_1,\ldots,p_m)$ with $1<p_1,\ldots,p_m<\infty$, and for all $\vec{w} \in A_{\vec{p}}$, where $\frac1p=\frac{1}{p_1}+\cdots+\frac{1}{p_m}$ and $w=\prod_{i=1}^m w_i$.
\end{theorem}

\begin{remark}
Theorem \ref{thm:CZO} improves the weighted boundedness given in \cite{LZ}, but also the weighted compactness for the bilinear Calder\'{o}n-Zygmund operator in \cite{BDMT2} since $w_i^{p_i} \in A_p$ $(i=1,\ldots,m)$ implies $\vec{w}=(w_1,\ldots,w_m) \in A_{\vec{p}}$. 
\end{remark}

\begin{proof}[\textbf{Proof of Theorem $\ref{thm:CZO}$.}] 
Let $\omega \in \text{Dini}$ and $T$ be an m-linear $\omega$-Calder\'{o}n-Zygmund operator. From \cite[Theorem~1.2]{LZ}, one has 
\begin{align}\label{eq:Tbddw}
T \text{ is bounded from $L^{p_1}(w_1^{p_1}) \times \cdots \times L^{p_m}(w_m^{p_m})$ to $L^p(w^p)$}, 
\end{align}
for all $\vec{p}=(p_1,\ldots,p_m)$ with $1<p_1,\ldots,p_m<\infty$, and for all $\vec{w} \in A_{\vec{p}}$, where $\frac1p=\frac{1}{p_1}+\cdots+\frac{1}{p_m}$ and $w=\prod_{i=1}^m w_i$. Thus, Theorem \ref{thm:CZO} will follow from Theorem \ref{thm:Tb} for $\vec{r}=(1,\ldots,1)$ and the fact that 
\begin{align}\label{eq:wTbdd}
[T, b]_{e_j} \text{ is compact from $L^{p_1}(\Rn) \times \cdots \times L^{p_m}(\Rn)$ to $L^p(\Rn)$}, 
\end{align}
for all $\frac1p=\frac{1}{p_1}+\cdots+\frac{1}{p_m}$ with $1<p_1,\ldots,p_m<\infty$. 

It remains to demonstrate \eqref{eq:wTbdd}. Fix $\frac1p=\frac{1}{p_1}+\cdots+\frac{1}{p_m}$ with $1<p_1,\ldots,p_m<\infty$. We first note that 
\begin{align}\label{eq:wTBMO}
\|[T, b]_{e_j}\|_{L^p(\Rn)} \lesssim \|b\|_{\BMO} \prod_{i=1}^m \|f_i\|_{L^{p_i}(\Rn)}, 
\end{align}
for all $b \in \BMO$. This is contained in \cite[Theorem~1.3]{LZ}. Applying Theorem \ref{thm:FK-1}, \eqref{eq:wTBMO} and the fact that $\mathscr{C}^{\infty }_c(\Rn)$ is dense in $\CMO$,  we are reduced to showing that for any $b\in \mathscr{C}_c^{\infty }(\Rn)$, the following two conditions hold:
\begin{enumerate}
\item[(a)] Given $\varepsilon>0$, there exists an $A=A(\varepsilon)>0$ independent of $\vec{f}$ such that 
\begin{align}\label{eq:FK-1}
\|[T,b]_{e_j}(\vec{f})\mathbf{1}_{\{|x|>A\}}\|_{L^p(\Rn)} \lesssim \varepsilon \prod_{j=1}^m\|f_j\|_{L^{p_j}(\Rn)}. 
\end{align}
\item[(b)] Given $\varepsilon\in (0,1)$, there exists a sufficiently small $\delta_0=\delta_0(\varepsilon)$ independent of $\vec{f}$ such that for all $0<|h|<\delta_0$, 
\begin{align}\label{eq:FK-2}
\|\tau_h [T,b]_{e_j}(\vec{f}) - [T,b]_{e_j}(\vec{f})\|_{L^p(\Rn)}\lesssim \varepsilon \prod_{j=1}^m\|f_j\|_{L^{p_j}(\Rn)}. 
\end{align}
\end{enumerate}
The proof of \eqref{eq:FK-1} is just an application of size condition, or see \cite{BT} for details. We are going to deal with \eqref{eq:FK-2}. We only focus on the case $j=1$. Let $\varepsilon\in (0,1)$. Since $\omega \in {\text Dini}$, there exists $t_0=t_0(\varepsilon) \in (0, 1)$ small enough such that 
\begin{align}\label{eq:wte}
\int_{0}^{t_0} \omega(t)\frac{dt}{t} < \varepsilon.
\end{align}
For $\delta>0$ chosen later and $0<|h|<\frac{\delta}{4}$, we split 
\begin{align}\label{eq:IIII}
&[T,b]_{e_1}(\vec{f})(x+h)-[T,b]_{e_1}(\vec{f})(x)
\\ \nonumber
&=(b(x+h)-b(x)) \int_{\sum_{i=1}^m |x-y_i|>\delta} K(x,\vec{y}) \prod^{m}_{j=1} f_j (y_j)d\vec{y}
\\ \nonumber
&\qquad+\int_{\sum_{i=1}^m |x-y_i|>\delta} (K(x+h,\vec{y})-K(x,\vec{y}))(b(x+h)-b(y_1))\prod^{m}_{j=1}f_{j}(y_{j})d\vec{y}
\\ \nonumber
&\qquad+\int_{\sum_{i=1}^m |x-y_i|\le \delta} K(x,\vec{y})(b(y_{1})-b(x))\prod^{m}_{j=1}f_{j}(y_{j})d\vec{y}
\\ \nonumber
&\qquad+\int_{\sum_{i=1}^m |x-y_i| \le \delta} K(x+h,\vec{y})(b(x+h)-b(y_1))\prod^{m}_{j=1}f_{j}(y_{j})d\vec{y}
\\ \nonumber
&=:I_{1}+I_{2}+I_{3}+I_{4}.
\end{align}
We will bound $I_1$, $I_2$, $I_3$ and $I_4$ separately.   

Let  $T_{*}$ be the maximal truncated m-linear $\omega$-Calder\'{o}n-Zygmund operator defined by 
\begin{align*}
T_{*}(\vec{f})(x)=\sup_{\delta>0}\left|\int_{\sum_{i=1}^m |x-y_i|>\delta} K(x,\vec{y})\prod^{m}_{j=1}f_j(y_j)d\vec{y}\right|
\end{align*}
By the size condition, one has 
\begin{align}\label{eq:II-1}
I_1 \lesssim |h|\|\nabla b\|_{L^{\infty}(\Rn)} T_{*}(\vec{f})(x) 
\lesssim \delta \|\nabla b\|_{L^{\infty}(\Rn)} T_{*}(\vec{f})(x).
\end{align}
For $I_2$, the smooth condition gives that 
\begin{align}\label{eq:II-2}
I_2 &\lesssim \|b\|_{L^{\infty}(\Rn)} \int_{\sum_{i=1}^m |x-y_i|>\delta} 
\frac{\prod^{m}_{j=1}|f_j(y_j)|}{(\sum^{m}_{j=1}|x-y_j|)^{mn}} \omega\bigg(\frac{|h|}{\sum^{m}_{j=1}|x-y_{j}|}\bigg)d\vec{y}
\\ \nonumber
&\lesssim \int_{\max\limits_{1\le i\le m}\{|x-y_i|\}>\delta/2}\frac{\prod^{m}_{j=1}|f_{j}(y_{j})|}{(\sum_{j=1}^{m} |x-y_j|)^{mn}}\omega\bigg(\frac{|h|}{\sum_{j=1}^{m} |x-y_j|}\bigg)d\vec{y}
\\ \nonumber
&=\sum_{k=0}^{\infty} \int_{2^{k-1}\delta<\max\limits_{1\le i \le m}\{|x-y_i|\}\le 2^k\delta} \frac{\prod^{m}_{j=1}|f_j(y_j)|}{(\sum_{j=1}^{m} |x-y_j|)^{mn}} \omega\bigg(\frac{|h|}{\sum_{j=1}^{m} |x-y_j|}\bigg)d\vec{y}
\\ \nonumber
&\lesssim \sum^{\infty}_{k=0}\omega\left(\frac{|h|}{2^{k-1}\delta}\right) 
\prod_{j=1}^{m} \fint_{B(x,2^k\delta)} |f_j(y_j)|dy_j
\lesssim \int_{0}^{\frac{4|h|}{\delta}} \omega(t)\frac{dt}{t} \, \mathcal{M}(\vec{f})(x). 
\end{align}
To control $I_3$, we use the size condition: 
\begin{align}\label{eq:II-3}
I_3 & \lesssim \|\nabla b\|_{L^{\infty}(\Rn)} \int_{\sum_{i=1}^m |x-y_i|<\delta} 
\frac{\prod_{j=1}^{m} |f_j(y_j)|}{(\sum_{j=1}^m |x-y_j|)^{mn-1}}d\vec{y}
\\ \nonumber
&\lesssim \sum_{k=0}^{\infty} \int_{2^{-k-1} \delta \leq\sum_{i=1}^m |x-y_i|<2^{-k}\delta} 
\frac{\prod_{j=1}^{m} |f_j(y_j)|}{(\sum_{j=1}^{m}|x-y_{j}|)^{mn-1}}d\vec{y}
\\ \nonumber
&\lesssim \sum_{k=0}^{\infty} 2^{-k}\delta \prod_{j=1}^{m} \fint_{B(x,2^{-j}\delta)}|f_j(y_j)|dy_j
\lesssim \delta \mathcal{M}(\vec{f})(x).
\end{align}
Since $\sum_{i=1}^m |x-y_i| \le \delta$ implies $\sum_{i=1}^m |x+h-y_i| \le \delta+m|h|$, the same argument as $I_3$ leads 
\begin{align}\label{eq:II-4}
I_4 \lesssim (\delta+m|h|) \mathcal{M}(\vec{f})(x+h) \lesssim \delta \mathcal{M}(\vec{f})(x+h). 
\end{align}
Note that by \cite[Theorem~3.7]{LOPTT} and \cite[Theorem~3.6]{DHL}, $T_{*}$ and $\mathcal{M}$ are bounded from $L^{r_1}(\Rn) \times \cdots \times L^{r_m}(\Rn)$ to $L^r(\Rn)$ for all $\frac1r=\frac{1}{r_1}+\cdots+\frac{1}{r_m}$ with $1<r_1,\ldots,r_m<\infty$. Choose $\delta_0 \in (0, \varepsilon t_0)$ and $\delta=\frac{4\delta_0}{t_0}$. Then, gathering \eqref{eq:IIII}--\eqref{eq:II-4}, we deduce that for any $0<|h|<\delta_0$, 
\begin{align*}
\|\tau_h [T,b]_{e_1}(\vec{f}) -& [T,b]_{e_1}(\vec{f})\|_{L^p(\Rn)} 
\lesssim \bigg(\delta+\int_{0}^{\frac{4|h|}{\delta}} \omega(t)\frac{dt}{t}\bigg) \prod_{j=1}^m \|f_j\|_{L^{p_j}(\Rn)} 
\\
&\lesssim \bigg(\frac{\delta_0}{t_0}+\int_{0}^{t_0} \omega(t)\frac{dt}{t}\bigg) \prod_{j=1}^m \|f_j\|_{L^{p_j}(\Rn)} 
\lesssim \varepsilon \prod_{j=1}^m \|f_j\|_{L^{p_j}(\Rn)}, 
\end{align*}
where \eqref{eq:wte} was used in the last inequality.  This shows \eqref{eq:FK-2} and completes the proof.  
\end{proof} 

The rest of this subsection is devoted to presenting some examples, which lie in the category of $m$-linear $\omega$-Calder\'{o}n-Zygmund operators. Given $r \in \R$ and $\rho,\delta \in [0, 1]$, we say $\sigma \in S_{\rho,\delta}^r(n,m)$ if for each triple of multi-indices
$\alpha$ and $\beta=(\beta_1,\ldots,\beta_m)$ there exists a constant $C_{\alpha,\beta}$ such that
\begin{align*}
\big|\partial_{x}^{\alpha} \partial_{\xi_1}^{\beta_1} \cdots \partial_{\xi_m}^{\beta_m}  \sigma(x,\vec{\xi}) \big|
\leq C_{\alpha,\beta} \Big(1+\sum_{i=1}^m |\xi_i| \Big)^{r-\rho \sum_{j=1}^m |\beta_j| + \delta|\alpha|}.
\end{align*}
For $r \in \R$, $\rho \in [0, 1]$ and $\Omega:[0, \infty) \to [0, \infty)$, we say $\sigma \in S_{\rho,\omega, \Omega}^r(n,m)$ if for each multi-indix $\beta=(\beta_1,\ldots,\beta_m)$ there exists a constant $C_{\beta}$ such that
\begin{align*}
\big|\partial_{\xi_1}^{\beta_1} \cdots \partial_{\xi_m}^{\beta_m}  \sigma(x,\vec{\xi}) \big|
&\leq C_{\beta} \Big(1+\sum_{i=1}^m |\xi_i|\Big)^{r-\rho \sum_{j=1}^m |\beta_j|}, 
\\
\big|\partial_{\xi_1}^{\beta_1} \cdots \partial_{\xi_m}^{\beta_m}  (\sigma(x,\vec{\xi})-\sigma(x',\vec{\xi}))\big|
&\leq C_{\beta} \omega(|x-x'|) \Omega\Big(\sum_{i=1}^m |\xi_i|\Big) \Big(1+\sum_{i=1}^m |\xi_i|\Big)^{r-\rho \sum_{j=1}^m |\beta_j|}, 
\end{align*}
for all $x, x' \in \Rn$ and $\vec{\xi} \in \R^{nm}$. 

Given a symbol $\sigma$, the $m$-linear pseudo-differential operators $T_{\sigma}$ is defined by
\begin{align*}
T_{\sigma}(\vec{f})(x) := \int_{(\Rn)^m} \sigma(x,\vec{\xi})
e^{2\pi i x \cdot (\xi_1+\cdots+\xi_m)} \widehat{f}_1(\xi_1)
\cdots \widehat{f}_m(\xi_m) d\vec{\xi},
\end{align*}
for all $\vec{f}=(f_1,\ldots,f_m) \in \S(\Rn) \times \cdots \times \S(\Rn)$, where $\widehat{f}$ denotes the Fourier transform of $f$.

From \cite[Theorem~1]{BO}, we see that for any $\sigma \in S_{1,0}^1(n, 2)$ and for each $i=1,2$, $[T_{\sigma}, a]_{e_i}$ is a bilinear Calder\'{o}n-Zygmund operator, where $a$ is a Lipschitz function such that $\nabla a \in L^{\infty}(\Rn)$. Using this fact and Theorem \ref{thm:CZO}, we obtain an extension of \cite[Theorem~2]{BO} to the weighted spaces and the case $p<1$ as follows. 
\begin{theorem}
Let $\sigma \in S_{1,0}^1(n, 2)$ and $a$ be a Lipschitz function such that $\nabla a \in L^{\infty}(\Rn)$. If $b\in \CMO$,  then for all $i, j=1,2$, $[[T_{\sigma}, a]_i, b]_j$ is compact from $L^{p_1}(w_1^{p_1}) \times L^{p_2}(w_2^{p_2})$ to $L^p(w^p)$ for all $\vec{p}=(p_1, p_2)$ with $1<p_1,p_2<\infty$ and for all $\vec{w}=(w_1,w_2) \in A_{\vec{p}}$, where $\frac1p=\frac{1}{p_1}+\frac{1}{p_2}$ and $w=w_1 w_2$.
\end{theorem}

Suppose that there exists $a \in (0, 1)$ such that 
\begin{align}\label{eq:wa} 
\sup_{0<t<1} \omega(t)^{1-a} \Omega(1/t)<\infty. 
\end{align}
If in addition it is assumed that $\sigma \in S_{1,\omega, \Omega}^0(n, 2)$, \cite[Theorem~4.1]{MN} asserts that $T_{\sigma}$ is a bilinear $\omega^a$-Caldr\'{o}n-Zygmund operator. Hence, this and Theorem \ref{thm:CZO} imply the following. 

\begin{theorem}
Let $\omega, \Omega:[0, \infty) \to [0, \infty)$ be nondecreasing functions with $\omega$ concave. Assume that $\sigma \in S_{1,\omega, \Omega}^0(n, 2)$, and $\omega$ satisfies \eqref{eq:wa} and $\omega^a \in {\text Dini}$. If $b\in \CMO$,  then for each $j=1,2$, $[T_{\sigma}, b]_j$ is compact from $L^{p_1}(w_1^{p_1}) \times L^{p_2}(w_2^{p_2})$ to $L^p(w^p)$ for all $\vec{p}=(p_1, p_2)$ with $1<p_1,p_2<\infty$ and for all $\vec{w}=(w_1,w_2) \in A_{\vec{p}}$, where $\frac1p=\frac{1}{p_1}+\frac{1}{p_2}$ and $w=w_1 w_2$.
\end{theorem}

Let $\omega:[0, \infty) \to [0, \infty)$ be a nondecreasing and concave function. Given a dyadic cube $Q$, a function $\phi_Q: \Rn \to \C$ is called an $\omega$-molecule associated to $Q$ if for some $N>10n$, it satisfies the decay condition 
\begin{align*}
|\phi_Q(x)| \le \frac{A \cdot 2^{kn/2}}{(1+2^k|x-x_Q|)^N},\quad \forall x \in \Rn, 
\end{align*}
and the regularity condition
\begin{align*}
|\phi_Q(x)-\phi_Q(y)| \le A \bigg(\frac{2^{kn/2} \omega(2^k|x-y|) }{(1+2^k|x-c_Q|)^N} 
+ \frac{2^{kn/2} \omega(2^k|x-y|)}{(1+2^k|y-c_Q|)^N}\bigg), \quad \forall x, y \in \Rn, 
\end{align*}
where $\ell(Q)=2^{-k}$ and $c_Q$ is lower left-corner of $Q$. 
 
Given three families of $\omega$-molecules $\{\phi_Q^{i}\}_{Q\in \D}$, $i=1,2,3$, we define the para-product $\Pi_{\D}$ by
\begin{align*}
\Pi_{\D}(\vec{f}) := \sum_{Q \in \D} |Q|^{-\frac12} \langle f_1, \phi_Q^1 \rangle \langle f_2, \phi_Q^2 \rangle \phi_Q^3, 
\end{align*}
for all $\vec{f}=(f_1,\ldots,f_m) \in \S(\Rn) \times \cdots \times \S(\Rn)$. It was proved in \cite[Theorem~5.3]{MN} that $\Pi_{\D}$ is a bilinear $\widetilde{\omega}$-Calder\'{o}n-Zygmund operator, where $\widetilde{\omega}(t):=A^3A_N \omega(C_N t)$ for some positive constants $A_N$ and $C_N$. Observe that $\omega \in {\text Dini}$ implies $\widetilde{\omega} \in {\text Dini}$. As a consequence, together with Theorem \ref{thm:CZO}, these facts yield the weighted compactness of $[\Pi_{\D}, b]_j$ below. 

\begin{theorem}
Let $\omega$ be concave with $\omega \in {\text Dini}$, and $\{\phi_Q^j\}_{Q \in \D}$, $j=1, 2, 3$, be three families of $\omega$-molecules with decay $N>10n$ and such that at least two of them enjoy the cancellation property. If $b\in \CMO$,  then for each $j=1,2$, $[\Pi_{\D}, b]_j$ is compact from $L^{p_1}(w_1^{p_1}) \times L^{p_2}(w_2^{p_2})$ to $L^p(w^p)$ for all $\vec{p}=(p_1, p_2)$ with $1<p_1,p_2<\infty$ and for all $\vec{w}=(w_1,w_2) \in A_{\vec{p}}$, where $\frac1p=\frac{1}{p_1}+\frac{1}{p_2}$ and $w=w_1 w_2$.
\end{theorem}

\subsection{Multilinear fractional integral operators}
Given $0<\alpha<mn$, we define the multilinear fractional integral operator $\I_{\alpha}$ by 
\begin{align*}
\I_{\alpha}(\vec{f})(x) := \int_{\R^{nm}} 
\frac{\prod_{i=1}^m f_i(x-y_i)}{|(y_1,\ldots,y_m)|^{mn-\alpha}} d\vec{y}. 
\end{align*}
From \cite{CT}, one has that for each $j=1,\ldots,m$,  $b \in \CMO$ if and only if 
\begin{align}\label{eq:Ia-1}
[b, \I_{\alpha}]_{e_j} \text{ is compact from } L^{p_1}(\Rn) \times \cdots \times L^{p_m}(\Rn) \text{ to } L^q(\Rn), 
\end{align}
for all $\frac{1}{q}=\frac{1}{p_1}+\cdots+\frac{1}{p_m}-\frac{\alpha}{n}$ with $1<p, p_1,\ldots,p_m<\infty$. On the other hand, it follows from \cite[Theorem~3.5]{M} that 
\begin{align}\label{eq:Ia-2} 
\I_{\alpha} \text{ is bounded from $L^{p_1}(w_1^{p_1}) \times \cdots \times L^{p_m}(w_m^{p_m})$ to $L^q(w^q)$}, 
\end{align}
for all $\frac1q=\frac{1}{p_1}+\cdots+\frac{1}{p_m}-\frac{\alpha}{n}$ with $1<p_1,\ldots,p_m<\infty$, and for all $\vec{w}=(w_1,\ldots,w_m) \in A_{\vec{p}, q}$, where $w=\prod_{i=1}^m w_i$. Thus, \eqref{eq:Ia-1}, \eqref{eq:Ia-2} and Theorem  \ref{thm:Tb} give the following. 

\begin{theorem}\label{thm:Ia}
Let $0<\alpha<mn$. If $b \in \CMO$,  then for each $j=1,\ldots,m$, $[b, \I_{\alpha}]_{e_j}$ is compact from $L^{p_1}(w_1^{p_1}) \times \cdots \times L^{p_m}(w_m^{p_m})$ to $L^q(w^q)$ for all $\frac{1}{q}=\frac{1}{p_1}+\cdots+\frac{1}{p_m}-\frac{\alpha}{n}$ with $1<p_1,\ldots,p_m<\infty$ and for all $\vec{w}=(w_1, \ldots, w_m) \in A_{\vec{p}, q}$, where $w=\prod_{i=1}^m w_i$.  
\end{theorem}

\begin{remark}
Theorem \ref{thm:Ia} improves the conclusion in \cite[Theorem~3.1]{CT}, which obtained the weighted compactness of $\I_{\alpha}$ in the bilinear case and $w_1^{p_1q/p}, w_1^{p_2q/p} \in A_p$ with $1<p,q<\infty$. In fact, the latter is a particular case of $A_{\vec{p}, q}$ classes $($see \cite[Lemma~2.1]{CT}$)$. This also shows the power of extrapolation theorem. 
\end{remark}

Assume that $\Omega_i \in L^{\infty}(\Sn)$ $(i=1,\ldots,m)$ is a homogeneous function with degree zero on $\Rn$, i.e. $\Omega_i(\lambda x)=\Omega_i(x)$ for any $\lambda>0$ and $x \in \Rn$. Given $0<\alpha<mn$, we define the multilinear fractional integral with homogeneous kernels as
\begin{align*}
\I_{\Omega, \alpha}(\vec{f})(x) := \int_{\R^{nm}} 
\frac{\prod_{i=1}^m \Omega_i(x-y_i) f_i(y_i)}{|(x-y_1,\ldots,x-y_m)|^{mn-\alpha}} d\vec{y}. 
\end{align*}

\begin{theorem}\label{thm:IOa}
Let $0<\alpha<mn$ and $\Omega_i \in \Lip(\Sn)$ be homogeneous of degree zero, $i=1,\ldots,m$. If $b \in \CMO$,  then for each $j=1,\ldots,m$, $[b, \I_{\Omega, \alpha}]_{e_j}$ is compact from $L^{p_1}(w_1^{p_1}) \times \cdots \times L^{p_m}(w_m^{p_m})$ to $L^q(w^q)$ for all $\frac{1}{q}=\frac{1}{p_1}+\cdots+\frac{1}{p_m}-\frac{\alpha}{n}$ with $1<p_1,\ldots,p_m<\infty$ and for all $\vec{w}=(w_1, \ldots, w_m) \in A_{\vec{p}, q}$, where $w=\prod_{i=1}^m w_i$.  
\end{theorem}

\begin{proof}
We first claim that 
\begin{align}\label{eq:IOa-1}
\I_{\Omega, \alpha} \text{ is bounded from $L^{p_1}(w_1^{p_1}) \times \cdots \times L^{p_m}(w_m^{p_m})$ to $L^q(w^q)$}, 
\end{align}
for all $\frac1q=\frac{1}{p_1}+\cdots+\frac{1}{p_m}-\frac{\alpha}{n}$ with $1<p_1,\ldots,p_m<\infty$, and for all $\vec{w}=(w_1,\ldots,w_m) \in A_{\vec{p}, q}$, where $w=\prod_{i=1}^m w_i$. Indeed, letting $\vec{w} \in A_{\vec{p}, q}$ and by \eqref{eq:Apqw-1}, we have $w^q \in A_{\infty}$. This gives that $w^q \in RH_r$ for some $r>1$. Then it follows from definition that $\vec{w} \in A_{\vec{p}, qr}$. Note that one can choose $0<\epsilon<\min\{\alpha, mn-\alpha\}$ small enough such that $\frac{1}{q_{\epsilon}}:=\frac{1}{q}-\frac{\epsilon}{n}>\frac{1}{qr}$, that is $q_{\epsilon}<qr$. By definition, it is easy to see that $A_{\vec{p}, qr} \subset A_{\vec{p}, q_{\epsilon}} \subset A_{\vec{p}, q} \subset A_{\vec{p}, q_{-\epsilon}}$, where $\frac{1}{q_{-\epsilon}}:=\frac{1}{q}+\frac{\epsilon}{n}$. Hence, \eqref{eq:IOa-1} follows at once from \cite[Theorem~2.6]{CX} for $\Omega_i \in L^{\infty}(\Sn)$, $i=1,\ldots,m$. 

We next claim that for any $b \in \CMO$, 
\begin{align}\label{eq:IOa-2}
[\I_{\Omega, \alpha}, b]_{e_j} \text{ is compact from $L^{p_1}(\Rn) \times \cdots \times L^{p_m}(\Rn)$ to $L^q(\Rn)$}, 
\end{align}
for all $\frac1q=\frac{1}{p_1}+\cdots+\frac{1}{p_m}-\frac{\alpha}{n}$ with $1 \le q<\infty$ and $1<p_1,\ldots,p_m<\infty$. 
Assuming that \eqref{eq:IOa-2} holds, one can immediately conclude Theorem \ref{thm:IOa} from \eqref{eq:IOa-1} and Theorem \ref{thm:Tb}. Thus, it suffice to demonstrate \eqref{eq:IOa-2}. To proceed, we denote 
\[
K_{\alpha}(x, \vec{y}) :=\frac{\prod_{i=1}^m \Omega_i(x-y_i)}{|(x-y_1,\ldots,x-y_m)|^{mn-\alpha}}. 
\]
Since $\Omega_i \in \Lip(\Sn)$ and it is homogeneous of degree zero, one has that $\Omega_i$ is bounded on $\Rn$. Using these, one can verify the following conditions:  
\begin{itemize}
\item Size condition: ${\displaystyle |K_{\alpha}(x,\vec{y})|  \lesssim  \frac{1}{\big(\sum_{j=1}^{m}|x-y_j|\big)^{mn-\alpha}} }$, 
\item Smoothness conditions: 
\begin{align*}
|K_{\alpha}(x,\vec{y}) - K_{\alpha}(x',\vec{y})| 
& \lesssim \frac{|x-x'|}{(\sum_{j=1}^{m}|x-y_j|)^{mn-\alpha+1}},
\end{align*}
whenever $|x-x'| \leq \frac{1}{2} \sum_{j=1}^m |x-y_j|$, and for each $i=1,\ldots,m$, 
\begin{align*}
|K_{\alpha}(x,\vec{y}) - K_{\alpha}(x,y_1,\ldots,y'_i,\ldots,y_m)| 
& \lesssim \frac{|y_i-y'_i|}{(\sum_{j=1}^{m}|x-y_j|)^{mn-\alpha+1}},
\end{align*}
whenever
$|y_i-y_i'| \leq \frac{1}{2} \sum_{j=1}^m |x-y_j|$. 
\end{itemize}
We here point out that although \cite[Theorem~2.1]{BDMT1} is stated in the bilinear case with $q \ge 1$, by the strategy of \cite{BDMT1} and Theorem \ref{thm:CZO}, it can be extended to in the $m$-linear setting for all exponents $0<q<\infty$. Therefore,  \eqref{eq:IOa-2} is a consequence of these. With \eqref{eq:IOa-1} and \eqref{eq:IOa-2} in hand, we conclude Theorem \ref{thm:IOa} from Theorem \ref{thm:Tb}. 
\end{proof}

\subsection{Multilinear Fourier multipliers} 
For $s \in \N$, a function $\m \in \mathscr{C}^s(\R^{nm} \setminus \{0\})$ is said to belong to $\mathcal{M}^s(\R^{nm})$ if 
\begin{align*}
\big|\partial_{\xi_1}^{\alpha_1} \cdots \partial_{\xi_m}^{\alpha_m}  \m(\vec{\xi}) \big|
&\leq C_{\alpha} (|\xi_1|+\cdots+|\xi_m|)^{-\sum_{i=1}^m |\alpha_i|}, \quad \forall \vec{\xi} \in \R^{nm} \setminus\{0\}, 
\end{align*}
for each multi-indix $\alpha=(\alpha_1,\ldots,\alpha_m)$ with $\sum_{i=1}|\alpha_i| \le s$. 

Given $s \in \R$, the (usual) Sobolev space $W^s(\R^{nm})$ is defined by the norm
\begin{align*} 
\|f\|_{W^s(\R^{nm})} :=\bigg(\int_{\R^{nm}} (1+|\vec{\xi}|^2)^{s}|\widehat{f}(\vec{\xi})|^2 d\vec{\xi}\bigg)^{\frac12}, 
\end{align*}
where $\widehat{f}$ is the Fourier transform in all the variables. For $\vec{s}=(s_1,\ldots,s_m) \in \R^m$, the Sobolev space of product type $W^{\vec{s}}(\R^{nm})$  is defined by 
\begin{align*} 
\|f\|_{W^{\vec{s}}(\R^{nm})} :=\bigg(\int_{\R^{nm}} (1+|\xi_1|^2)^{s_1} \cdots 
(1+|\xi_m|^2)^{s_m} |\widehat{f}(\vec{\xi})|^2 d\vec{\xi}\bigg)^{\frac12}. 
\end{align*}

Let $\Phi \in \S(\R^{nm})$ satisfy $\supp(\Phi) \subset \{(\xi_1, \ldots, \xi_m):\frac12 \le |\xi_1|+\cdots+|\xi_m|\le 2\}$ and $\sum_{j \in \Z} \Phi(2^{-j}\vec{\xi})=1$ for each $\vec{\xi} \in \R^{nm} \setminus\{0\}$. Denote $\m_j(\vec{\xi}) :=\Phi(\vec{\xi}) \m(2^j \vec{\xi})$ for each $j \in \Z$. Denote  
\begin{align*}
\mathcal{W}^s(\R^{nm}) &:= \big\{\m \in L^{\infty}(\R^{nm}): \sup_{j \in \Z} \|\m_j\|_{W^s(\R^{nm})}<\infty\big\}, 
\\
\mathcal{W}^{\vec{s}}(\R^{nm}) &:= \big\{\m \in L^{\infty}(\R^{nm}): \sup_{j \in \Z} \|\m_j\|_{W^{\vec{s}}(\R^{nm})}<\infty\big\}. 
\end{align*}
Then one has 
\begin{align}
\mathcal{M}^{s}(\R^{nm}) \subsetneq \mathcal{W}^{s}(\R^{nm}) 
\subsetneq \mathcal{W}^{(\frac{s}{m},\ldots,\frac{s}{m})}(\R^{nm}). 
\end{align}

Given a symbol $\m$, the $m$-linear Fourier multiplier $T_{\m}$ is defined by
\begin{align*}
T_{\m}(\vec{f})(x) := \int_{(\Rn)^m} \m(\vec{\xi})
e^{2\pi i x \cdot (\xi_1+\cdots+\xi_m)} \widehat{f}_1(\xi_1)
\cdots \widehat{f}_m(\xi_m) d\vec{\xi},
\end{align*}
for all $f_i \in \S(\Rn)$, $i=1,\ldots,m$. 

Let us present a result about the compactness of $T_{\m}$. Indeed, modifying the proof of \cite[Theorem~1.1]{Hu17} to the $m$-linear case, we get that for every $b \in \CMO$ and for each $j=1,\ldots,m$, 
\begin{align}\label{eq:Tm-com}
[T_{\m}, b]_{e_j} \text{ is compact from $L^{p_1}(\Rn) \times \cdots \times L^{p_m}(\Rn)$ to $L^p(\Rn)$}, 
\end{align}
for all $\frac1p=\frac{1}{p_1}+\cdots+\frac{1}{p_m}$ with $1<p<\infty$ and $r_i<p_i<\infty$, $i=1,\ldots,m$, where $\m \in \mathcal{W}^s(\R^{nm})$ with $s \in (mn/2, mn]$, and $\frac{s}{n}=\frac{1}{r_1}+\cdots+\frac{1}{r_m}$ with $1\le r_1,\ldots,r_m<2$. On the other hand, it follows from \cite{Hu14} that \eqref{eq:Tm-com} also holds for all $\frac1p=\frac{1}{p_1}+\cdots+\frac{1}{p_m}$ with $1\le p<\infty$ and $n/s_i=:r_i<p_i< \infty$, $i=1,\ldots,m$, provided $\m \in \mathcal{W}^{\vec{s}}(\R^{nm})$ with $\vec{s}=(s_1,\ldots,s_m)$ and $s_1,\ldots,s_m \in (n/2, n]$. 

We are going to extend \eqref{eq:Tm-com} to the weighted Lebesgue spaces. Let $\m \in \mathcal{W}^s(\R^{nm})$ with $s \in (mn/2, mn]$, and let $\frac{1}{r_1}+\cdots+\frac{1}{r_m}=\frac{s}{n}$ with $1\le r_1,\ldots,r_m<2$. Jiao \cite{J} obtained that for all $\vec{r}:=(r_1,\ldots,r_m,1) \prec \vec{p}$ and for all $\vec{w}=(w_1, \ldots, w_m) \in A_{\vec{p}, \vec{r}}$, 
\begin{align}\label{eq:TmWs}
T_{\m} \text{ is bounded from $L^{p_1}(w_1^{p_1}) \times \cdots \times L^{p_m}(w_m^{p_m})$ to $L^p(w^p)$}, 
\end{align}
where $\frac1p=\frac{1}{p_1}+\cdots+\frac{1}{p_m}$ and $w=\prod_{i=1}^m w_i$. Consequently, using \eqref{eq:Tm-com} with $\m \in \mathcal{W}^s(\R^{nm})$, \eqref{eq:TmWs} and Theorem \ref{thm:Tb}, we conclude the following. 

\begin{theorem}\label{thm:TaWs}
Assume that $\m \in \mathcal{W}^s(\R^{nm})$ with $s \in (mn/2, mn]$. Let $\frac{s}{n}=\frac{1}{r_1}+\cdots+\frac{1}{r_m}$ with $1\le r_1,\ldots,r_m<2$. If $b \in \CMO$, then for each $j=1,\ldots,m$, $[T_{\m}, b]_{e_j}$ is compact from $L^{p_1}(w_1^{p_1}) \times \cdots \times L^{p_m}(w_m^{p_m})$ to $L^p(w^p)$ for all $\frac1p=\frac{1}{p_1}+\cdots+\frac{1}{p_m}$ with $\vec{r} \prec \vec{p}$ and for all $\vec{w}=(w_1, \ldots, w_m) \in A_{\vec{p}, \vec{r}}$, where $\vec{r}=(r_1,\ldots,r_m,1)$ and $w=\prod_{i=1}^m w_i$. 
\end{theorem}

For the general case $\m \in \mathcal{W}^{\vec{s}}(\R^{nm})$ with $s_1,\ldots,s_m \in (n/2, n]$, Fujita and Tomita \cite[Theorem~6.2]{FT12} proved that for all $(w_1^{p_1}, \ldots, w_m^{p_m}) \in A_{p_1/r_1} \times \cdots \times A_{p_m/r_m}$ with $n/s_i=:r_i<p_i< \infty$, $i=1,\ldots,m$, 
\begin{align}\label{eq:TmWss}
T_{\m} \text{ is bounded from $L^{p_1}(w_1^{p_1}) \times \cdots \times L^{p_m}(w_m^{p_m})$ to $L^p(w^p)$}, 
\end{align}
where $\frac1p=\frac{1}{p_1}+\cdots+\frac{1}{p_m}$ and $w=\prod_{i=1}^m w_i$. Accordingly, together with \eqref{eq:Tm-com} applied to $\m \in \mathcal{W}^{\vec{s}}(\R^{nm})$ and \eqref{eq:TmWss}, Theorem \ref{thm:limTb}  with $\p_1^{-}=\cdots=\p_m^{-}=1$ and $\p_1^{+}=\cdots=\p_m^{+}=\infty$ gives the following result.  

\begin{theorem}\label{thm:TaWss}
Assume that $\m \in \mathcal{W}^{\vec{s}}(\R^{nm})$ with $\vec{s}=(s_1,\ldots,s_m)$ and $s_1,\ldots,s_m \in (n/2, n]$. If $b \in \CMO$, then for each $j=1,\ldots,m$, $[T_{\m}, b]_{e_j}$ is compact from $L^{p_1}(w_1^{p_1}) \times \cdots \times L^{p_m}(w_m^{p_m})$ to $L^p(w^p)$ for all $\frac1p=\frac{1}{p_1}+\cdots+\frac{1}{p_m}$ with $r_i<p_i< \infty$, $i=1,\ldots,m$, and for all $(w_1^{p_1}, \ldots, w_m^{p_m}) \in A_{p_1/r_1} \times \cdots \times A_{p_m/r_m}$, where $r_i=n/s_i$ and $w=\prod_{i=1}^m w_i$. 
\end{theorem}

\begin{remark}
By establishing the compactness, Theorem $\ref{thm:TaWs}$ recovers the weighted boundedness of commutators in \cite[Theorem~4.2]{BD} and \cite[Theorem~1.4]{LS}. Also, since $(w_1^{p_1}, \ldots, w_m^{p_m}) \in A_{p_1/r} \times \cdots \times A_{p_m/r}$ implies $\vec{w}=(w_1,\ldots,w_m) \in A_{\vec{p}/r}$, Theorem $\ref{thm:TaWs}$ improves the weighted compactness in \cite[Corollary~4]{ZL}. On the other hand, by enlarging the range of $p$ to the case $p\le 1$, Theorems $\ref{thm:TaWs}$ and $\ref{thm:TaWss}$ respectively refines the compactness on weighted Lebesgue spaces in \cite{Hu17} and \cite[Theorem~2]{ZL}. 

Maybe one would like to seek a better result than Theorems $\ref{thm:TaWs}$ and $\ref{thm:TaWss}$, that is, the weighted compactness holds for the more general case $\m \in \mathcal{W}^{\vec{s}}(\R^{nm})$ and $\vec{w} \in A_{\vec{p}, \vec{r}}$. Unfortunately, this is not true in the general case since the weighted boundedness \eqref{eq:TmWs} does not hold even if $\vec{s}=(\frac{s}{m}, \ldots, \frac{s}{m})$ and $s \in (mn/2, mn]$. This fact can be found in Theorem~$1.1$ and Remark~$3.2$ in \cite{FT14}. 
\end{remark}

\subsection{Higher order Calder\'{o}n commutators} 
In this subsection, we will consider the higher order Calder\'{o}n commutators. Let $A_1,\ldots, A_m$ be functions defined on $\R$ such that $a_j=A'_j$, $j=1,\ldots,m$. Given a function $A$ on $\R$, we define 
\begin{align*}
\mathcal{C}_{m, A}(\vec{a}; f)(x) := \mathrm{p.v. } \int_{\R} \frac{R(A;x,y) 
\prod_{j=1}^{m-1} (A_j(x)-A_j(y))}{(x-y)^{m+1}} f(y) dy, 
\end{align*}
where $R(A; x, y) :=A(x)-A(y)-A'(y)(x-y)$. The operator $\mathcal{C}_{m, A}$ with $a_j \in L^{\infty}(\R)$ was introduced by Cohen \cite{Coh}. When $m=2$, such type operator was introduced by A. Calder\'{o}n \cite{P.C} and then studied by C.  Calder\'{o}n \cite{C.C} and Christ and Journ\'{e} \cite{CJ}. The results for the higher order were also presented in \cite{DGGLY} and \cite{DGY}. 

Using the strategy in \cite{DGY}, we rewrite $\mathcal{C}_{m, A}$ as the following multilinear singular integral operator 
\begin{align}\label{eq:CmA}
\mathcal{C}_{m, A}(\vec{a}; f)(x) = \int_{\R^m} K_A(x, y_1,\ldots,y_m)  \prod_{j=1}^{m-1} a_j(y_j) f(y_m) d\vec{y}, 
\end{align}
where 
\begin{align}
\label{eq:KA-1} K_A(x, y_1,\ldots,y_m) &:= K(x, y_1,\ldots,y_m) \frac{R(A;x,y_m)}{x-y_m}, 
\\
\label{eq:KA-2} K(x, y_1,\ldots,y_m) &:= \frac{(-1)^{(m-1)e(y_m - x)}}{(x-y_m)^m} 
\prod_{j=1}^{m-1} {\bf 1}_{(x \land y_m, x \lor y_m)}(y_j). 
\end{align}
Here, $e(x)={\bf 1}_{(0, \infty)}(x)$, $x \land y=\min{x, y}$ and $x \lor y=\max\{x, y\}$. From \cite{CH}, one has 
\begin{equation}\label{eq:CCk-1}
|K(x, \vec{y})| \lesssim \frac{1}{(\sum_{j=1}^m |x-y_j|)^m}, 
\end{equation}
and 
\begin{align}\label{eq:CCk-2}
|K(x,\vec{y}) - K(x',\vec{y})| \lesssim \frac{|x-x'|}{(\sum_{j=1}^{m}|x-y_j|)^{m+1}},
\end{align}
whenever $|x-x'| \leq \frac{1}{8} \min\limits_{1\le j \le m} |x-y_j|$. 

To generalize $\mathcal{C}_{m, A}$, we define 
\begin{align}\label{eq:CAdef}
\mathscr{C}_{A}(\vec{f})(x) := \int_{\R^m} K_A(x, \vec{y})  \prod_{j=1}^m f_j(y_j) d\vec{y}, 
\end{align}
where the kernel $K_A$ is defined in \eqref{eq:KA-1} and \eqref{eq:KA-2}. 
Denote by $\mathscr{A}(\R)$ the closure of $\mathscr{C}_c^{\infty}(\R)$ in the seminorm $\|A\|_{\BMO_1}:=\|A'\|_{\BMO}$.

\begin{theorem}\label{thm:CC}
Suppose that $A \in \mathscr{A}(\R)$ and $\mathscr{C}_A$ is defined in \eqref{eq:CAdef}. Then $\mathscr{C}_A$ is compact from $L^{p_1}(w_1^{p_1}) \times \cdots \times L^{p_m}(w_m^{p_m})$ to $L^p(w^p)$ for all $\vec{p}=(p_1,\ldots,p_m)$ with $1<p_1,\ldots,p_m<\infty$, and for all $\vec{w} \in A_{\vec{p}}$, where $\frac1p=\frac{1}{p_1}+\cdots+\frac{1}{p_m}$ and $w=\prod_{i=1}^m w_i$.
\end{theorem}

\begin{proof}
It was proved in \cite[Theorem~1.4]{CH} that for any $A' \in \BMO$, 
\begin{align}\label{eq:CC-1}
\|\mathscr{C}_A\|_{L^{p_1}(w_1^{p_1}) \times \cdots \times L^{p_m}(w_m^{p_m}) \to L^p(w^p)} 
\lesssim \|A\|_{\BMO_1} [\vec{w}]_{A_{\vec{p}}}^{\max\limits_{1\le i \le m}\{p, p'_i\}} [w_{m-1}^{-p'_{m-1}}]_{A_{\infty}}, 
\end{align}
for all $\vec{p}=(p_1,\ldots,p_m)$ with $1<p_1,\ldots,p_m<\infty$, and for all $\vec{w} \in A_{\vec{p}}$, where $\frac1p=\frac{1}{p_1}+\cdots+\frac{1}{p_m}$ and $w=\prod_{i=1}^m w_i$. Thus, by Theorem \ref{thm:Ap}, the matters are reduced to showing 
\begin{align}\label{eq:CC-2}
\mathscr{C}_A \text{ is compact from $L^{p_1}(\R) \times \cdots \times L^{p_m}(\R)$ to $L^p(\R)$}, 
\end{align}
for all (or for some) $\frac1p=\frac{1}{p_1}+\cdots+\frac{1}{p_m}$ with $1<p, p_1,\ldots,p_m<\infty$, whenever $A \in \mathscr{A}(\R)$. 

For any $A \in \mathscr{A}(\R)$, there exists a sequence $\{A_j\}_{j \in \N} \subset \mathscr{C}_c^{\infty}(\R)$ such that $\lim\limits_{j \to \infty} \|A_j-A\|_{\BMO_1} = 0$. Then, \eqref{eq:CC-1} gives that 
\begin{align*}
\|\mathscr{C}_{A_j} - \mathscr{C}_A\|_{L^{p_1}(\Rn) \times \cdots \times L^{p_2}(\Rn) \to L^p(\R)} 
&=\|\mathscr{C}_{A_j-A}\|_{L^{p_1}(\Rn) \times \cdots \times L^{p_2}(\Rn) \to L^p(\R)}
\\ 
&\lesssim \|A_j-A\|_{\BMO_1} \to 0,\quad\text{as } j \to 0. 
\end{align*}
Hence, it suffices to prove \eqref{eq:CC-2} for $A \in \mathscr{C}_c^{\infty}(\R)$. In what follows, we assume that $A \in \mathscr{C}_c^{\infty}(\R)$ with $\supp(A) \subset B(0, a_0)$ for some $a_0>1$. By Theorem \ref{thm:FK-1} and \eqref{eq:CC-1}, it is enough to show 
\begin{enumerate}
\item[(i)] Given $\varepsilon>0$, there exists an $a=a(\varepsilon)>0$ independent of $\vec{f}$ such that 
\begin{align}\label{eq:CC-3}
\|\mathscr{C}_A(\vec{f})\mathbf{1}_{\{|x|>a\}}\|_{L^p(\R)} 
\lesssim \varepsilon \prod_{j=1}^m\|f_j\|_{L^{p_j}(\R)}. 
\end{align}
\item[(ii)] Given $\varepsilon\in (0,1)$, there exists a sufficiently small $\delta_0=\delta_0(\varepsilon)$ independent of $\vec{f}$ such that for all $0<|h|<\delta_0$, 
\begin{align}\label{eq:CC-4}
\|\tau_h \mathscr{C}_A(\vec{f}) - \mathscr{C}_A(\vec{f})\|_{L^p(\R)} 
\lesssim \varepsilon \prod_{j=1}^m\|f_j\|_{L^{p_j}(\R)}. 
\end{align}
\end{enumerate}

Let $a>2a_0$ and $|x|>a$. Then $|x-y_m| \simeq |x|$ for any $y_m \in B(0, a_0)$. Note that $(x_1\cdots x_n)^{\frac1n} \le (x_1+\cdots+x_n)/n$ for all $x_1,\ldots,x_n \ge 0$. Using this, \eqref{eq:CCk-1} and H\"{o}lder's inequality, we deduce that 
\begin{align}\label{eq:CC-5}
\bigg|\int_{\R^m} &K(x, \vec{y}) A'(y_m) \prod_{j=1}^m f_j(y_j) d\vec{y} \bigg| 
\\  \nonumber
&\lesssim \|A'\|_{L^{\infty}(\R)} \int_{B(0, a_0)}\int_{\R^{m-1}} \frac{\prod_{j=1}^m |f_j(y_j)|}{(\sum_{i=1}^m |x-y_i|)^m} d\vec{y} 
\\  \nonumber
&\lesssim \int_{B(0, a_0)}\int_{\R^{m-1}} \frac{\prod_{j=1}^m |f_j(y_j)|}{(\sum_{i=1}^m (1+|x-y_i|))^m} d\vec{y} 
\\  \nonumber
&\lesssim \bigg(\prod_{j=1}^{m-1} \int_{\R}  \frac{|f_j(y_j)|}{1+|x-y_j|} dy_j \bigg) 
\int_{B(0, a_0)} \frac{|f_m(y_m)|}{1+|x-y_m|} dy_m 
\\  \nonumber
&\lesssim |x|^{-1} \prod_{j=1}^{m-1} \|f_j\|_{L^{p_j}(\R)} \bigg(\int_{\R}  \frac{dy_j}{(1+|x-y_j|)^{p'_j}}\bigg) 
\|f_m\|_{L^{p_m}(\R)} a_0^{\frac{1}{p'_m}} 
\\  \nonumber
&\lesssim |x|^{-1} \prod_{j=1}^m \|f_j\|_{L^{p_j}(\R)} 
\end{align}
Likewise, for any $\theta \in (0, 1)$, 
\begin{align}\label{eq:CC-6}
\bigg|\int_{\R^m} &K(x, \vec{y}) A'(\theta x +(1-\theta)y_m) \prod_{j=1}^m f_j(y_j) d\vec{y} \bigg| 
\lesssim |x|^{-1} \prod_{j=1}^m \|f_j\|_{L^{p_j}(\R)}.  
\end{align}
By the mean value theorem, there exists some $\theta \in (0, 1)$, 
\begin{equation}\label{eq:RAA}
R(A;x,y_m) = [A'(\theta x +(1-\theta)y_m) - A'(y_m)] (x-y_m). 
\end{equation}
Gathering \eqref{eq:CC-5}, \eqref{eq:CC-6} and \eqref{eq:RAA}, we have 
\begin{equation}\label{eq:CC-7}
|\mathscr{C}_A(\vec{f})(x)| \lesssim |x|^{-1} \prod_{j=1}^m \|f_j\|_{L^{p_j}(\R)},\quad |x|>a.   
\end{equation}
Pick $a>\max\{2a_0, \varepsilon^{-p'}\}$. Thus, \eqref{eq:CC-7} implies \eqref{eq:CC-3}. 

To show \eqref{eq:CC-4}, we may assume that $\|f_j\|_{L^{p_j}(\R)}=1$, $j=1,\ldots,m$. Let $\varepsilon>0$. By \eqref{eq:CC-1}, we choose $\widetilde{f}_m \in \mathscr{C}_c^{\infty}(\R)$ so that 
\begin{equation}\label{eq:fmfm}
\|\mathscr{C}_A(f_1,\ldots,f_{m-1}, f_m-\widetilde{f}_m)\|_{L^p(\R)} < \varepsilon. 
\end{equation}
Then for $\vec{\widetilde{f}}:=(f_1,\ldots,f_{m-1}, \widetilde{f}_m)$, \eqref{eq:fmfm} implies 
\begin{align*}
\|\tau_h \mathscr{C}_A(\vec{f}) - \mathscr{C}_A(\vec{f})\|_{L^p(\R)} 
&\le \|\tau_h \mathscr{C}_A(\vec{f}) - \tau_h \mathscr{C}_A(\vec{\widetilde{f}})\|_{L^p(\R)} 
+ \|\tau_h \mathscr{C}_A(\vec{\widetilde{f}}) - \mathscr{C}_A(\vec{\widetilde{f}})\|_{L^p(\R)} 
\\
&\qquad+ \|\mathscr{C}_A(\vec{\widetilde{f}}) - \mathscr{C}_A(\vec{f})\|_{L^p(\R)} 
\\
&\le 2\varepsilon + \|\tau_h \mathscr{C}_A(\vec{\widetilde{f}}) - \mathscr{C}_A(\vec{\widetilde{f}})\|_{L^p(\R)}. 
\end{align*}
This means that to prove \eqref{eq:CC-4} we may assume that $\supp(f_m) \subset B(0, b_0)$ for some $b_0>0$. 

In order to demonstrate \eqref{eq:CC-4}, we set $\delta>0$ chosen later and $0<|h|<\frac{\delta}{8m}$. Observe that 
\[
K_A(x, \vec{y}) = |K(x, \vec{y})| \frac{R(A; x, y_m)}{|x-y_m|}.
\]
Then, 
\begin{equation}\label{eq:JJJJ}
|\mathscr{C}_A(\vec{f})(x+h) - \mathscr{C}_A(\vec{f})(x)| 
\le J_1 + J_2 + J_3 + J_4, 
\end{equation}
where 
\begin{align*}
J_1 &:= \int_{\sum_{i=1}^m |x-y_i|>\delta} |K(x+h, \vec{y})|
\bigg|\frac{R(A;x+h,y_m)}{|x+h-y_m|} - \frac{R(A;x,y_m)}{|x-y_m|}\bigg| \prod_{j=1}^m |f_j(y_j)|d\vec{y}, 
\\ 
J_2 &:= \int_{\sum_{i=1}^m |x-y_i|>\delta} |K(x+h, \vec{y}) - K(x, \vec{y})| \frac{|R(A;x,y_m)|}{|x-y_m|} \prod_{j=1}^m |f_j(y_j)| d\vec{y}, 
\\ 
J_3 &:= \int_{\sum_{i=1}^m |x-y_i| \le \delta} |K(x, \vec{y})| \frac{|R(A;x,y_m)|}{|x-y_m|} \prod_{j=1}^m |f_j(y_j)| d\vec{y}, 
\\ 
J_4 &:= \int_{\sum_{i=1}^m |x-y_i| \le \delta} |K(x+h, \vec{y})| \frac{|R(A;x+h,y_m)|}{|x+h-y_m|} \prod_{j=1}^m |f_j(y_j)| d\vec{y}. 
\end{align*}
Considering $J_1$, we split $J_1=J_{1,1}+J_{1,2}$, where 
\begin{align*}
J_{1,1} &:= \int_{\substack{\sum_{i=1}^m |x-y_i|>\delta \\ |x-y_m|>\frac{\delta}{m}}} |K(x+h, \vec{y})|
\bigg|\frac{R(A;x+h,y_m)}{|x+h-y_m|} - \frac{R(A;x,y_m)}{|x-y_m|}\bigg| \prod_{j=1}^m |f_j(y_j)|d\vec{y}, 
\\ 
J_{1,2} &:= \int_{\substack{\sum_{i=1}^m |x-y_i|>\delta \\ |x-y_m| \le \frac{\delta}{m}}} |K(x+h, \vec{y})|
\bigg|\frac{R(A;x+h,y_m)}{|x+h-y_m|} - \frac{R(A;x,y_m)}{|x-y_m|}\bigg| \prod_{j=1}^m |f_j(y_j)|d\vec{y}. 
\end{align*}
The condition $|x-y_m|>\frac{\delta}{m}$ implies $|h| < \frac18 |x-y_m|$, and hence, by \eqref{eq:RAA}, 
\begin{align*}
\bigg|\frac{R(A;x+h,y_m)}{|x+h-y_m|} - \frac{R(A;x,y_m)}{|x-y_m|}\bigg| 
&\le \frac{|R(A;x+h,y_m)-R(A;x,y_m)|}{|x+h-y_m|} 
\\
&\quad + |R(A;x,y_m)| \bigg|\frac{1}{|x+h-y_m|} - \frac{1}{|x-y_m|}\bigg|
\\
&\lesssim \|A'\|_{L^{\infty}(\R)} \frac{|h|}{|x-y_m|}. 
\end{align*}
Then, this and \eqref{eq:CCk-1} yield 
\begin{align}\label{eq:J11}
J_{1,1} &\lesssim |h| \int_{\substack{\sum_{i=1}^m |x-y_i|>\delta \\ |x-y_m|>\frac{\delta}{m}}} 
\frac{\prod_{j=1}^{m-1} |f_j(y_j)|}{(\sum_{i=1}^m |x-y_i|)^m} \frac{|f_m(y_m)|}{|x-y_m|} d\vec{y} 
\\  \nonumber
&\lesssim |h| \int_{|x-y_m|>\frac{\delta}{m}} \bigg(\int_{\sum_{i=1}^m |x-y_i|>\delta}  
\frac{\prod_{j=1}^{m-1} |f_j(y_j)|\, dy_j}{(\sum_{i=1}^m |x-y_i|)^{m-\alpha}} \bigg) 
\frac{|f_m(y_m)|}{|x-y_m|^{1+\alpha}} dy_m
\\  \nonumber
&\lesssim |h| \delta^{\alpha-1} \prod_{j=1}^{m-1} Mf_j(x) 
\int_{|x-y_m|>\frac{\delta}{m}} \frac{|f_m(y_m)|}{|x-y_m|^{1+\alpha}} dy_m 
\lesssim \delta^{-1} |h| \prod_{j=1}^m Mf_j(x), 
\end{align}
where $\alpha \in (0, 1)$ is an auxiliary parameter. 
For $J_{1,2}$, we observe that 
\begin{equation}\label{eq:RA2}
R(A;x,y_m) = \frac12 A''(\eta x +(1-\eta)y_m) (x-y_m)^2,\quad \text{for some } \eta \in (0,1). 
\end{equation}
Additionally, the condition $\sum_{i=1}^m |x-y_i|>\delta$ and $|x-y_m| \le \frac{\delta}{m}$ implies that $\sum_{i=1}^m |x+h-y_i| \gtrsim \delta$ and $|x+h-y_m| \lesssim \delta$. Using these and \eqref{eq:CCk-1}, we derive 
\begin{align}\label{eq:J12}
J_{1,2} &\lesssim \int_{\substack{\sum_{i=1}^m |x+h-y_i| \gtrsim \delta \\ |x+h-y_m| \lesssim \delta}} 
|K(x+h, \vec{y})| (|x+h-y_m| + |x-y_m|) \prod_{j=1}^m |f_j(y_j)|d\vec{y} 
\\  \nonumber
&\lesssim \delta \int_{\substack{\sum_{i=1}^m |x+h-y_i| \gtrsim \delta \\ |x+h-y_m| \lesssim \delta}} 
\frac{\prod_{j=1}^m |f_j(y_j)|}{(\sum_{i=1}^m |x+h-y_i|)^m} d\vec{y} 
\lesssim \delta \prod_{j=1}^m Mf_j(x+h). 
\end{align}
Combining \eqref{eq:J11} and \eqref{eq:J12}, we obtain 
\begin{align}\label{eq:J1}
J_1 \lesssim (\delta + \delta^{-1}|h|) \prod_{j=1}^m Mf_j(x). 
\end{align}
To analyze $J_2$, we write 
\begin{align*}
J_{2,1} &:= \int_{\forall i: |x-y_i|>\frac{\delta}{m}} |K(x+h, \vec{y}) - K(x, \vec{y})| \frac{|R(A;x,y_m)|}{|x-y_m|} \prod_{j=1}^m |f_j(y_j)| d\vec{y}, 
\\ 
J_{2,2} &:= \int_{\substack{\sum_{i=1}^m |x-y_i|>\delta \\ \exists i: |x-y_i| \le \frac{\delta}{m}}} 
|K(x+h, \vec{y}) - K(x, \vec{y})| \frac{|R(A;x,y_m)|}{|x-y_m|} \prod_{j=1}^m |f_j(y_j)| d\vec{y}. 
\end{align*}
The estimates \eqref{eq:CCk-2} and \eqref{eq:RAA} lead 
\begin{align}\label{eq:J21}
J_{2,1} & \lesssim |h| \|A'\|_{L^{\infty}(\Rn)} \int_{\sum_{i=1}^m |x-y_i| > \delta} 
\frac{\prod_{j=1}^m |f_j(y_j)|}{(\sum_{i=1}^m |x-y_i|)^{m+1}} d\vec{y} 
\lesssim \delta^{-1} |h| \mathcal{M}(\vec{f})(x). 
\end{align} 
For $J_{2,2}$, we claim that 
\begin{align}\label{eq:J22}
J_{2,2} \lesssim \delta \mathcal{M}(\vec{f})(x+h) + \delta \mathcal{M}(\vec{f})(x).  
\end{align}
Indeed, if the case $|x-y_m| \lesssim \delta$ occurs in $J_{2,2}$,  then the same argument as $J_{1,2}$ yields \eqref{eq:J22}. Now we treat the case $|x-y_m|>N\delta$ for any large number $N$. Then, for any given $\eta \in (0, 1)$, 
\begin{equation}\label{eq:Nab}
|\eta x + (1-\eta)y_m| \ge \eta |x-y_m|- |y_m| \ge N\eta \delta-b_0>a_0, 
\end{equation}
provided that $N$ is large enough. Together with \eqref{eq:RA2} and $\supp(A) \subset B(0, a_0)$, \eqref{eq:Nab} implies that $J_{2,2}=0$, and hence \eqref{eq:J22} holds in this scenario. Collecting \eqref{eq:J21} and \eqref{eq:J22}, one has 
\begin{align}\label{eq:J2}
J_2 \lesssim (\delta + \delta^{-1}|h|) \mathcal{M}(\vec{f})(x) + \delta \mathcal{M}(\vec{f})(x+h). 
\end{align}
As for $J_3$, applying \eqref{eq:RA2} and the same calculation as \eqref{eq:II-3}, we obtain 
\begin{align}\label{eq:J3}
J_3 \lesssim \|A''\|_{L^{\infty}(\R)} \int_{\sum_{i=1}^m |x-y_i| \le \delta} 
\frac{\prod_{j=1}^m |f_j(y_j)|}{(\sum_{i=1}^m |x-y_i|)^{m-1}} d\vec{y} 
\lesssim \delta \mathcal{M}(\vec{f})(x). 
\end{align}
Analogously, 
\begin{align}\label{eq:J4}
J_4 \lesssim (\delta+m|h|) \mathcal{M}(\vec{f})(x+h) \lesssim \delta \mathcal{M}(\vec{f})(x+h). 
\end{align}

In order to conclude \eqref{eq:CC-4}, we pick $\delta=8m\varepsilon^{-1}|h|$ and $\delta_0=\frac{\varepsilon^2}{2(1+\varepsilon)}$ such that $|h|<\frac{\delta}{8m}$ and $\delta_0<\frac{\varepsilon^2}{1+\varepsilon}$. Now, using \eqref{eq:JJJJ}, \eqref{eq:J1}, \eqref{eq:J2}, \eqref{eq:J3} and \eqref{eq:J4}, we obtain that for $0<|h|<\delta_0$, 
\begin{align*}
\|\tau_h \mathscr{C}_A(\vec{f}) - \mathscr{C}_A(\vec{f})\|_{L^p(\R)} 
&\lesssim (\delta + |h| + \delta^{-1}|h|) \prod_{j=1}^m\|f_j\|_{L^{p_j}(\R)}  
\\ 
&=(8m\varepsilon^{-1} +1) |h| + \frac{\varepsilon}{8m}
\lesssim (\varepsilon^{-1} +1) \delta_0 + \varepsilon 
\lesssim \varepsilon. 
\end{align*}
This shows \eqref{eq:CC-4}. 
\end{proof}

\subsection{Bilinear rough singular integrals}  
Given $\Omega \in L^q(\mathbb{S}^{2n-1})$ with $1 \le q \le \infty$ and $\int_{\mathbb{S}^{2n-1}} \Omega\, d\sigma=0$,  we define the rough bilinear singular integral operator $T_{\Omega}$ by 
\begin{equation*}
T_{\Omega}(f, g)(x)=\mathrm{p.v.} \int_{\R^{2n}} K_{\Omega}(x-y, x-z) f(y) g(z) dy dz,  
\end{equation*}
where the rough kernel is given by 
\begin{equation*}
K_{\Omega}(y, z) = \frac{\Omega\left((y, z)/\left|(y, z)\right|\right)}{\left|(y, z)\right|^{2n}}. 
\end{equation*}

A typical example of the rough bilinear operators is the Calder\'{o}n commutator defined in \cite{P.C} as 
\begin{align*}
\mathcal{C}(a, f)(x):= \text{p.v.} \int_{\R} \frac{A(x)-A(y)}{|x-y|^2} f(y) dy, 
\end{align*}
where $a$ is the derivative of $A$. The boundedness of $\mathcal{C}(a, f)$ in the full range of exponents 
$1<p_1, p_2< \infty$ was established in \cite{C.C}. It was shown in \cite{P.C} that the Calder\'{o}n 
commutator can be written as 
\begin{align*}
\mathcal{C}(a, f)(x):= \text{p.v.} \int_{\R \times \R} K(x-y, x-z) f(y) a(z) dydz,  
\end{align*}
with the kernel 
\begin{align*}
K(y, z)=\frac{e(z)-e(z-y)}{y^2}=\frac{\Omega((y, z)/|(y,z)|)}{|(y, z)|^2}, 
\end{align*}
where $e(t)=1$ if $t>0$ and $e(t)=0$ if $t<0$. Observe that $K(y, z)$ is odd and homogeneous of degree $-2$ whose restriction on $\mathbb{S}^1$ is $\Omega(y,z)$. It is also easy to check that $\Omega$ is odd, bounded and thus Theorem \ref{thm:rough} below can be applied to $\mathcal{C}(a, f)$.

\begin{theorem}\label{thm:rough}
Let $\Omega \in L^q(\mathbb{S}^{2n-1})$ with $\frac43<q\le \infty$ and $\int_{\mathbb{S}^{2n-1}} \Omega\, d\sigma=0$. Let $\vec{r}=(r_1,r_2,r_3)$ with $r_1=r_2=r_3=1$ if $q=\infty$, $\max\big\{\frac{24n+3q-4}{8n+3q-4}, \frac{24n+q}{8n+q}\big\} <r_1, r_2, r_3<3$ if $q<\infty$. Then for each $k=1,2$ and $b \in \CMO$, $[T_{\Omega}, b]_{e_k}$ is compact from $L^{p_1}(w_1^{p_1}) \times L^{p_2}(w_2^{p_2})$ to $L^p(w^p)$ for all $\vec{p}=(p_1, p_2)$ with $\vec{r} \prec \vec{p}$ and for all $\vec{w}=(w_1, w_2) \in A_{\vec{p}, \vec{r}}$, where $\frac1p=\frac{1}{p_1}+\frac{1}{p_2}$ and $w=w_1w_2$. 
\end{theorem}

\begin{proof}
It was proved in \cite{CHS} that if $\Omega \in L^{\infty}(\mathbb{S}^{2n-1})$, then for every $w=(w_1, w_2) \in A_{(2,2)}$, 
\begin{align}\label{eq:TO-1} 
T_{\Omega}: L^2(w_1^2) \times L^2(w_2^2) \to L^1(w).  
\end{align}
For $\Omega \in L^q(\mathbb{S}^{2n-1})$ with $\frac43<q<\infty$, Grafakos et al. \cite{GWX} obtained that 
\begin{align}\label{eq:TO-2} 
T_{\Omega}: L^{p_1}(w_1^{p_1}) \times L^{p_2}(w_2^{p_2}) \to L^p(w^p), 
\end{align}
for all $\frac1p=\frac{1}{p_1}+\frac{1}{p_2}$ with $\vec{r} \prec \vec{p}$ and $1<p<\infty$ and for all $\vec{w}=(w_1, w_2) \in A_{\vec{p}, \vec{r}}$. Therefore, Theorem \ref{thm:rough} follows from Theorem \ref{thm:Tb}, \eqref{eq:TO-1}, \eqref{eq:TO-2} and that 
\begin{align}
\label{eq:TO-3} &[T_{\Omega}, b]_{e_k} \text{ is compact from $L^2(\Rn) \times L^2(\Rn)$ to $L^1(\Rn)$},  
\quad\text{if } q=\infty, 
\\
\label{eq:TO-4} & [T_{\Omega}, b]_{e_k} \text{ is compact from $L^3(\Rn) \times L^3(\Rn)$ to $L^{\frac32}(\Rn)$},  
\quad\text{if } q<\infty. 
\end{align}

Next, let us demonstrate \eqref{eq:TO-3} and \eqref{eq:TO-4}. Fix $k \in \{1,2\}$ and $b \in \CMO$. Let $\frac43<q\le \infty$ and $\Omega \in L^q(\mathbb{S}^{2n-1})$ with mean value zero. Pick a smooth function $\alpha$ in $\R^+$ such that $\alpha(t)=1$ for $t \in (0, 1]$, $0<\alpha(t)<1$ for $t \in (1, 2)$ and $\alpha(t)=0$ for $t \ge 2$. For $(y, z) \in \R^{2n}$ and $j \in \Z$ we introduce the function 
\[ 
\beta_j(y, z) = \alpha(2^{-j}|(y, z)|) - \alpha(2^{-j+1}|(y, z)|). 
\] 
We write $\beta:=\beta_0$, which is supported in $[1/2, 2]$. We denote $\Delta_j$ the Littlewood-Paley operator $\widehat{\Delta_j f} = \beta_ j f$. We decompose the kernel $K_{\Omega}$ as follows. Set 
\[
K^i := \beta_i K_{\Omega} \quad\text{and}\quad K_j^i := \Delta_{j-i} K^i, \quad\forall \, i, j \in \Z.
\] 
Then we decompose the kernel 
\begin{align*}
K_{\Omega} = \sum_{j \in \Z} K_j \quad\text{and}\quad K_j = \sum_{i \in \Z} K_j^i, 
\end{align*}
and hence, the operator $T_{\Omega}$ can be written as 
\begin{align*}
T_{\Omega}(f, g)(x) =\sum_{j \in \Z} \int_{\R^{2n}} K_j(x-y, x-z) f(y) g(z)dydz =: \sum_{j \in \Z}T_j(f, g)(x). 
\end{align*}

We first deal with the case $q=\infty$. By means of \cite[Theorem~2.22]{LMO}, \eqref{eq:TO-1} gives that 
\begin{align}\label{eq:Tj-1}
\|[T_{\Omega}, b]_{e_k}\|_{L^2(\Rn) \times L^2(\Rn) \to L^1(\Rn)} \lesssim \|b\|_{\BMO}.   
\end{align}
Additionally, it follows from Proposition~5 and Lemma~11 in \cite{GHH} that 
\begin{align}\label{eq:Tj-2}
T_j \text{ is a bilinear Calder\'{o}n-Zygmund operator}, \quad \forall j \in \Z, 
\end{align}
and 
\begin{align}\label{eq:Tj-3}
\|T_j\|_{L^2(\Rn) \times L^2(\Rn) \to L^1(\Rn)} 
\lesssim 2^{-|j| \delta} \|\Omega\|_{L^{\infty}(\mathbb{S}^{2n-1})},  \quad \forall j \in \Z,
\end{align}
where $\delta>0$ is a fixed constant. Then, Theorem \ref{thm:CZO} and \eqref{eq:Tj-2} imply that 
\begin{align}\label{eq:Tj-4}
[T_j, b]_{e_k} \text{ is compact from $L^2(\Rn) \times L^2(\Rn)$ to $L^1(\Rn)$}, \quad \forall j \in \Z. 
\end{align}
Consequently, \eqref{eq:TO-3} immediately follows from \eqref{eq:Tj-1}, \eqref{eq:Tj-3}, \eqref{eq:Tj-4} and Lemma \ref{lem:bTTj}.   

It remains to handle the case $q<\infty$. Invoking \cite[Theorem~2.22]{LMO} and \eqref{eq:TO-2}, we have 
\begin{align}\label{eq:Tj-5}
\|[T_{\Omega}, b]_{e_k}\|_{L^{p_1}(\Rn) \times L^{p_2}(\Rn) \to L^p(\Rn)} \lesssim \|b\|_{\BMO},    
\end{align}
for all $\frac1p=\frac{1}{p_1}+\frac{1}{p_2}$ with $\vec{r} \prec \vec{p}$. On the other hand, it was proved in \cite[Lemmas~3.1, 4.3]{GWX} that \eqref{eq:Tj-2} holds and 
\begin{align}\label{eq:Tj-6}
\|T_j\|_{L^{p_1}(\Rn) \times L^{p_2}(\Rn) \to L^p(\Rn)} 
\lesssim |j| 2^{-|j| \delta} \|\Omega\|_{L^q(\mathbb{S}^{2n-1})},  \quad \forall j \in \Z,
\end{align}
for all $\frac1p=\frac{1}{p_1}+\frac{1}{p_2}$ with $1\le p \le 2 \le p_1,p_2<\infty$, where $\delta=\delta(q)>0$ is independent of $j$.  By Theorem \ref{thm:CZO} and \eqref{eq:Tj-2} again, 
\begin{align}\label{eq:Tj-7}
[T_j, b]_{e_k} \text{ is compact from $L^{p_1}(\Rn) \times L^{p_2}(\Rn)$ to $L^p(\Rn)$}, 
\end{align}
for all $\frac1p=\frac{1}{p_1}+\frac{1}{p_2}$ with $1<p_1,p_2<\infty$. Therefore, by Lemma \ref{lem:bTTj}, \eqref{eq:TO-4} follows at once from \eqref{eq:Tj-5}, \eqref{eq:Tj-6} and \eqref{eq:Tj-7} for the exponents $p_1=p_2=3$ and $p=\frac32$.    
\end{proof}

\subsection{Bilinear Bochner-Riesz means}  
Given $\alpha>0$, the Bochner-Riesz multiplier $\B^{\alpha}$ is defined by 
\begin{align*}
\widehat{\B^{\alpha}f}(\xi) := (1-|\xi|^2)^{\alpha}_{+} \widehat{f}(\xi), \quad\forall f \in \S(\Rn). 
\end{align*}  
From \cite{BBL}, we see that for $n=2$ and $\alpha>\frac16$, 
\begin{align}\label{eq:BR-1}
\B^{\alpha} \text{ is bounded on } L^p(w^p),\quad \forall p \in [1.2, 2) \text{ and } \forall w^p \in A_{\frac{p}{1.2}} \cap RH_{(\frac2p)'}.
\end{align}
Recently, the compactness of commutators of $\B^{\alpha}$ was also established in \cite{BCH}. Indeed, for $n=2$ and $0<\alpha<\frac12$, 
\begin{align}\label{eq:BR-2}
[\B^{\alpha}, b] \text{ is compact on $L^p(\Rn)$},    
\quad \forall p \in \Big(\frac{4}{3+2\alpha}, \frac{4}{1-2\alpha}\Big). 
\end{align}

Observe that for any $\alpha>0$, 
\begin{align}\label{eq:BR-3}
\frac{4}{3+2\alpha}<\frac65 \quad\Longleftrightarrow\quad \alpha>\frac16, \quad\text{ and }\quad
2<\frac{4}{1-2\alpha} \quad\Longleftrightarrow\quad \alpha<\frac12. 
\end{align}
Thus, combining \eqref{eq:BR-1}, \eqref{eq:BR-2}, \eqref{eq:BR-3},  and Theorem \ref{thm:limTb}, we obtain the compactness of $[\B^{\alpha}, b]$ on the weighted Lebesgue spaces as follows.   
\begin{theorem}
Let $n=2$ and $\frac16<\alpha<\frac12$. If $b \in \CMO$, then $[\B^{\alpha}, b]$ is compact on $L^p(w^p)$ for all $p \in (1.2, 2)$ and for all $w^p \in A_{\frac{p}{1.2}} \cap RH_{(\frac2p)'}$. 
\end{theorem}

Next, we turn to bilinear Bochner-Riesz means of order $\alpha$, which is defined by 
\begin{align*}
\B^{\alpha}(f, g)(x) := \int_{\R^{2n}} (1-|\xi|^2-|\eta|^2)^{\alpha}_{+} \, 
\widehat{f}(\xi) \, \widehat{g}(\eta) e^{2\pi i x \cdot(\xi+\eta)} d\xi d\eta. 
\end{align*}  

\begin{theorem}
Let $n \ge 2$ and $b \in \CMO$. Then for each $k=1,2$, $[\B^{n-1/2}, b]_{e_k}$ is compact from $L^{p_1}(w_1^{p_1}) \times L^{p_2}(w_2^{p_2})$ to $L^p(w^p)$ for all $\vec{p}=(p_1, p_2)$ with $1<p_1,p_2<\infty$ and for all $\vec{w}=(w_1, w_2) \in A_{\vec{p}}$, where $\frac1p=\frac{1}{p_1}+\frac{1}{p_2}$ and $w=w_1 w_2$.  
\end{theorem}

\begin{proof}
Fix $k \in \{1,2\}$. Let us present a weighted estimates for $\B^{n-1/2}$. Indeed, it was shown in \cite{JSS} that 
\begin{align}\label{eq:BB-1}
\B^{n-1/2} \text{ is bounded from $L^{p_1}(w_1^{p_1}) \times L^{p_2}(w_2^{p_2})$ to $L^p(w^p)$}, 
\end{align}
for all $\vec{p}=(p_1, p_2)$ with $1<p_1,p_2<\infty$ and for all $\vec{w}=(w_1, w_2) \in A_{\vec{p}}$, where $\frac1p=\frac{1}{p_1}+\frac{1}{p_2}$ and $w= w_1 w_2$. Considering Theorem \ref{thm:Tb} and \eqref{eq:BB-1}, we are reduced to showing that 
\begin{align}\label{eq:BB-2}
[\B^{n-1/2}, b]_{e_k} \text{ is compact from $L^{p_1}(\Rn) \times L^{p_2}(\Rn)$ to $L^p(\Rn)$},  
\end{align}
for all $b \in \CMO$ and for all (or for some) $\frac1p=\frac{1}{p_1}+\frac{1}{p_2}$ with $1<p_1,p_2<\infty$. 

The rest of the proof is devoted to demonstrating \eqref{eq:BB-2}. Pick a nonnegative function $\phi \in \mathscr{C}_c^{\infty}(1/2, 2)$  satisfying $\sum_{j \in \Z} \phi(2^j t)=1$ for $t>0$. For each $j \ge 0$, we set 
\begin{align*}
\m_j^{\alpha}(\xi, \eta) := (1-\xi^2-\eta^2)^{\alpha}_{+} \phi(2^j(1-\xi^2-\eta^2)), 
\end{align*}
and define the bilinear operator
\begin{align}\label{eq:Tja}
T_j^{\alpha}(f, g)(x) := \int_{\R^{2n}} \m_j^{\alpha}(\xi, \eta) \, 
\widehat{f}(\xi) \, \widehat{g}(\eta) e^{2\pi i x \cdot(\xi+\eta)} \, d\xi d\eta. 
\end{align} 
It is obvious that 
\begin{align}\label{eq:Ba-Tj}
\B^{\alpha} = \sum_{j=0}^{\infty} T_j^{\alpha}. 
\end{align}
By \cite[eq. (3.1)]{LW}, one has 
\begin{align}\label{eq:BB-3}
\|T_j^{\alpha}\|_{L^{p_1}(\Rn) \times L^{p_2}(\Rn) \to L^p(\Rn)} \le 2^{-\delta j}, \quad\forall j \ge 0,  
\end{align}
for some $\delta>0$, whenever $\frac1p=\frac{1}{p_1}+\frac{1}{p_2}$ with $1\le p_1,p_2\le 2$ and $\alpha>n(\frac1p-1)$. On the other hand, from \eqref{eq:BB-1} and \cite[Theorem~2.22]{LMO}, one has 
\begin{align}\label{eq:BB-4}
\|[\B^{n-1/2}, b]_{e_k}\|_{L^{p_1}(\Rn) \times L^{p_2}(\Rn) \to L^p(\Rn)} \lesssim \|b\|_{\BMO}, 
\end{align}
for all $b \in \BMO$ and for all $\frac1p=\frac{1}{p_1}+\frac{1}{p_2}$ with $1<p_1,p_2<\infty$. By \eqref{eq:Ba-Tj}, \eqref{eq:BB-3}, \eqref{eq:BB-4} and Lemma \ref{lem:bTTj}, it suffices to prove that for each $j \ge 0$ and for any $b \in \CMO$, 
\begin{align}\label{eq:BB-5}
[T_j^{\alpha}, b]_{e_k} \text{ is compact from $L^{p_1}(\Rn) \times L^{p_2}(\Rn)$ to $L^p(\Rn)$}, 
\end{align}
for all $\alpha \in \R$ and for all $\frac1p=\frac{1}{p_1}+\frac{1}{p_2}$ with $1<p_1,p_2<\infty$. 

To proceed, we may assume that $b \in \mathscr{C}_c^{\infty}(\Rn)$ with $\supp(b) \subset B(0, R)$ for some $R>0$. We will only focus on the case $k=1$. Let $K_j^{\alpha}$ denote the kernel of $T_j^{\alpha}$. By \eqref{eq:Tja}, we have 
\begin{equation}\label{eq:KK-1}
K_j^{\alpha}(x, y_1, y_2) = \mathbf{K}_j^{\alpha}(x-y_1, x-y_2)
\end{equation}
and 
\begin{equation}\label{eq:KK-2} 
\mathbf{K}_j^{\alpha}(x, y) = \int_{\R^{2n}} \m_j^{\alpha}(\xi, \eta) e^{2\pi i (x \cdot \xi + y \cdot \eta)}\, d\xi d\eta. 
\end{equation}
The estimates for $\mathbf{K}_j^{\alpha}$ will be given in Lemma \ref{lem:kernel} below. By \eqref{eq:kernel-1} with $\rho>n$, one has 
\begin{align}\label{eq:AA}
|[T_j^{\alpha}, b]_{e_1}(f_1, f_2)(x)|
&=\bigg|\int_{\R^{2n}} (b(x)-b(y_1)) K_j^{\alpha}(x, y_1, y_2) f_1(y_1) f_2(y_2) dy_1 dy_2\bigg|
\\  \nonumber
&\lesssim 2^{-j\alpha} \|b\|_{L^{\infty}(\Rn)} \prod_{i=1}^2 \int_{\Rn} \frac{|f_i(y_i)| dy_i}{(1+2^{-j}|x-y_i|)^{\rho}} 
\\  \nonumber
&\lesssim 2^{j(2n-\alpha)} \|b\|_{L^{\infty}(\Rn)} Mf_1(x) Mf_2(x),  
\end{align}
Then using \eqref{eq:AA} and H\"{o}lder's inequality, we deduce that 
\begin{align*}
\|[T_j^{\alpha}, b]_{e_1}(f_1, f_2)\|_{L^p(\Rn)} 
\lesssim 2^{j(2n-\alpha)} \|b\|_{L^{\infty}(\Rn)} \|f_1\|_{L^{p_1}(\Rn)} \|f_2\|_{L^{p_2}(\Rn)}
\end{align*}
and hence, 
\begin{align}\label{eq:AA-1}
\sup_{\substack{\|f_1\|_{L^{p_1}(\Rn)} \le 1 \\ \|f_2\|_{L^{p_2}(\Rn)} \le 1}} 
\|[T_j^{\alpha}, b]_{e_1}(f_1, f_2)\|_{L^p(\Rn)} \le C 2^{j(2n-\alpha)} \|b\|_{L^{\infty}(\Rn)}. 
\end{align}
Let $A>\max\{2R, 1\}$. Then for any $|x|>A$, 
\begin{align*}
[T_j^{\alpha}, b]_{e_1}(f_1, f_2)(x) 
=-\int_{B(0,R) \times \Rn} b(y_1) K_j^{\alpha}(x, y_1, y_2) f_1(y_1) f_2(y_2) dy_1 dy_2. 
\end{align*}
This and \eqref{eq:kernel-1} with $\rho>n$ give 
\begin{align*}
|[T_j^{\alpha}, b]_{e_1}(f_1, f_2)(x)| 
& \lesssim \|b\|_{L^{\infty}(\Rn)}  \int_{B(0, R)} \frac{|f_1(y_1)| dy_1}{(1+2^{-j}|x-y_1|)^{\rho}} 
\int_{\Rn} \frac{|f_2(y_2)| dy_2}{(1+2^{-j}|x-y_2|)^{\rho}} 
\\
& \lesssim 2^{j(\rho+n)} \|b\|_{L^{\infty}(\Rn)}  \int_{B(0, R)} \frac{|f_1(y_1)|}{(1+|x|)^{\rho}} dy_1 \, Mf_2(x) 
\\
&\lesssim 2^{j(\rho+n)} \|b\|_{L^{\infty}(\Rn)} R^{n/p'_1} \|f_1\|_{L^{p_1}(\Rn)} \frac{Mf_2(x)}{(1+|x|)^{\rho}}. 
\end{align*}
Hence, we have 
\begin{align*}
\|[T_j^{\alpha}, b]_{e_1}(f_1, f_2) \mathbf{1}_{\{|x|>A\}}\|_{L^p(\Rn)} 
&\lesssim \|f_1\|_{L^{p_1}(\Rn)} \bigg(\int_{|x|>A}\frac{Mf_2(x)}{(1+|x|)^{\rho}} dx \bigg)^{\frac1p}
\\ 
&\lesssim \|f_1\|_{L^{p_1}(\Rn)}  \|Mf_2\|_{L^{p_2}(\Rn)} \bigg(\int_{|x|>A}\frac{dx}{(1+|x|)^{\rho p_1}}\bigg)^{\frac{1}{p_1}}
\\
&\lesssim A^{-(\rho p_1-n)/p_1} \|f_1\|_{L^{p_1}(\Rn)}  \|f_2\|_{L^{p_2}(\Rn)}, 
\end{align*}
which implies 
\begin{align}\label{eq:AA-2}
\lim_{A \to \infty} \sup_{\substack{\|f_1\|_{L^{p_1}(\Rn)} \le 1 \\ \|f_2\|_{L^{p_2}(\Rn)} \le 1}} 
\|[T_j^{\alpha}, b]_{e_1}(f_1, f_2) \mathbf{1}_{\{|x|>A\}}\|_{L^p(\Rn)} =0. 
\end{align}

For $\delta \in (0, 1)$ chosen later and $0<|h|<\frac{\delta}{2}$, we split 
\begin{align}\label{eq:Tjab}
[T_j^{\alpha}, b]_{e_1}(\vec{f})(x+h)-[T_j^{\alpha}, b]_{e_1}(\vec{f})(x) =I_1 + I_2 + I_3 + I_4,  
\end{align}
where 
\begin{align*}
I_1 &:=(b(x+h)-b(x)) \int_{\max\limits_{i=1,2} \{|x-y_i|\}>\delta} K_j^{\alpha}(x, \vec{y}) f_1(y_1) f_2(y_2) d\vec{y}, 
\\
I_2 &:=\int_{\max\limits_{i=1,2}\{|x-y_i|\}>\delta} (K_j^{\alpha}(x+h, \vec{y})-K_j^{\alpha}(x, \vec{y})) 
(b(x+h)-b(y_1)) f_1(y_1) f_2(y_2) d\vec{y}, 
\\
I_3 &:=\int_{\max\limits_{i=1,2} \{|x-y_i|\}\le \delta} K_j^{\alpha}(x,\vec{y})(b(y_{1})-b(x)) f_1(y_1) f_2(y_2) d\vec{y}, 
\\ 
I_4 &:=\int_{\max\limits_{i=1,2} \{|x-y_i|\} \le \delta} K_j^{\alpha}(x+h,\vec{y})(b(x+h)-b(y_1)) f_1(y_1) f_2(y_2) d\vec{y}. 
\end{align*}
In view of \eqref{eq:kernel-1} with $\rho>n$, we obtain 
\begin{align}\label{eq:Tj-I1}
|I_1| \lesssim |h| \|\nabla b\|_{L^{\infty}(\Rn)} \prod_{i=1}^2 \int_{\Rn} \frac{|f_i(y_i)|}{(1+2^{-j}|x-y_i|)^{\rho}} dy_i
\lesssim \delta\, Mf_1(x) Mf_2(x). 
\end{align}
Denote 
\begin{align*}
\mathcal{E}_1(x, \vec{y}) &:= |\K_j^{\alpha}(x+h-y_1, x+h-y_2) - \K_j^{\alpha}(x-y_1, x+h-y_2)|, 
\\ 
\mathcal{E}_2(x, \vec{y}) &:= |\K_j^{\alpha}(x-y_1, x+h-y_2) - \K_j^{\alpha}(x-y_1, x-y_2)|. 
\end{align*}
Since $|h|<\frac{\delta}{2}$, the estimates \eqref{eq:kernel-2} and \eqref{eq:kernel-3} give 
\begin{align}\label{eq:Tj-I2}
|I_2| &\lesssim \|b\|_{L^{\infty}(\Rn)} \int_{|x-y_1|>\delta} \mathcal{E}_1(x, \vec{y})  |f_1(y_1)| |f_2(y_2)| d\vec{y} 
\\  \nonumber
&\qquad + \|b\|_{L^{\infty}(\Rn)} \int_{\substack{|x-y_1| \le \delta \\ |x-y_2|>\delta}} \mathcal{E}_1(x, \vec{y})| |f_1(y_1)| |f_2(y_2)| d\vec{y} 
\\  \nonumber
&\qquad + \|b\|_{L^{\infty}(\Rn)} \int_{|x-y_2|>\delta} \mathcal{E}_2(x, \vec{y})| |f_1(y_1)| |f_2(y_2)| d\vec{y} 
\\  \nonumber
&\qquad + \|b\|_{L^{\infty}(\Rn)} \int_{\substack{|x-y_1|>\delta \\ |x-y_2|<\delta}} \mathcal{E}_2(x, \vec{y})| |f_1(y_1)| |f_2(y_2)| d\vec{y} 
\\  \nonumber
&\lesssim |h| \|b\|_{L^{\infty}(\Rn)} \int_{\R^{2n}} \frac{|f_1(y_1)| |f_2(y_2)|}{1+|x-y_1|^{2\rho}+|x+h-y_2|^{2\rho}} d\vec{y}
\\  \nonumber
&\qquad+|h| \|b\|_{L^{\infty}(\Rn)} \int_{\substack{|x-y_1| \le \delta \\ |x-y_2|>\delta}} \frac{|f_1(y_1)| |f_2(y_2)|}{1+|x+h-y_2|^{2\rho}} d\vec{y}
\\  \nonumber
&\qquad+|h| \|b\|_{L^{\infty}(\Rn)} \int_{|x-y_2|>\delta} \frac{|f_1(y_1)| |f_2(y_2)|}{1+|x-y_1|^{2\rho}+|x-y_2|^{2\rho}} d\vec{y}
\\  \nonumber
&\qquad+|h| \|b\|_{L^{\infty}(\Rn)} \int_{\substack{|x-y_1|>\delta \\ |x-y_2| \le \delta}} \frac{|f_1(y_1)| |f_2(y_2)|}{(1+|x-y_1|)^{2\rho}} d\vec{y}
\\  \nonumber
&\lesssim |h| \int_{\Rn} \frac{|f_1(y_1)|}{1+|x-y_1|^{\rho}} dy_1 \int_{\Rn} \frac{|f_2(y_2)|}{1+|x+h-y_2|^{\rho}} dy_2
\\  \nonumber
&\qquad+|h| \int_{|x-y_1| \le \delta} |f_1(y_1)| dy_1 \int_{\Rn} \frac{|f_2(y_2)|}{1+|x+h-y_2|^{\rho}} dy_2
\\  \nonumber
&\qquad+|h| \int_{\Rn} \frac{|f_1(y_1)|}{1+|x-y_1|^{\rho}} dy_1 \int_{\Rn} \frac{|f_2(y_2)|}{1+|x-y_2|^{\rho}} dy_2
\\  \nonumber
&\qquad+|h| \int_{\Rn} \frac{|f_1(y_1)|}{1+|x-y_1|^{\rho}} dy_1 \int_{|x-y_2| \le \delta} |f_2(y_2)| dy_2
\\  \nonumber
&\lesssim \delta\, Mf_1(x) Mf_2(x) + \delta\, Mf_1(x) Mf_2(x+h). 
\end{align}
Furthermore, using \eqref{eq:kernel-1}, we get 
\begin{align}\label{eq:Tj-I3}
|I_3| \lesssim \delta \|\nabla b\|_{L^{\infty}(\Rn)} \prod_{i=1}^2 \int_{|x-y_i| \le \delta} |f_i(y_i)| \, dy_i 
\lesssim \delta^{2n+1} \M(f_1, f_2)(x). 
\end{align}
Similarly, one has 
\begin{align}\label{eq:Tj-I4}
|I_4| \lesssim (\delta+|h|) \delta^{2n} \M(f_1, f_2)(x) \lesssim \delta^{2n+1} \M(f_1, f_2)(x). 
\end{align}
Collecting \eqref{eq:Tjab}--\eqref{eq:Tj-I4} and using H\"{o}lder inequality and the boundedness of $\M$, we derive 
\begin{align*}
\|\tau_h [T_j^{\alpha}, b]_{e_1}(\vec{f})-[T_j^{\alpha}, b]_{e_1}(\vec{f})\|_{L^p(\Rn)} 
\lesssim \delta \|f_1\|_{L^{p_1}(\Rn)} \|f_2\|_{L^{p_2}(\Rn)}. 
\end{align*}
From the estimate above, for any $\varepsilon>0$, taking $\delta>0$ such that $\delta<\min\{\varepsilon, 1\}$, we conclude that 
\begin{align}\label{eq:AA-3}
\lim_{|h| \to 0} \sup_{\substack{\|f_1\|_{L^{p_1}(\Rn)} \le 1 \\ \|f_2\|_{L^{p_2}(\Rn)} \le 1}} 
\|\tau_h [T_j^{\alpha}, b]_{e_1}(\vec{f})-[T_j^{\alpha}, b]_{e_1}(\vec{f})\|_{L^p(\Rn)} =0. 
\end{align}
As a consequence, \eqref{eq:BB-5} follows from Theorem \ref{thm:FK-1}, \eqref{eq:AA-1}, \eqref{eq:AA-2} and \eqref{eq:AA-3}.
\end{proof}

\begin{lemma}\label{lem:kernel}
Given $j \ge 0$ and $\alpha$, we define $\mathbf{K}_j^{\alpha}$ as in \eqref{eq:KK-2}. Then for any $\rho \in \N_+$, 
\begin{align}
\label{eq:kernel-1} |\K_j^{\alpha}(x, y)| &\lesssim \frac{2^{-j \alpha}}{(1+2^{-j}|x|)^{\rho}} \frac{2^{-j}}{(1+2^{-j}|y|)^{\rho}}, 
\\
\label{eq:kernel-2} |\K_j^{\alpha}(x+h, y) &- \K_j^{\alpha}(x, y)| \lesssim \frac{2^{-\alpha j} |h|}{1+|y|^{2\rho}}, \quad \forall h \in \Rn, 
\\ 
\label{eq:kernel-3} |\K_j^{\alpha}(x+h, y) - \K_j^{\alpha}(x, y)| 
&\lesssim \frac{2^{-j\alpha} |h|}{1+|x|^{2\rho}+|x|^{2\rho+1}+|y|^{2\rho}},\quad\forall |h| \le |x|/2, 
\\ 
\label{eq:kernel-4} |\K_j^{\alpha}(x+h, y+h) - \K_j^{\alpha}(x, y)| 
&\lesssim \frac{2^{-j\alpha} |h|}{1+|x|^{2\rho}+|y|^{2\rho}},\quad\forall |h| \le \min\{|x|, |y|\}/2. 
\end{align}
\end{lemma}

\begin{proof}
Set $\Delta_{\xi}:=\partial^2_{\xi_1}+\cdots+\partial^2_{\xi_n}$ and let $\Delta_{\xi}^k$ denote the $k$-th iteration of $\Delta_{\xi}$ for any $k \in \N$. Applying Leibniz's rule and the integration by parts, we obtain \eqref{eq:kernel-1} and 
\begin{equation}\label{eq:Dmj}
\|\Delta_{\xi}^k \m_j\|_{L^{\infty}} \le C_k 2^{2k j} 2^{-j\alpha}, \quad\forall k \in \N. 
\end{equation}
Note that for all $k, \ell \in \N$, 
\begin{align}\label{eq:DDmj}
\mathbf{K}_j^{\alpha}(x, y) = \frac{1}{(2\pi |x|)^{2k}} \frac{1}{(2\pi |y|)^{2\ell}} 
\int_{\R^{2n}} \Delta_{\xi}^k \Delta_{\eta}^{\ell} \m_j^{\alpha}(\xi, \eta) e^{2\pi i (x \cdot \xi + y \cdot \eta)}\, d\xi d\eta. 
\end{align}
Then using \eqref{eq:Dmj} and \eqref{eq:DDmj}, we get for all $h \in \Rn$, 
\begin{multline}\label{eq:KjKj-1}
|\K_j^{\alpha}(x+h, y)-\K_j^{\alpha}(x, y)| 
=\bigg|\int_{\R^{2n}} \m_j^{\alpha}(\xi, \eta) e^{2\pi i(x\cdot \xi+y \cdot \eta)} \big(e^{2\pi i h\cdot \xi} -1\big) d\xi d\eta\bigg| 
\\
\lesssim \|\m_j^{\alpha}\|_{L^{\infty}} |h| \bigg(1-\frac{1}{2^{j+1}}-\Big(1-\frac{2}{2^j}\Big) \bigg)
\lesssim 2^{-j(\alpha+1)} |h|, 
\end{multline}
and 
\begin{multline}\label{eq:KjKj-2}
|\K_j^{\alpha}(x+h, y)-\K_j^{\alpha}(x, y)| 
\simeq |y|^{-2k} \bigg|\int_{\R^{2n}} \Delta_{\eta}^k\m_j^{\alpha}(\xi, \eta) e^{2\pi i(x\cdot \xi+y \cdot \eta)} \big(e^{2\pi i h\cdot \xi} -1\big) d\xi d\eta\bigg| 
\\
\lesssim |y|^{-2k} \|\Delta^k_{\eta}\m_j^{\alpha}\|_{L^{\infty}} |h| 2^{-j}
\lesssim 2^{2kj} 2^{-j(\alpha+1)} |h| |y|^{-2k}. 
\end{multline}
Hence, \eqref{eq:KjKj-1} and \eqref{eq:KjKj-2} imply \eqref{eq:kernel-2}. To show \eqref{eq:kernel-3}, we apply \eqref{eq:DDmj} again to get 
\begin{align*}
\K_j^{\alpha}(x+h, y)-\K_j^{\alpha}(x, y) 
&=\frac{1}{(2\pi |x+h|)^{2k}} \int_{\R^{2n}} \Delta_{\xi}^k \m_j^{\alpha}(\xi, \eta) 
e^{2\pi i((x+h)\cdot \xi+y \cdot \eta)} d\xi d\eta 
\\
&\qquad-\frac{1}{(2\pi |x|)^{2k}} \int_{\R^{2n}} \Delta_{\xi}^k \m_j^{\alpha}(\xi, \eta) 
e^{2\pi i(x\cdot \xi+y \cdot \eta)} d\xi d\eta 
\\
&\simeq \bigg[\frac{1}{|x+h|^{2k}}-\frac{1}{|x|^{2k}}\bigg] \int_{\R^{2n}} \Delta_{\xi}^k \m_j^{\alpha}(\xi, \eta) 
e^{2\pi i((x+h)\cdot \xi+y \cdot \eta)} d\xi d\eta 
\\
&\qquad+ \frac{1}{|x|^{2k}} \int_{\R^{2n}} \Delta_{\xi}^k \m_j^{\alpha}(\xi, \eta) e^{2\pi i(x\cdot \xi+y \cdot \eta)}   \big(e^{2\pi i h\cdot \xi} -1\big) d\xi d\eta, 
\end{align*}
which together with \eqref{eq:DDmj} and $|h| \le |x|/2$ implies 
\begin{align}\label{eq:KjKj-3}
|\K_j^{\alpha}(x+h, y)-\K_j^{\alpha}(x, y)| 
& \lesssim \bigg[ 2^{-j} \bigg(\frac{1}{|x+h|^{2k}}-\frac{1}{|x|^{2k}}\bigg)  
+ \frac{|h|}{|x|^{2k}}\bigg] \|\Delta_{\xi}^k\m_j^{\alpha}\|_{L^{\infty}}
\\  \nonumber
&\lesssim 2^{2kj} 2^{-j(\alpha+1)} \bigg(\frac{|h|}{|x|^{2k+1}} + \frac{|h|}{|x|^{2k}}\bigg). 
\end{align}
Observe that for all $a_1,\ldots,a_n>0$, 
\begin{equation}\label{eq:aa}
\min_{1 \le j \le n} \frac{1}{a_j} \le \frac{n}{a_1+\cdots+a_n}. 
\end{equation}
Therefore, gathering \eqref{eq:KjKj-1}, \eqref{eq:KjKj-2}, \eqref{eq:KjKj-3} and \eqref{eq:aa}, we conclude that 
\begin{align*}
|\K_j^{\alpha}(x+h, y)-\K_j^{\alpha}(x, y)|  &\lesssim 2^{-j\alpha} |h| 
\min\bigg\{1,\, \frac{2^{2kj}}{|y|^{2k}},\, \frac{2^{2kj}}{|x|^{2k+1}} + \frac{2^{2kj}}{|x|^{2k}} \bigg\}
\\
&\lesssim \frac{2^{-j\alpha} |h| }{1+|x|^{2k+1}+|x|^{2k}+|y|^{2k}},  
\end{align*}
which agrees with \eqref{eq:kernel-3}. This in turn implies 
\begin{align*}
|&\K_j^{\alpha}(x+h, y+h)-\K_j^{\alpha}(x, y)|  
\\
&\qquad\le |\K_j^{\alpha}(x+h, y+h)-\K_j^{\alpha}(x, y+h)|  
+ |\K_j^{\alpha}(x, y+h)-\K_j^{\alpha}(x, y)|  
\\ 
&\qquad\lesssim \frac{2^{-j\alpha} |h|}{1+|x|^{2\rho+1}+|x|^{2\rho}+|y+h|^{2\rho}}  
+ \frac{2^{-j\alpha} |h| }{1+|x|^{2\rho}+|y|^{2\rho}+|y|^{2\rho+1}}  
\\
&\qquad\lesssim \frac{2^{-j\alpha} |h| }{1+|x|^{2\rho}+|y|^{2\rho}},   \quad\text{whenever } |h| \le \min\{|x|, |y|\}/2. 
\end{align*}
This proves \eqref{eq:kernel-4}.  
\end{proof}

\subsection{Riesz transforms related to Schr\"{o}dinger operators} 

Let $L=-\Delta+V$ be the Schr\"{o}dinger operator on $\Rn$ with $n \ge 3$. Here $V$ is a non-zero, non-negative potential, and belongs to $RH_q$ for some $q>n/2$. Denote 
\begin{align*}
\mathcal{R}_1 := VL^{-1},\quad \mathcal{R}_2 := V^{\frac12}L^{-\frac12}, \quad\text{and}\quad 
\mathcal{R}_3 := \nabla L^{-\frac12}. 
\end{align*}

By Theorem 5.6 and Remark 5.7 in \cite{BCDH}, one has that if $n/2 < q < n$, then $\mathcal{R}_i$ is bounded on $L^p(w^p)$ for all $p \in (1, p_i)$ and for all $w^p \in A_p \cap RH_{(p_i/p)'}$,  $i=1,2,3$, where $p_1=q$, $p_2=2q$ and $p_3=\frac{nq}{n-q}$. This together with \cite[Theorem~3.17]{BMMST} gives that if $b \in \BMO$, then for each $i=1,2,3$, 
\begin{align}\label{eq:Sch-1}
[\mathcal{R}_i, b] \text{ is bounded on $L^p(w^p)$}, \quad \forall p \in (1, p_i) \text{ and } \forall w^p \in A_p \cap RH_{(p_i/p)'}. 
\end{align}
On the other hand, it was shown in \cite{LP} that if if $n/2 < q < n$ and $b \in \CMO$, 
\begin{align}\label{eq:Sch-2}
[\mathcal{R}_i, b] \text{ is compact on $L^p(\Rn)$}, \quad \forall p \in (1, p_i),\quad i=1,2,3.  
\end{align}
As a consequence, from \eqref{eq:Sch-1}, \eqref{eq:Sch-2} and Theorem \ref{thm:lim}, we conclude the following. 

\begin{theorem}
Let $L=-\Delta+V$ be the Schr\"{o}dinger operator on $\Rn$ with $n \ge 3$. Assume that $V \in RH_q$ with $n/2 < q < n$.  If $b \in \CMO$, then $[\mathcal{R}_i, b]$, $i=1,2,3$, is compact on $L^p(w^p)$ for all $p \in (1, p_i)$ and for all $w^p \in A_p \cap RH_{(p_i/p)'}$, where $p_1=q$, $p_2=2q$ and $p_3=\frac{nq}{n-q}$. 
\end{theorem}


\end{document}